\documentclass[reqno]{amsproc}
\usepackage{amssymb}
\usepackage{euscript} %txfonts,pxfonts
\usepackage{mathrsfs} %\mathscr T
\usepackage{color,pifont}
\usepackage{tabularx}
\usepackage{makecell}
\usepackage{tikz}
\numberwithin{equation}{subsection}

\newtheorem{thm}{Theorem}[subsection]
\newtheorem{cor}[thm]{Corollary}
\newtheorem{lem}[thm]{Lemma}
\newtheorem{pro}[thm]{Proposition}

\newtheorem*{thm*}{Theorem}
\newtheorem{opq}[thm]{Problem}
\newtheorem*{opq*}{Problem}

\theoremstyle{remark}
\newtheorem{rem}[thm]{Remark}

\theoremstyle{definition}
\newtheorem{exa}[thm]{Example}
\newtheorem{dfn}[thm]{Definition}

\DeclareMathOperator{\E}{e}

\newcommand*{\ascr}{\mathscr{A}}
\newcommand*{\borel}[1]{{\mathfrak B}(#1)}
\newcommand*{\bscr}{\mathscr{B}}

\newcommand*{\cbb}{\mathbb C}

\newcommand*{\D}{\mathrm{d}}
\newcommand*{\dbb}{{\mathbb D}}

\newcommand*{\fscr}{\mathscr F}
\newcommand*{\Ge}{\geqslant}
\newcommand*{\gammab}{\boldsymbol \gamma}
\newcommand*{\gqb}{\mathcal{Q}}
\newcommand*{\gqbh}{\mathcal{Q}_{\hh_1,\hh_2}}
\newcommand*{\hh}{\mathcal H}

\newcommand*{\I}{{\mathrm i\hspace{.1ex}}}
\newcommand*{\idealo}{\mathscr I}
\newcommand*{\ideal}{\mathscr I_{T}}
\newcommand*{\is}[2]{\langle#1,#2\rangle}
\newcommand*{\jd}[1]{\mathscr N(#1)}

\newcommand*{\kk}{\mathcal K}
\newcommand*{\Le}{\leqslant}

\newcommand*{\mscr}{\mathscr M}
\newcommand*{\nbb}{\mathbb N}

\newcommand*{\ogr}[1]{\boldsymbol B(#1)}
\newcommand*{\ob}[1]{{\mathscr R}(#1)}
\newcommand*{\lrangle}[1]{\langle #1 \rangle}
\newcommand*{\rbb}{\mathbb R}

\newcommand*{\supp}[1]{\mathrm{supp}(#1)}

\newcommand*{\tbb}{\mathbb T}

\newcommand*{\zbb}{\mathbb Z}

\begin{document}
   \title[Conditional positive definiteness
in operator theory]{Conditional positive
definiteness in operator theory}
   \author[Z.\ J.\ Jab{\l}o\'nski]{Zenon Jan
Jab{\l}o\'nski}
   \address{Instytut Matematyki,
Uniwersytet Jagiello\'nski, ul.\ \L ojasiewicza 6,
PL-30348 Kra\-k\'ow, Poland}
\email{Zenon.Jablonski@im.uj.edu.pl}
   \author[I.\ B.\ Jung]{Il Bong Jung}
   \address{Department of Mathematics, Kyungpook National University,
Da\-egu 41566, Korea} \email{ibjung@knu.ac.kr}
   \author[J.\ Stochel]{Jan Stochel}
\address{Instytut Matematyki, Uniwersytet
Jagiello\'nski, ul.\ \L ojasiewicza 6, PL-30348
Kra\-k\'ow, Poland} \email{Jan.Stochel@im.uj.edu.pl}
   \thanks{The second author was supported by Basic Science
Research Program through the National Research
Foundation of Korea (NRF) funded by the Ministry
of Education (NRF-2021R111A1A01043569).}

   \maketitle
   \begin{abstract}
In this paper we extensively investigate the
class of conditionally positive definite
operators, namely operators generating
conditionally positive definite sequences. This
class itself contains subnormal operators, $2$-
and $3$-isometries, complete hypercontractions
of order $2$ and much more beyond them. Quite a
large part of the paper is devoted to the study
of conditionally positive definite sequences of
exponential growth with emphasis put on finding
criteria for their positive definiteness, where
both notions are understood in the semigroup
sense. As a consequence, we obtain semispectral
and dilation type representations for
conditionally positive definite operators. We
also show that the class of conditionally
positive definite operators is closed under the
operation of taking powers. On the basis of
Agler's hereditary functional calculus, we build
an $L^{\infty}(M)$-functional calculus for
operators of this class, where $M$ is an
associated semispectral measure. We provide a
variety of applications of this calculus to
inequalities involving polynomials and analytic
functions. In addition, we derive new necessary
and sufficient conditions for a conditionally
positive definite operator to be a subnormal
contraction (including a telescopic one).
   \end{abstract}
   \setcounter{tocdepth}{2}
\tableofcontents

   \section{Introduction}
   \subsection{Motivation}
The concepts of positive and conditional
positive definiteness (at least in the group
setting) have their origins in stochastic
processes that are stationary or which have
stationary increments
\cite{Kol41,Ma72,Ml83,Bi-Sa00,Sa13}. It seems
that conditional positive definiteness appeared
in operator theory for the first time on the
occasion of investigating subnormal operators
(see \cite{Sto}). Later it appeared sporadically
in the context of complete hyperexpansivity and
complete hypercontractivity of finite order,
both related to $m$-tuples of commuting
operators \cite{At2,Cha-Sh,Cha-Sh17}. The main
goal of the present paper is to exploit
conditional positive definiteness in the
semigroup setting to study a class of operators
which is large enough to subsume subnormal
operators \cite{Hal50,Con91} (which are
integrally tied to positive definiteness), $2$-
and $3$-isometries \cite{Ag-St1,Ag-St2,Ag-St3},
complete hypercontractions of order $2$
\cite{Cha-Sh}, certain algebraic operators which
are neither subnormal nor $m$-isometric, and
much more. Below we give a more detailed
discussion on this.

Throughout this paper $\hh$ stands for a complex
Hilbert space and $\ogr{\hh}$ for the
$C^*$-algebra of all bounded linear operators on
$\hh$. An operator $T\in \ogr{\hh}$ is said to
be {\em subnormal} if there exist a complex
Hilbert space $\kk$ and a normal operator $N\in
\ogr{\kk}$, called a {\em normal extension} of
$T$, such that $\hh \subseteq \kk$ (isometric
embedding) and $Th=Nh$ for all $h\in \hh$. A
sequence $\{\gamma_n\}_{n=0}^{\infty}$ of real
numbers is said to be {\em positive definite}
({\em PD} for brevity) if
   \begin{align} \label{virek}
\sum_{i,j=0}^k \gamma_{i+j} \lambda_i
\bar\lambda_j \Ge 0
   \end{align}
for all finite sequences of complex numbers
$\lambda_0, \ldots, \lambda_k$. The celebrated
Lambert's characterization of subnormality
\cite{lam} can be adapted to the context of not
necessarily injective operators as follows (for
(i)$\Leftrightarrow$(ii) see \cite[Theorem~
7]{St-Sz89}, while for
(ii)$\Leftrightarrow$(iii) apply
Theorem~\ref{Stiech} substituting $Th$ in place
of $h$).
   \begin{thm} \label{lamb}
If $T\in \ogr{\hh}$, then the following conditions are
equivalent{\em :}
   \begin{enumerate}
   \item[(i)] $T$ is subnormal,
   \item[(ii)] the sequence $\{\|T^n
h\|^2\}_{n=0}^{\infty}$ is a Stieltjes moment sequence
for every $h\in \hh$,
   \item[(iii)] the sequence $\{\|T^n
h\|^2\}_{n=0}^{\infty}$ is PD for every $h\in
\hh$.
   \end{enumerate}
   \end{thm}
   The above theorem, which fails for unbounded
operators (see \cite{J-J-S12-jfa,B-J-J-S17}),
turned out to be very useful when studying the
concrete classes of bounded operators (see
\cite{Lam88,J-J-S12,B-J-J-S15,St-St17,B-J-J-S18}).
Some of them are associated with the set $\fscr$
of nonconstant entire functions with nonnegative
Taylor's coefficients at $0$. The question of
characterizing the subnormality of composition
operators with matrix symbols on $L^2(\rbb^d,
\rho(x)dx)$ with a density function $\rho$
coming from $\varPhi \in \fscr$ (see
\cite{Sto90}) led to the following problem,
which for thirty years remains unsolved even for
second-degree monomials (see \cite[p.\
237]{Sto}).
   \begin{opq} \label{prob-haha}
Let $T\in \ogr{\hh}$ be a contraction and
$\varPhi\in \fscr$. Is it true that if
$\{\varPhi(\|T^n h\|^2)\}_{n=0}^{\infty}$ is a
PD sequence for every $h\in \hh,$ then $T$ is
subnormal\/$?$
   \end{opq}
The answer to the question in
Problem~\ref{prob-haha} is in the affirmative as
long as $\varPhi'(0)$, the derivative of
$\varPhi$ at $0$, is positive or $T$ is
algebraic (see \cite[Theorems~5.1 and
6.3]{Sto}). Problem~\ref{prob-haha} without the
assumption that $T$ is contractive has a
negative solution (see \cite[Example~5.4]{Sto}).
Note also that the converse implication in
Problem~\ref{prob-haha} is true even if $T$ is
not contractive (see the proof of
\cite[Theorem~5.1]{Sto}).

Before we continue the discussion, let us give a
necessary definition. A sequence
$\{\gamma_n\}_{n=0}^{\infty}$ of real numbers is
said to be {\em conditionally positive definite}
({\em CPD} for brevity) if inequality
\eqref{virek} holds for all finite sequences of
complex numbers $\lambda_0, \ldots, \lambda_k$
such that $\sum_{i=0}^k \lambda_i=0$. Continuing
our discussion, we note that if $T \in
\ogr{\hh}$ and $\varPhi= \exp$ (which is a
member of $\fscr$ with $\varPhi'(0) > 0$), then,
by the Schoenberg characterization of CPD
sequences (see Lemma~\ref{cpdpd2}), the sequence
$\{\exp(\|T^n h\|^2)\}_{n=0}^{\infty}$ is PD for
every $h\in \hh$ if and only if the sequence
$\{\|T^n h\|^2\}_{n=0}^{\infty}$ is CPD for
every $h\in \hh$. The situation becomes more
complex if the function $\exp$ is replaced by an
arbitrary member $\varPhi$ of $\fscr$; then the
hypothesis that the sequence $\{\varPhi(\|T^n
h\|^2)\}_{n=0}^{\infty}$ is PD for every $h\in
\hh$ implies that for some positive integer $j$
(depending only on $\varPhi$) and for every
$h\in \hh$, the sequence $\{\|T^n
h\|^{2j}\}_{n=0}^{\infty}$ is CPD (see
\cite[Lemma~5.2]{Sto}). It was shown in
\cite[Theorem~4.1]{Sto} that if $T$ is a
contraction, then $T$ is subnormal if and only
if the sequence $\{\|T^n h\|^2\}_{n=0}^{\infty}$
is CPD for every $h\in \hh$. The contractivity
hypothesis cannot be removed (see
\cite[Example~5.4]{Sto}).

The above-mentioned results of \cite{Sto} were
obtained by using {\em ad hoc} methods. The main
goal of the present paper is to systematically
and rigorously study operators $T\in\ogr{\hh}$
having the property that for every $h\in \hh$,
the sequence $\{\|T^n h\|^2\}_{n=0}^{\infty}$ is
CPD. Such operators are called here {\em
conditionally positive definite} ({\em CPD} for
brevity). In view of Theorem~ \ref{lamb} and the
fact that PD sequences are CPD, subnormal
operators are CPD but not conversely (see
\cite[Example~ 5.4]{Sto}). Our investigations
are preceded by developing harmonic analysis of
CPD functions of (at most) exponential growth on
the additive semigroup of nonnegative integers.
As a consequence, we gain, among other things, a
deeper insight into the subtle relationship
between subnormality and conditional positive
definiteness.
   \subsection{Intuition} To develop an intuition about CPD
operators, we answer a few natural simple
questions that are usually asked when
considering new classes of operators.
   \begin{enumerate}
   \item \label{circ1} Is the class of CPD operators on $\hh$
closed\footnote{Using the fact that the multiplication in
$\ogr{\hh}$ is sequentially continuous in {\sc SOT} (the
strong operator topology), one can show that the class of CPD
operators is sequentially SOT-closed. The question of whether
it is SOT-closed remains open. This is related to the
celebrated Bishop theorem stating that the class of subnormal
operators is the SOT-closure of the set of normal operators
(see \cite{Bis57}; see also \cite[Theorem~II.1.17]{Con91}).}
in the operator norm?
   \item \label{circ2}
Is the orthogonal sum of CPD operators CPD?
   \item \label{circ3} Is the restriction of a CPD
operator to its invariant subspace CPD?
   \item \label{circ6} Are positive integer powers of CPD operators still CPD?
   \item \label{circ7} Is the inverse of an invertible CPD operator CPD?
   \item \label{circ4} Is the tensor product of two CPD
operators CPD?
   \item \label{circ9}
   Is it true that if $T$ is CPD, then $T +
\lambda I$ is CPD for any complex number
$\lambda$, where $I$ stands for the identity
operator?
   \item \label{circ10}
Is it true that if $T$ is CPD, then $\lambda T$
is CPD for any complex number $\lambda$?
   \end{enumerate}
The answers to questions
(\ref{circ1})-(\ref{circ7}) are in the
affirmative. The rest of the questions are
answered in the negative. The affirmative
answers to questions (\ref{circ1}) and
(\ref{circ2}) follow from the fact that the
class of CPD sequences is closed in the topology
of pointwise convergence. In turn, the
affirmative answer to question (\ref{circ3}) is
a direct consequence of the definition. The
affirmative answer to question (\ref{circ6}) is
given in Theorem~\ref{pow-dwa}. The negative
answer to question (\ref{circ9}) implies that
the algebraic sum of two commuting CPD operators
may not be CPD. The negative answer to question
(\ref{circ10}) implies that the product of two
commuting CPD operators may not be CPD.

Now, we show that the answer to question
(\ref{circ7}) is in the affirmative.
   \begin{pro}
If $T\in \ogr{\hh}$ is a CPD operator which is
invertible in $\ogr{\hh}$, then its inverse
$T^{-1}$ is CPD.
   \end{pro}
   \begin{proof}
Let $\lambda_0, \ldots, \lambda_k$ be a finite
sequence of complex numbers such that
$\sum_{j=0}^k \lambda_j=0$ and let $h\in \hh$.
Since $T$ is surjective, there exists $f\in \hh$
such that $h=T^{2k}f$. Using the assumption that
$T$ is CPD, we conclude that
   \begin{align*}
\sum_{i,j=0}^k \|(T^{-1})^{i+j} h\|^2 \lambda_i
\bar \lambda_j = \sum_{i,j=0}^k
\|(T^{(k-i)+(k-j)} f\|^2 \lambda_i \bar
\lambda_j \Ge 0,
   \end{align*}
which completes the proof.
   \end{proof}
The negative answer to question (\ref{circ4}) is
justified in Example~\ref{tenzinpr} below, which
is essentially based on the concept of a strict
$m$-isometry. Following \cite{Ag90}, we call an
operator $T\in \ogr{\hh}$ an $m$-isometry, where
$m$ is a positive integer, if
   \begin{align*}
\sum_{k=0}^m (-1)^k {m \choose k}{T^*}^kT^k = 0.
   \end{align*}
An $m$-isometry $T$ is strict if $m=1$ and
$\hh\neq \{0\}$, or $m\Ge 2$ and $T$ is not an
$(m-1)$-isometry. It is worth noting that any
non-isometric $3$-isometry is CPD, but not
subnormal. This fact can be deduced from
Proposition~\ref{sub-mzero-n} and
\cite[Proposition~4.5]{Sh-At}. Since any
$m$-isometry is $(m+1)$-isometry (see \cite[p.\
389]{Ag-St1}), we conclude that any strict
$2$-isometry is CPD, but not subnormal.
   \begin{exa} \label{tenzinpr}
In this example, we use the following fact,
which can be deduced from
\cite[Corollary~3.5]{J-J-S20}.
   \begin{align*}
   \begin{minipage}{70ex}
{\em The tensor product of a strict
$m_1$-isometry and a strict $m_2$-isometry is a
strict $(m_1+m_2-1)$-isometry.}
   \end{minipage}
   \end{align*}
Hence, the tensor product of a strict
$2$-isometry and a strict $3$-isometry is a
strict $4$-isometry, which, according to
Proposition~\ref{sub-mzero-n}, is not CPD.
However, by the same proposition, $2$- and
$3$-isometries are CPD. This gives the negative
answer to question (\ref{circ4}). A similar
conclusion can be drawn considering the tensor
product of two strict $3$-isometries (the
resulting tensor product is a strict
$5$-isometry). We refer the reader to
\cite[Proposition~ 8]{At91} for examples of
strict $2$- and $3$-isometries, which are
unilateral weighted shifts.
   \hfill $\diamondsuit$
   \end{exa}
That the answers to questions (\ref{circ9}) and
(\ref{circ10}) are in the negative is shown in
the following example (see also
Remark~\ref{asty}).
   \begin{exa} \label{zplusnil}
Let $N\in \ogr{\hh}$ be a nonzero operator such
that $N^2=0$. Fix $\theta\in [0,2\pi)$ and set
$T_{\theta}=N - \E^{\I \theta} I$. Note that
   \begin{align} \label{tuplu}
T_{\theta} + \lambda I = N + (\lambda - \E^{\I
\theta})I \quad \text{for any complex number
$\lambda$.}
   \end{align}
It follows from \cite[Theorem~2.2]{Ber-Mar-No13}
that $T_{\theta}$ is a $3$-isometry, so by
Proposition~\ref{sub-mzero-n}, $T_{\theta}$ is
CPD. Denote by $\varXi_{T_{\theta}}$ the set of
all complex numbers $\lambda$ for which the
operator $T_{\theta} + \lambda I$ is CPD. Since
the class of CPD operators is closed in the
operator norm, we see that $\varXi_{T_{\theta}}$
is a closed subset of the complex plane. If
$\lambda$ is a complex number such that
$|\lambda - \E^{\I \theta}| < 1$, then by
\eqref{tuplu} and Corollary~\ref{quasi-1-nil},
$\lambda \notin \varXi_{T_{\theta}}$. This gives
the negative answer to question (\ref{circ9}).
Observe also that if $|\lambda - \E^{\I \theta}|
= 1$, then by \eqref{tuplu},
\cite[Theorem~2.2]{Ber-Mar-No13} and
Proposition~\ref{sub-mzero-n}, $\lambda \in
\varXi_{T_{\theta}}$.

Denote by $\widetilde \varXi_{T_{\theta}}$ the
set of all complex numbers $\lambda$ for which
the operator $\lambda T_{\theta}$ is CPD. As
above, we verify that $\widetilde
\varXi_{T_{\theta}}$ is a closed subset of the
complex plane. Since $\lambda T_{\theta}=
(\lambda N) - \lambda \E^{\I\theta} I$, we infer
from Corollary~\ref{quasi-1-nil} that $\lambda
\notin \widetilde \varXi_{T_{\theta}}$ whenever
$0 < |\lambda| < 1$. This answers question
(\ref{circ10}) in the negative. Let us also
notice that if $|\lambda| \in \{0,1\}$, then by
the definition of a CPD operator, $\lambda \in
\widetilde \varXi_{T_{\theta}}$.
   \hfill $\diamondsuit$
   \end{exa}
   \begin{rem} \label{asty}
Regarding question (\ref{circ10}), note that by
Corollary~\ref{cpdalp-c}, for every
non-subnor\-mal CPD operator $T$, $r(T) \Ge 1$
and $\lambda T$ is not CPD for any complex
number $\lambda$ such that $0 < |\lambda| <
\frac{1}{r(T)}$, where $r(T)$ stands for the
spectral radius of $T$. Note that in
Example~\ref{zplusnil}, $r(T_{\theta})=1$.

Using the description of CPD algebraic operators
as in \cite{J-J-S21p}, one can show that
   \allowdisplaybreaks
   \begin{align*}
\varXi_{T_{\theta}} & = \{\lambda\colon \lambda
\text{ is a complex number and } |\lambda -
\E^{\I \theta}| = 1\},
   \\
\widetilde \varXi_{T_{\theta}} & =
\{\lambda\colon \lambda
 \text{ is a complex number and } |\lambda| =
1\} \cup \{0\},
   \end{align*}
where the sets $\varXi_{T_{\theta}}$ and
$\widetilde \varXi_{T_{\theta}}$ are as in
Example~\ref{zplusnil}.
   \hfill $\diamondsuit$
   \end{rem}
According to Proposition~\ref{sub-mzero-n}, the
only $m$-isometries that are CPD are
$3$-isometries. Since subnormal operators are
also CPD, this raises another natural question.
   \begin{enumerate}
   \item[(9)]
Are there CPD operators that are not orthogonal sums of a
subnormal operator and a $3$-isometry?
   \end{enumerate}
As shown in Example~\ref{niesuim}, the answer to
question (9) is in the affirmative.
   \begin{exa} \label{niesuim}
Let $a\in (1,\infty)$ and let $W_a$ be the
unilateral weighted shift on $\ell^2$ as in
Example~\ref{prz-do-na}. Then
   \begin{align} \label{wittsh}
   \begin{minipage}{70ex}
$W_{a}$ is CPD, but neither subnormal nor
$3$-isometric.
   \end{minipage}
   \end{align}
Suppose to the contrary that $W_{a}$ is an
orthogonal sum of a number of subnormal
operators and a number of $3$-isometries. This
implies that there exists a nonzero closed
subspace of $\ell^2$ which reduces $W_{a}$
either to a subnormal operator or to a
$3$-isometry. Since (injective) unilateral
weighted shifts are irreducible (see
\cite[(3.0)]{Ml88}), $W_{a}$ is either subnormal
or $3$-isometric, which contradicts
\eqref{wittsh}.
   \hfill $\diamondsuit$
   \end{exa}
As is well known, the class of unilateral
weighted shifts is an important research area,
providing useful tools for constructing examples
and counterexamples (see \cite{Shi74}). This is
also true for our paper, as seen in
Examples~\ref{tenzinpr}, \ref{niesuim} and
\ref{prz-do-na} and Remark~\ref{manyrem} (see
also Propositions~\ref{cpd-exo} and
\ref{chyp-exo}), where we make extensive use of
weighted shifts. In particular, CPD operators
that are neither subnormal nor $3$-isometric can
be implemented as unilateral weighted shifts
(see \eqref{wittsh}). A natural question then is
to characterize unilateral weighted shifts that
are CPD. An in-depth study of CPD unilateral
weighted shifts based on a L\'{e}vy-Khinchin
type formula (cf.\ \cite[Theorem~4.3.19]{B-C-R})
is carried out in the forthcoming paper
\cite{J-J-L-S21}. It gives explicit methods to
construct weighted shifts of this class and
solves the flatness and the $n$-step backward
extension problems in this class.
   \subsection{\label{Subs.1.3}Ideas and concise description}
In this paper we provide several
characterizations of CPD operators. For the
reader's convenience, we make an excerpt from
characterizations that are contained in
Theorems~\ref{cpdops}, \ref{dyltyprep}
and~\ref{dyl-an} (see also
Theorem~\ref{boundiff}).
   \begin{thm} \label{cpdoppry}
Let $T\in \ogr{\hh}$. Then the following
conditions are equivalent{\em :}
   \begin{enumerate}
   \item[(i)] $T$ is CPD,
   \item[(ii)] there exist $B,C\in \ogr{\hh}$ and a
$\ogr{\hh}$-valued Borel semispectral measure
$F$ on $[0,\infty)$ with compact support such
that $B=B^*$, $C\Ge 0$, $F(\{1\})=0$~and
   \begin{align} \label{titun}
T^{*n}T^n = I + n B + n^2 C + \int_{[0,\infty)}
Q_n(x) F(\D x), \quad n = 0,1,2, \ldots,
   \end{align}
where $Q_n$ is the polynomial as in {\em
\eqref{klaud}},
   \item[(iii)] there exists a $\ogr{\hh}$-valued Borel
semispectral measure $M$ on $[0,\infty)$ with
compact support such that
   \begin{align} \label{titir}
T^{*n}(I-2T^*T+T^{*2}T^{2})T^n =
\int_{[0,\infty)} x^n M(\D x), \quad n = 0,1,2,
\ldots.
   \end{align}
   \end{enumerate}
Moreover, the triplet $(B,C,F)$ in {\em (ii)}
and the measure $M$ in {\em (iii)} are unique,
and
   \begin{align} \label{Ibt1}
B & = T^*T - I - \frac{1}{2} M(\{1\}),
   \\ \label{Ibt2}
C & =\frac{1}{2} M(\{1\}),
   \\ \label{Ibt3}
F(\varDelta) &
=(1-\chi_{\varDelta}(1))M(\varDelta), \quad
\text{$\varDelta$ is a Borel subset of
$[0,\infty)$}.
   \end{align}
   \end{thm}
Denote by $\mathit{CPD}_{\hh}$ the class of all
CPD operators on $\hh$ and by $\mathscr T_{\hh}$
the class of all triplets $(B,C,F)$, where
$B,C\in \ogr{\hh}$ are such that $B=B^*$ and
$C\Ge 0$, and $F$ is a $\ogr{\hh}$-valued Borel
semispectral measure on $[0,\infty)$ with
compact support such that $F(\{1\})=0$. In view
of Theorem~\ref{cpdoppry}, the mapping
   \begin{align} \label{psur}
\varPsi_{\hh} \colon \mathit{CPD}_{\hh}
\longrightarrow {\mathscr T}_{\hh} \text{ given
by } \varPsi_{\hh}(T) = (B,C,F),
   \end{align}
where $(B,C,F)\in \mathscr T_{\hh}$ satisfies
\eqref{titun}, is well defined. The mapping
$\varPsi_{\hh}$ is never injective (provided
$\hh\neq \{0\}$). Indeed, an operator $T\in
\ogr{\hh}$ is an isometry if and only if
\eqref{titun} holds with $(B,C,F)=(0,0,0)$, so
   \begin{align} \label{sosua}
\text{$\varPsi_{\hh}^{-1}(\{(0,0,0)\})$ is the
class of all isometries on $\hh$.}
   \end{align}
More generally, if $T\in \ogr{\hh}$ is a CPD
operator and $V\in \ogr{\hh}$ is any isometry
that commutes with $T$, then the operator $VT$
is CPD and $\varPsi_{\hh}(T)=\varPsi_{\hh}(VT)$.
Moreover, the mapping $\varPsi_{\hh}$ is not
surjective in general. To have an example,
consider the triplet $(B,C,0) \in \mathscr
T_{\hh}$. Then it can happen that
$\varPsi_{\hh}^{-1}(\{(B,C,0)\}) = \emptyset$
(e.g., when $\hh=\cbb$ and $(B,C)\neq (0,0)$),
which means that in this case the expression on
the right-hand side of the equality in
\eqref{titun} is a polynomial in $n$ with
operator coefficients, but there is no $T\in
\ogr{\hh}$ satisfying \eqref{titun}. In other
words, in view of \cite[Corollary~3.5]{J-J-S20},
the lack of surjectivity appears not only in the
class of CPD operators but also in the class of
$m$-isometries. What is more, this drawback also
applies to other classes of operators, not to
mention subnormal ones (cf.\ \eqref{tobemom}).

As for the second characterization of CPD
operators given in Theorem~\ref{cpdoppry}(iii),
it reduces the number of parameters to one, but
at the cost of increasing the complexity of the
expression on the left-hand side of
\eqref{titir}. Denoting by $\widetilde{\mathscr
T}_{\hh}$ the class of all $\ogr{\hh}$-valued
Borel semispectral measures on $[0,\infty)$ with
compact support, observe that according to
Theorem~\ref{cpdoppry}, the mapping
   \begin{align} \label{vsur}
\widetilde \varPsi_{\hh} \colon
\mathit{CPD}_{\hh} \to \widetilde{\mathscr
T}_{\hh} \text{ given by } \widetilde
\varPsi_{\hh}(T) = M,
   \end{align}
where $M \in \widetilde{\mathscr{T}}_{\hh}$
satisfies \eqref{titir}, is well defined. Since
   \begin{align} \label{sosub}
\text{$\widetilde \varPsi_{\hh}^{-1}(\{0\})$ is
the class of all $2$-isometries,}
   \end{align}
the mapping $\widetilde \varPsi_{\hh}$ is never
injective (provided $\hh\neq \{0\}$). Moreover,
it is not surjective in general (e.g., if
$\hh=\cbb$ and $c\in (0,\infty)$, then
$\widetilde \varPsi_{\hh}^{-1}(\{c \delta_1
I\})=\emptyset$). To compare the surjectivity of
$\varPsi_{\hh}$ and $\widetilde \varPsi_{\hh}$,
note that if $\varPsi_{\hh}^{-1}(\{(B,C,F)\})
\neq \emptyset$, then by Theorem~\ref{cpdoppry}
any $T\in \varPsi_{\hh}^{-1}(\{(B,C,F)\}$ is CPD
and $T\in \widetilde \varPsi_{\hh}^{-1}(\{M\})$,
where $M$ is defined by \eqref{Ibt2} and
\eqref{Ibt3}. And {\em vice versa}, if
$\widetilde \varPsi_{\hh}^{-1}(\{M\}) \neq
\emptyset$, then any $T\in \widetilde
\varPsi_{\hh}^{-1}(\{M\})$ is CPD and $T\in
\varPsi_{\hh}^{-1}(\{(B,C,F)\})$, where $B$, $C$
and $F$ are defined by
\eqref{Ibt1}-\eqref{Ibt3}. Note also that if
$(B,C,F)=(0,0,0)$, then $M=0$ corresponds to
$(B,C,F)$ via \eqref{Ibt2} and \eqref{Ibt3} and
by \eqref{sosua} and \eqref{sosub}, we have
   \begin{align*}
\varPsi_{\hh}^{-1}(\{(0,0,0)\}) \varsubsetneq
\widetilde \varPsi_{\hh}^{-1}(\{0\}) \quad
\text{(provided $\dim \hh\Ge \aleph_0$).}
   \end{align*}
The reader must be aware of the fact that there
are operators $T\in \widetilde
\varPsi_{\hh}^{-1}(\{0\})$ such that $T\in
\varPsi_{\hh}^{-1}(\{(B,0,0)\})$ with
$B=T^*T-I\neq 0$, where $(B,0,0)$ corresponds to
$M=0$ via \eqref{Ibt1}-\eqref{Ibt3}. These are
exactly non-isometric 2-isometries. It is worth
mentioning that the ranges of the mappings
$\varPsi_{\hh}$ and $\widetilde \varPsi_{\hh}$
when restricted to operators of class $\gqb$ can
be described explicitly (see
Remark~\ref{imprym}). However, the problem of
describing the ranges of the mappings
$\varPsi_{\hh}$ and $\widetilde\varPsi_{\hh}$ in
full generality is highly non-trivial.

\bigskip

The organization of this paper is as follows. We
begin by introducing notation and terminology in
Subsection~\ref{Sec1.2} and collecting more or
less known facts about PD and CPD (scalar)
sequences in Subsection~\ref{Subs2.1}. The
remainder of Section~\ref{Sec2} is devoted to
systematic study of CPD sequences. In
Subsection~\ref{Sec2.2} we provide an integral
representation for a CPD sequence of exponential
growth and relate the rate of its growth to the
``size'' of the closed support of its
representing measure (see
Theorem~\ref{cpd-expon}). We also compare the
integral representations for PD and CPD
sequences (see Theorem~\ref{dyszcz3}).
Theorem~\ref{Gyeon}, which is the main result of
this subsection, states that a sequence
$\{\gamma_n\}_{n=0}^{\infty}$ of exponential
growth is PD if and only if $0$ is an
accumulation point of the set of all $\theta \in
(0,\infty)$ for which the sequence $\{\theta^n
\gamma_n\}_{n=0}^{\infty}$ is CPD. In
Subsection~\ref{Sec2.3} we characterize CPD
sequences of exponential growth for which the
sequence of consecutive differences is either
convergent or bounded from above plus some
additional constraints (see
Theorems~\ref{boundiff-scalar} and
\ref{boundiff-scalar2}). As a consequence, we
show that, subject to some mild constraints,
convergent CPD sequences of exponential growth
are PD (see Corollary~\ref{pd2cpd}).

Starting from Section~\ref{Sec3}, we begin the
study of CPD operators. In
Subsection~\ref{Sec3.1} we give a semispectral
integral representation for such operators and
relate their spectral radii to the closed
supports of representing semispectral measures
(see Theorem~\ref{cpdops}). Certain semispectral
integral representations for completely
hypercontractive and completely hyperexpansive
operators of finite order appeared in
\cite{Ja02,Cha-Sh,Cha-Sh17} with the
representing semispectral measures concentrated
on the closed interval $[0,1]$. In our case
there is no limitation on the size of the
support (the reader should be aware of the fact
that CPD operators are not scalable in general,
see Corollary~\ref{scalcpd}).
Theorem~\ref{boundiff} offers yet another
semispectral integral representation for CPD
operators satisfying a telescopic-like
condition. Theorem~\ref{dyltyprep}, which is the
main result of Subsection~\ref{Sec3.2}, provides
a dilation representation for CPD operators
based on Agler's hereditary functional calculus
and relates their spectral radii to the norms of
positive operators appearing in the dilation
representation. Subsection~\ref{Sec3.3-n}
contains simplified semispectral and dilation
representations of CPD operators (see
Theorem~\ref{dyl-an}). As an application, we
show that the class of CPD operators is closed
under the operation of taking powers (see
Theorem~\ref{pow-dwa}). We also completely
characterize CPD operators of class $\gqb$ (see
Theorem~\ref{cpd-q}). In both cases, we describe
explicitly the corresponding semispectral
integral and dilation representations. In
Theorem~\ref{subn-1} we give necessary and
sufficient conditions for a CPD operator $T$ to
be subnormal written in terms of the
semispectral integral representation of $T$.
Theorem~\ref{glow-main}, which is the main
result of Subsection~\ref{Sec4.1}, provides
several characterizations of subnormal
contractions via conditional positive
definiteness including the one appealing to the
telescopic condition. This is a generalization
of \cite[Theorem~4.1]{Sto}. On the basis of
earlier results, we characterize conditional
positive definiteness of a (bounded) operator
$T$ on $\hh$ by subnormality of (in general
unbounded) unilateral weighted shifts $W_{T,h}$,
$h\in \hh$, canonically associated with $T$ (see
Proposition~\ref{cpd-exo}). If the sequence
$\{T^{*(n+1)}T^{n+1} -
T^{*n}T^n\}_{n=0}^{\infty}$ is not convergent in
the weak operator topology, then some (or even
all except $h=0$) weighted shifts $W_{T,h}$ may
be unbounded. If the limit exists and is
nonzero, then all $W_{T,h}$ are bounded, but $T$
is not subnormal. Finally, if the limit exists
and is equal to zero, then $T$ is a subnormal
contraction.

In Subsection~\ref{Sec4.1-n}, we construct an
$L^{\infty}(M)$-functional calculus for CPD
operators, where $M$ is a semispectral measure
on the closed half-line $[0,\infty)$ associated
to $T$ (see Theorem~\ref{dyl-an2}). As a
consequence, we obtain a variety of estimates on
norms of polynomial and analytic expressions
coming from operators in question (see
Corollary~\ref{wni-pol} and
Subsection~\ref{appl-to}). The last subsection
of this paper is devoted to characterizing CPD
operators for which the closed support of the
associated semispectral measure is one of the
three sets $\emptyset$, $\{1\}$ and $\{0\}$. It
is shown that the first two cases completely
characterize CPD $m$-isometries (see
Proposition~\ref{sub-mzero-n}). The third case
leads to CPD operators that are beyond the
classes of subnormal and $m$-isometric operators
(see Proposition~\ref{kop-2izo} and
Example~\ref{prz-do-na}).
   \subsection{\label{Sec1.2}Notation and terminology}
   We denote by $\rbb$ and $\cbb$ the fields of real
and complex numbers, respectively. Since we consider
suprema of subsets of $\rbb$ which may be empty, we
adhere to the often-used convention that
   \begin{align} \label{konw-1}
\sup{\emptyset} = \sup_{x\in \emptyset} f(x):= -\infty
\quad \text{whenever $f\colon \rbb \to \rbb$.}
   \end{align}
We write $\nbb$, $\zbb_+$ and $\rbb_+$ for the sets of
positive integers, nonnegative integers and
nonnegative real numbers, respectively. As usual,
$\cbb[X]$ stands for the ring of all polynomials in
indeterminate $X$ with complex coefficients. We
customarily identify members of $\cbb[X]$ with
polynomial functions of one real variable. The unique
involution on $\cbb[X]$ which sends $X$ to itself is
denoted by ${}^{*}$, that is, if $p=\sum_{i\Ge 0}
\alpha_i X^i \in \cbb[X]$, then $p^*=\sum_{i\Ge 0}
\overline{\alpha_i} X^i$, or in the language of
polynomial functions $p^*(x)=\overline{p(x)}$ for all
$x \in \rbb$. If no ambiguity arises, the
characteristic function of a subset $\varOmega_1$ of a
set $\varOmega$ is denoted by $\chi_{\varOmega_1}$.
Given a compact topological Hausdorff space
$\varOmega,$ let $C(\varOmega)$ stand for the Banach
space of all continuous complex functions on
$\varOmega$ with the supremum~norm
   \begin{align*}
\|f\|_{C(\varOmega)}=\sup_{x\in \varOmega} |f(x)|,
\quad f \in C(\varOmega).
   \end{align*}
We write $\borel{\varOmega}$ for the $\sigma$-algebra
of all Borel subsets of a topological Hausdorff space
$\varOmega.$ If not stated otherwise, measures
considered in this paper are assumed to be positive.
The {\em closed support} of a finite Borel measure
$\mu$ on $\rbb$ (or $\cbb$) is denoted by $\supp{\mu}$
(recall that $\supp{\mu}$ exists because $\mu$ is
automatically regular, see
\cite[Theorem~2.18]{Rud87}). Given $x\in \rbb$, we
write $\delta_x$ for the Borel probability measure on
$\rbb$ such that $\supp{\delta_x} = \{x\}$.

All Hilbert spaces considered in this paper are
assumed to be complex. Given Hilbert spaces
$\hh$ and $\kk$, we denote by $\ogr{\hh,\kk}$
the Banach space of all bounded linear operators
from $\hh$ to $\kk$. We abbreviate
$\ogr{\hh,\hh}$ to $\ogr{\hh}$ and denote by
$\ogr{\hh}_+$ the convex cone $\{T \in
\ogr{\hh}\colon T \Ge 0\}$ of nonnegative
operators on $\hh$. We write $I_{\hh}$ (or
simply $I$ if no ambiguity arises) for the
identity operator on $\hh$. Let $T \in
\ogr{\hh}$. In what follows, $\jd{T}$, $\ob{T}$,
$\sigma(T)$, $\sigma_{\mathrm{p}}(T)$, $r(T)$
and $|T|$ stand for the kernel, the range, the
spectrum, the point spectrum, the spectral
radius and the modulus of $T$, respectively. To
comply with Gelfand's formula for spectral
radius, we adhere to the convention that $r(T) =
0$ if $\hh= \{0\}.$ We say that $T$ is {\em
normaloid} if $r(T)=\|T\|$, or equivalently, by
Gelfand's formula for spectral radius, if and
only if $\|T^n\|=\|T\|^n$ for all $n\in\nbb$.
Let us recall the following basic fact (see
\cite[Proposition~II.4.6]{Con91}, see also
\cite[p.\ 116]{Fur}).
   \begin{align} \label{subn-norm}
\text{\em Any subnormal operator is normaloid.}
   \end{align}
This will be used several times in this
article. Given an operator $T\in
\ogr{\hh}$, we set
   \begin{align} \label{bmt}
\bscr_m(T) = \sum_{k=0}^m (-1)^k {m \choose k}{T^*}^kT^k,
\quad m\in \zbb_+.
   \end{align}
Recall that if $m\in \nbb$ and $\bscr_m(T)=0$,
then $T$ is called an {\em $m$-isometry} (see
\cite[p.\ 11]{Ag90} and
\cite{Ag-St1,Ag-St2,Ag-St3}). An $m$-isometry
$T$ is said to be {\em strict} if $m=1$ and
$\hh\neq \{0\}$, or $m\Ge 2$ and $T$ is not an
$(m-1)$-isometry; in both cases $\hh\neq \{0\}$
(see \cite{Bo-Ja}). Examples of strict
$m$-isometries for each $m\Ge 2$ are given in
\cite[Proposition~ 8]{At91}. We say that $T$ is
{\em $2$-hyperexpansive} if $\bscr_2(T) \Le 0$
(see \cite{Rich}). We call $T$ {\em completely
hyperexpansive} if $\bscr_m(T) \Le 0$ for all
$m\in \nbb$ (see \cite{At}).

Let $F\colon \ascr \to \ogr{\hh}$ be a {\em
semispectral measure} on a $\sigma$-algebra $\ascr$ of
subsets of a set $\varOmega$, i.e., $F$ is
$\sigma$-additive in the weak operator topology
(briefly, {\sc wot}) and $F(\varDelta)\Ge 0$ for all
$\varDelta\in \ascr$. Denote by $L^1(F)$ the linear
space of all complex $\ascr$-measurable functions
$\zeta$ on $\varOmega$ such that $\int_{\varOmega}
|\zeta(x)| \is{F(\D x)h}h < \infty$ for all $h\in
\hh$. Then for every $\zeta\in L^1(F)$, there exists a
unique operator $\int_\varOmega \zeta \D F \in
\ogr{\hh}$ such that (see e.g., \cite[Appendix]{Sto3})
   \begin{align}  \label{form-ua}
\Big\langle\int_\varOmega \zeta \D F h, h\Big\rangle =
\int_\varOmega \zeta(x) \is{F(\D x)h}h, \quad h\in\hh.
   \end{align}
If $\varOmega=\rbb, \cbb$ and $F\colon
\borel{\varOmega} \to \ogr{\hh}$ is a semispectral
measure, then its closed support is denoted by
$\supp{F}$ (recall that such $F$ is automatically
regular so $\supp{F}$ exists). By a {\em semispectral
measure} of a subnormal operator $T\in \ogr{\hh}$ we
mean a nor\-malized compactly supported semispectral
measure $G\colon \borel{\cbb} \to \ogr{\hh}$ defined
by $G(\varDelta) = PE(\varDelta)|_{\hh}$ for
$\varDelta \in\borel{\cbb}$, where $E\colon
\borel{\cbb} \to \ogr{\kk}$ is the spectral measure of
a minimal normal extension $N\in \ogr{\kk}$ of $T$ and
$P\in \ogr{\kk}$ is the orthogonal projection of $\kk$
onto $\hh$ (the minimality means that $\kk$ has no
proper closed vector subspace that reduces $N$ and
contains $\hh$). It follows from \cite[Proposition~
5]{Ju-St} and \cite[Proposition~II.2.5]{Con91} that a
subnormal operator has exactly one semispectral
measure. It is also easily seen that $T^{*n}T^n =
\int_{\cbb} |z|^{2n} G(\D z)$ for all $n\in \zbb_+$.
Applying \eqref{form-ua} and the measure transport
theorem (cf.\ \cite[Theorem~ 1.6.12]{Ash}) yields
   \begin{align}  \label{tobemom}
T^{*n}T^n = \int_{\rbb_+} x^n G\circ \phi^{-1}(\D x),
\quad n\in \zbb_+,
   \end{align}
where $\phi\colon \cbb\to \rbb_+$ is defined by
$\phi(z)=|z|^2$ for $z\in \cbb$ and
$G\circ\phi^{-1}\colon \borel{\rbb_+} \to \ogr{\hh}$
is the semispectral measure defined by
$G\circ\phi^{-1}(\varDelta) = G(\phi^{-1}(\varDelta))$
for $\varDelta \in \borel{\rbb_+}$. We refer the
reader to \cite{Con91} for the foundations of the
theory of subnormal operators.
   \section{\label{Sec2}Conditionally positive definite sequences}
   \subsection{\label{Subs2.1}Basic facts}
Let $\gammab = \{\gamma_n\}_{n=0}^{\infty}$ be a
sequence of real numbers. Recall that $\gammab$
is PD if \eqref{virek} holds for all finite
sequences $\{\lambda_i\}_{i=0}^{k} \subseteq
\cbb$; $\gammab$ is CPD if \eqref{virek} holds
for all finite sequences
$\{\lambda_i\}_{i=0}^{k} \subseteq \cbb$ such
that $\sum_{i=0}^k \lambda_i=0$. It is a matter
of routine to verify that $\gammab$ is PD
(resp., CPD) if and only if \eqref{virek} holds
for all finite sequences
$\{\lambda_i\}_{i=0}^{k} \subseteq \rbb$ (resp.,
for all finite sequences
$\{\lambda_i\}_{i=0}^{k} \subseteq \rbb$ such
that $\sum_{i=0}^k \lambda_i=0$). Another
important observation (cf.\
\cite[Remark~3.1.2]{B-C-R}) is that $\gammab$ is
PD (resp., CPD) if and only if
   \begin{align*}
\sum_{i,j=1}^k \gamma_{n_i+n_j} \lambda_i
\bar\lambda_j \Ge 0
   \end{align*}
for all finite sequences $\{n_i\}_{i=1}^k
\subseteq \zbb_+$ and $\{\lambda_i\}_{i=1}^k
\subseteq \cbb$ (resp., for all finite sequences
$\{n_i\}_{i=1}^k \subseteq \zbb_+$ and
$\{\lambda_i\}_{i=1}^k \subseteq \cbb$ such that
$\sum_{i=1}^{k} \lambda_i=0$). This shows that
our definitions of positive definiteness and
conditional positive definiteness are consistent
with those in \cite[Section~3.1]{B-C-R}. Let us
mention further that according to the
terminology in \cite{B-C-R}, $\gammab$ is CPD if
and only if $-\gammab$ is ``negative definite''.
It follows from the definition that if $\gammab$
is PD (resp., CPD), then so is the sequence
$\{\gamma_{n+2k}\}_{n=0}^{\infty}$ for every
$k\in \zbb_+$. However, it may happen that
$\gammab$ is PD but
$\{\gamma_{n+1}\}_{n=0}^{\infty}$ is not (e.g.,
$\gamma_n=(-1)^n$ for $n\in \zbb_+$).

The following fundamental characterization of
conditional positive definiteness in terms of positive
definiteness is essentially due to Schoenberg.
   \begin{lem}[\mbox{\cite[Lemma~1.7]{Pa-Sch}},
\mbox{\cite[Theorem~ 3.2.2]{B-C-R}}] \label{cpdpd2} If
$\gammab = \{\gamma_n\}_{n=0}^{\infty}$ is a sequence of
real numbers, then the following conditions are
equivalent{\em :}
   \begin{enumerate}
   \item[(i)] $\gammab$ is CPD,
   \item[(ii)] $\{\E^{t \gamma_n}\}_{n=0}^{\infty}$ is
PD for every positive real number $t$.
   \end{enumerate}
   \end{lem}
A sequence $\gammab = \{\gamma_n\}_{n=0}^{\infty}$ of
real numbers is said to be a {\em Hamburger} (resp.,
{\em Stieltjes}, {\em Hausdorff}\/) {\em moment
sequence} if there exists a Borel measure $\mu$ on
$\mathbb R$ (resp., $\rbb_+$, $[0,1]$) such that
   \begin{align} \label{hamb}
\gamma_n = \int t^n d \mu(t), \quad n\in \zbb_+.
   \end{align}
A Borel measure $\mu$ on $\rbb$ satisfying
\eqref{hamb} is called a {\em representing measure} of
$\gammab$. If $\gammab$ is a Hamburger moment sequence
which has a unique representing measure on $\rbb$,
then we say that $\gammab$ is {\em determinate}. Note
that by \cite[Ex.\ 4(e), p.\ 71]{Rud87}, the
Weierstrass theorem (see \cite[Theorem~7.26]{Rud76})
and the Riesz representation theorem (see
\cite[Theorem~2.14]{Rud87}) the following holds.
   \begin{lem} \label{csmad}
A Hamburger moment sequence
$\gammab=\{\gamma_n\}_{n=0}^{\infty}$ of real numbers
has a compactly supported representing measure if and
only if $\theta:=\limsup_{n\to \infty}
|\gamma_n|^{1/n} < \infty$. Moreover, if this is the
case, then $\gammab$ is determinate and $\supp{\mu}
\subseteq [-\theta, \theta]$, where $\mu$ is a unique
representing measure of $\gammab$.
   \end{lem}
In particular, a Hausdorff moment sequence is always
determinate. For our later needs, we recall a theorem
due to Stieltjes.
   \begin{thm}[\mbox{\cite{Sti},\cite[Theorem~ 6.2.5]{B-C-R}}] \label{Stiech}
A sequence $\{\gamma_n\}_{n=0}^{\infty}
\subseteq \rbb$ is a Stieltjes moment sequence
if and only if the sequences
$\{\gamma_n\}_{n=0}^{\infty}$ and
$\{\gamma_{n+1}\}_{n=0}^{\infty}$ are PD.
   \end{thm}
We refer the reader to \cite{B-C-R,sim} for the
fundamentals of the theory of moment problems.
   \subsection{\label{Sec2.2}Exponential growth} In this
subsection we give an integral representation
for CPD sequences of (at most) exponential
growth (see Theorem~ \ref{cpd-expon}). PD
sequences of exponential growth are
characterized by means of parameters appearing
in the above-mentioned integral representation
(see Theorem~\ref{dyszcz3}). Theorem~
\ref{Gyeon} states that a sequence
$\{\gamma_n\}_{n=0}^{\infty}$ of exponential
growth is PD if and only if the sequences
$\{\theta^n \gamma_n\}_{n=0}^{\infty}$, $\theta
\in \rbb$, are CPD.

We begin by introducing the difference transformation
$\triangle$ which plays an important role in further
considerations. Denote by $\cbb^{\zbb_+}$ the complex
vector space of all complex sequences
$\{\gamma_n\}_{n=0}^{\infty}$ with linear operations
defined coordinatewise. The difference transformation
$\triangle \colon \cbb^{\zbb_+} \to \cbb^{\zbb_+}$ is
given by
   \begin{align*}
(\triangle \gammab)_n = \gamma_{n+1} - \gamma_n, \quad n\in
\zbb_+, \, \gammab = \{\gamma_n\}_{n=0}^{\infty} \in
\cbb^{\zbb_+}.
   \end{align*}
Clearly, $\triangle$ is a linear. Denote by $\triangle^k$
the $k$th composition power of $\triangle$, i.e.,
$\triangle^0$ is the identity transformation of
$\cbb^{\zbb_+}$ and $\triangle^{k+1}\gammab =
\triangle^k(\triangle \gammab)$ for $\gammab =
\{\gamma_n\}_{n=0}^{\infty} \in \cbb^{\zbb_+}$.

   Given $n\in \zbb_+$, we define the polynomial $Q_n
\in \cbb[X]$ by
   \begin{align} \label{klaud}
Q_n(x) =
   \begin{cases}
0 & \text{if } x \in \rbb \text{ and } n=0,1,
   \\
\sum_{j=0}^{n-2} (n -j -1) x^j & \text{if } x \in \rbb
\text{ and } n=2,3,4,\dots.
   \end{cases}
   \end{align}
Below, for fixed $k\in \nbb$ and $x\in \rbb$, we
write $\triangle^k Q_{(\cdot)}(x)$ to denote the
action of the transformation $\triangle^k$ on
the sequence $\{Q_{n}(x)\}_{n=0}^{\infty}.$
   \begin{lem} \label{rnx}
The polynomials $Q_n$ have the following properties{\em :}
   \allowdisplaybreaks
   \begin{align}  \label{rnx-1}
Q_n(x) & = \frac{x^n-1 - n (x-1)}{(x-1)^2}, \quad n \in
\zbb_+, \, x\in \rbb\setminus \{1\},
   \\ \label{rnx-0}
Q_{n+1}(x) & = x Q_n(x) + n, \quad n \in \zbb_+, \, x\in
\rbb,
   \\ \label{monot-1}
\frac{Q_n(x)}{n} & \Le \frac{Q_{n+1}(x)}{n+1}, \quad n\in
\nbb, \, x \in [0,1],
   \\ \label{monot-2}
\lim_{n\to \infty} \frac{Q_n(x)}{n} & = \frac{1}{1-x},
\quad x\in [0,1),
   \\ \label{del1}
(\triangle Q_{(\cdot)}(x))_n & =
   \begin{cases}
0 & \text{if } n=0, \, x\in \rbb,
   \\
\sum_{j=0}^{n-1} x^j & \text{if } n \in \nbb, \, x\in
\rbb,
   \end{cases}
   \\ \label{del2}
(\triangle^2 Q_{(\cdot)}(x))_n & = x^n, \quad n\in \zbb_+,
\, x\in \rbb.
   \end{align}
   \end{lem}
   \begin{proof}
Suppose $n \Ge 2$. Then
   \begin{align*}
\frac{x^n-1 - n (x-1)}{(x-1)^2} &=
\frac{(\sum_{i=0}^{n-1} x^i) - n}{x-1} =
\sum_{i=0}^{n-1} \frac{x^i - 1}{x-1}
   \\
& = \sum_{i=1}^{n-1} \sum_{j=0}^{i-1} x^j =
\sum_{j=0}^{n-2} (n -j -1) x^j, \quad x\in
\rbb\setminus \{1\}.
   \end{align*}
This implies \eqref{rnx-1}. Identities
\eqref{rnx-0}, \eqref{del1} and \eqref{del2}
follow from \eqref{klaud} and the definition of
$\triangle$, while \eqref{monot-1} and
\eqref{monot-2} can be deduced from
\eqref{rnx-1}.
   \end{proof}
   Below, we denote by $|\mu|$ the total variation
measure of a complex Borel measure $\mu$ on $\rbb$.
Recall that a complex Borel measure on $\rbb$ is
automatically regular, i.e., its total variation
measure is regular (see \cite[Theorem~2.18]{Rud87}).
   \begin{lem} \label{uniq}
Suppose $a,b,c\in \cbb$ and $\mu$ is a complex Borel
measure on $\rbb$ such that $\mu(\{1\})=0$, the
measure $|\mu|$ is compactly supported and
   \begin{align*}
a+bn+cn^2+ \int_{\rbb} Q_n(x) \D\mu(x) = 0, \quad
n\in\zbb_+.
   \end{align*}
Then $a=b=c=0$ and $\mu=0$.
   \end{lem}
   \begin{proof}
Define $\gammab\in \cbb^{\zbb_+}$ by $\gamma_n=a+bn+cn^2+
\int_{\rbb} Q_n(x) \D\mu(x)$ for $n\in\zbb_+$. It follows
from \eqref{del2} that
   \begin{align*}
0=(\triangle^2 \gammab)_n = \int_{K} x^n \D
(\mu+2c\delta_1)(x), \quad n\in \zbb_+,
   \end{align*}
where $K:=\supp{|\mu+2c\delta_1|}$ is a compact subset
of $\rbb$. This implies that
   \begin{align*}
\int_{K} p(x) \D (\mu+2c\delta_1)(x) =0, \quad p\in
\cbb[X].
   \end{align*}
Applying the Weierstrass theorem and the uniqueness
part in the Riesz Representation Theorem (see
\cite[Theorem~6.19]{Rud87}), we deduce that
$(\mu+2c\delta_1)(\varDelta) = 0$ for all $\varDelta
\in \borel{\rbb}$. Substituting $\varDelta=\{1\}$, we
get $c=0$, and consequently $\mu=0$. Clearly,
$a=\gamma_0=0$. Putting all this together gives $b=0$,
completing the proof.
   \end{proof}
Now, for the reader's convenience we state
explicitly the fundamental characterization of
CPD sequences. Recall that a Borel measure on a
Hausdorff topological space is said to be {\em
Radon} if it is finite on compact sets and inner
regular with respect to compact sets.
   \begin{thm}[\mbox{\cite[Theorem 6.2.6]{B-C-R}}] \label{BCR} A sequence
$\gammab=\{\gamma_n\}_{n=0}^{\infty} \subseteq
\rbb$ is CPD if and only if it has a
representation of the form
   \begin{align*}
\gamma_n = \gamma_0 + bn + c n^2 + \int_{\rbb
\setminus \{1\}} (x^n-1 - n (x-1)) \D\mu(x), \quad
n\in \zbb_+,
   \end{align*}
where $b\in \rbb$, $c\in \rbb_+$ and $\mu$ is a Radon
measure on $\rbb\setminus \{1\}$ such that
   \allowdisplaybreaks
   \begin{gather*}
\int_{0 < |x-1| < 1} (x-1)^2 \D \mu(x) < \infty,
   \\
\int_{|x-1| \Ge 1} |x|^n \D \mu(x) < \infty, \quad
n\in \zbb_+.
   \end{gather*}
   \end{thm}
For our purpose, we need the following equivalent
variant of Theorem~ \ref{BCR}.
   \begin{thm}
\label{BCR-n} A sequence
$\gammab=\{\gamma_n\}_{n=0}^{\infty}$ of real
numbers is CPD if and only if it has a
representation of the form
   \begin{align} \label{abc1}
\gamma_n = \gamma_0 + bn + c n^2 + \int_{\rbb} Q_n(x)
\D\nu(x), \quad n\in \zbb_+,
   \end{align}
where $b\in \rbb$, $c\in \rbb_+$ and $\nu$ is a Borel
measure on $\rbb$ such that $\nu(\{1\})=0$ and
   \begin{gather}  \label{abc2}
\int_{\rbb} |x|^n \D \nu(x) < \infty, \quad n\in
\zbb_+.
   \end{gather}
   \end{thm}
   \begin{proof}
To prove the ``only if'' part apply
Theorem~ \ref{BCR} and define the
finite Borel measure $\nu$ on $\rbb$ by
   \begin{align*}
\nu(\varDelta) = \int_{\varDelta \cap (\rbb \setminus
\{1\})} (x-1)^2 \D \mu(x), \quad \varDelta \in
\borel{\rbb}.
   \end{align*}
Then, by Lemma~ \ref{rnx}, conditions
\eqref{abc1} and \eqref{abc2} are satisfied
(with the same $b,c$). The converse implication
goes through by applying Theorem~ \ref{BCR} to
the Radon measure $\mu$ defined by
   \begin{align*}
\mu(\varDelta) = \int_{\varDelta} (x-1)^{-2} \D
\nu(x), \quad \varDelta \in \borel{\rbb\setminus
\{1\}}.
   \end{align*}
That the so-defined $\mu$ is a Radon measure follows
from \cite[Theorem~2.18]{Rud87}.
   \end{proof}
CPD sequences of (at most) exponential growth
can be characterized as follows (below we use
the convention \eqref{konw-1}).
   \begin{thm} \label{cpd-expon}
Let $\gammab=\{\gamma_n\}_{n=0}^{\infty}$ be a
sequence of real numbers. Then the following
conditions are equivalent{\em :}
   \begin{enumerate}
   \item[(i)] $\gammab$ is CPD and there exist $\alpha,\theta \in \rbb_+$ such
that
   \begin{align*}
|\gamma_n| \Le \alpha\, \theta^n, \quad n\in \zbb_+,
   \end{align*}
   \item[(ii)] $\gammab$ is CPD and $\limsup_{n\to \infty}|\gamma_n|^{1/n} <
\infty$,
   \item[(iii)] there exist  $b\in \rbb$,
$c\in \rbb_+$ and a finite compactly supported Borel
measure $\nu$ on $\rbb$ such that $\nu(\{1\})=0$ and
   \begin{align} \label{cdr4}
\gamma_n = \gamma_0 + bn + c n^2 + \int_{\rbb} Q_n(x)
\D\nu(x), \quad n\in \zbb_+.
   \end{align}
   \end{enumerate}
Moreover, if {\em (iii)} holds, then the triplet
$(b,c,\nu)$ is unique and
   \allowdisplaybreaks
   \begin{gather} \label{limsup3}
\limsup_{n\to \infty}|\gamma_n|^{1/n} =\inf\Big\{\theta \in
\rbb_+\colon \exists \alpha\in \rbb_+ \, \forall n\in
\zbb_+ \; |\gamma_n| \Le \alpha\, \theta^n\Big\},
   \\ \label{limsup2}
c>0 \implies \limsup_{n\to \infty}|\gamma_n|^{1/n} \Ge 1,
   \\ \label{limsup}
\supp{\nu} \subseteq \Big[-\limsup_{n\to
\infty}|\gamma_n|^{1/n}, \limsup_{n\to
\infty}|\gamma_n|^{1/n}\Big],
   \\  \label{limsup1.5}
\sup_{x \in \supp{\nu}}|x| \Ge 1 \implies \limsup_{n\to
\infty}|\gamma_n|^{1/n} = \sup_{x \in \supp{\nu}}|x|,
   \\ \label{limsup1.6}
\limsup_{n\to \infty}|\gamma_n|^{1/n} \Le \max\bigg\{1,
\sup_{x \in \supp{\nu}}|x|\bigg\}.
   \end{gather}
   \end{thm}
   \begin{proof}
It is a matter of routine to show that
conditions (i) and (ii) are equivalent.

(ii)$\Rightarrow$(iii) By Theorem~ \ref{BCR-n},
there exist $b\in \rbb$, $c\in \rbb_+$ and a
finite Borel measure $\nu$ on $\rbb$ that
satisfy conditions \eqref{abc1} and \eqref{abc2}
and the equality $\nu(\{1\})=0$. It follows from
\eqref{del2} and \eqref{abc1} that
   \begin{align} \label{rca1}
(\triangle^2 \gammab)_n = \int_{\rbb} x^n \D
(\nu+2c\delta_1)(x), \quad n\in \zbb_+.
   \end{align}
Noting that
   \begin{align} \label{li-le}
\limsup_{n\to \infty}|(\triangle^2\gammab)_n|^{1/n}
\Le \limsup_{n\to \infty}|\gamma_n|^{1/n}
   \end{align}
and using Lemma~\ref{csmad}, we infer from
\eqref{rca1} that
   \begin{align} \label{rca2}
(\nu+2c\delta_1)\Big(\Big\{x\in \rbb\colon |x| >
\limsup_{n\to \infty}|\gamma_n|^{1/n}\Big\}\Big) = 0.
   \end{align}
This implies \eqref{limsup}, which gives (iii).

(iii)$\Rightarrow$(ii) The conditional
positive definiteness of $\gammab$
follows from Theorem~ \ref{BCR-n},
while the inequality $\limsup_{n\to
\infty}|\gamma_n|^{1/n} < \infty$ can
be deduced straightforwardly from
\eqref{klaud} and \eqref{cdr4}.

It remains to complete the proof of the
``moreover'' part. The uniqueness of the triplet
$(b,c,\nu)$ in (iii) follows from
Lemma~\ref{uniq}. Identity \eqref{limsup3} is a
well-known fact in analysis. Condition
\eqref{limsup2} can be deduced from
\eqref{rca2}. To prove \eqref{limsup1.5}, assume
that $R:=\sup_{x \in \supp{\nu}}|x| \Ge 1$ (then
$\supp{\nu} \neq \emptyset$). It is a matter of
routine to deduce from \eqref{klaud} and
\eqref{cdr4} that there exists a constant
$\alpha\in \rbb_+$ such that
   \begin{align*}
|\gamma_n| \Le \alpha \, n^2 R^n, \quad n\in\nbb.
   \end{align*}
This implies that $\limsup_{n\to
\infty}|\gamma_n|^{1/n} \Le R$. Combined with
\eqref{limsup}, this yields $\limsup_{n\to
\infty}|\gamma_n|^{1/n}=R$. Hence \eqref{limsup1.5}
holds. In view of \eqref{limsup1.5}, to prove
\eqref{limsup1.6}, it suffices to consider the case
when $\nu(\rbb\setminus [-1,1])=0$. Note that
   \begin{align*}
|Q_n(x)| \overset{\eqref{klaud}}\Le
\sum_{j=0}^{n-2} (n -j -1) |x|^j \Le
n^2, \quad x \in [-1,1], \, n\Ge 2.
   \end{align*}
Combined with \eqref{cdr4}, this implies that
$|\gamma_n| \Le \alpha \cdot n^2$ for all $n\in \nbb$
with $\alpha=|\gamma_0| + |b| + c + \nu(\rbb)$, so
$\limsup_{n\to \infty}|\gamma_n|^{1/n} \Le 1$. This
completes the proof.
   \end{proof}
   \begin{dfn} \label{deftryp}
If $\gammab=\{\gamma_n\}_{n=0}^{\infty}$ is a
CPD sequence such that
   \begin{align*}
\limsup_{n\to \infty}|\gamma_n|^{1/n} < \infty
   \end{align*}
and $b$, $c$ and $\nu$ are as in statement (iii)
of Theorem~ \ref{cpd-expon}, we call $(b,c,\nu)$
the {\em representing triplet} of $\gammab$, or
we simply say that $(b,c,\nu)$ {\em represents}
$\gammab$.
   \end{dfn}
The following example shows that the converse to
the implication~\eqref{limsup1.5} in
Theorem~\ref{cpd-expon} is not true in general.
   \begin{exa}
For $\theta \in (0,1),$ let
$\gammab=\{\gamma_n\}_{n=0}^{\infty}$ be the sequence
of real numbers defined by
   \begin{align*}
\gamma_n= \frac{\theta^n}{(\theta-1)^2}
\overset{\eqref{rnx-1}}= \frac{1}{(\theta-1)^2} +
\frac{n}{\theta-1} + Q_n(\theta), \quad n\in \zbb_+.
   \end{align*}
Then the sequence $\gammab$ is PD and thus CPD.
Its representing triplet $(b,c,\nu)$ takes the
form $b=\frac{1}{\theta-1}$, $c=0$ and
$\nu=\delta_{\theta}$. Moreover, we have
   \begin{align*}
\limsup_{n\to \infty}|\gamma_n|^{1/n} = \sup_{x \in
\supp{\nu}}|x| = \theta <1,
   \end{align*}
as required.
   \hfill $\diamondsuit$
   \end{exa}
Here are more examples of CPD sequences of
exponential growth.
   \begin{exa}
It is a matter of direct computation to see that
if $\mu$ is a finite Borel measure on $\rbb_+$,
then the sequence $\left\{\int_{[0,1)}
\frac{x^n-1}{1-x} \D
\mu(x)\right\}_{n=0}^{\infty}$ is CPD and
   \begin{align*}
\int_{[0,1)} \frac{x^n-1}{1-x} \D \mu(x)
\overset{\eqref{rnx-1}}= - n \mu([0,1)) +
\int_{\rbb_+} Q_n(x) \D\nu(x), \quad n \in \zbb_+,
   \end{align*}
where
   \begin{align*}
\nu(\varDelta) = \int_{\varDelta \cap [0,1)} (1-x)
\D\mu(x), \quad \varDelta \in \borel{\rbb_+}.
   \end{align*}
Similarly, if $\mu$ is a finite compactly
supported Borel measure on $\rbb_+$, then the
sequence $\left\{\int_{(1,\infty)}
\frac{x^n-1}{x-1} \D
\mu(x)\right\}_{n=0}^{\infty}$ is CPD and
   \begin{align*}
\int_{(1,\infty)} \frac{x^n-1}{x-1} \D \mu(x)
\overset{\eqref{rnx-1}}= n \mu((1,\infty)) +
\int_{\rbb_+} Q_n(x) \D\nu(x), \quad n \in \zbb_+,
   \end{align*}
where
   \begin{align*}
\nu(\varDelta) = \int_{\varDelta \cap (1,\infty)}
(x-1) \D\mu(x), \quad \varDelta \in \borel{\rbb_+}.
\tag*{$\diamondsuit$}
   \end{align*}
   \end{exa}
Yet another characterization of CPD sequences of
exponential growth is given below. Let us
mention that in view of
\cite[Theorem~4.6.11]{B-C-R} sequences
$\gammab=\{\gamma_n\}_{n=0}^{\infty} \subseteq
\rbb$ for which $\triangle^2\gammab$ is a
Hausdorff moment sequence coincide with
completely monotone sequences of order $2$
introduced in \cite{Cha-Sh}.
   \begin{pro}\label{traj-pd}
Let $\gammab=\{\gamma_n\}_{n=0}^{\infty}$ be a
sequence of real numbers such that $\limsup_{n\to
\infty}|\gamma_n|^{1/n} < \infty$. Then the following
conditions are equivalent{\em :}
   \begin{enumerate}
   \item[(i)] $\gammab$ is CPD
$($resp., $\gammab$ is CPD with the representing
triplet $(b,c,\nu)$ such that $\supp{\nu}
\subseteq \rbb_+$$)$,
   \item[(ii)] $\triangle^2\gammab$ is PD $($resp., $\triangle^2\gammab$ and
$\{(\triangle^2\gammab)_{n+1}\}_{n=0}^{\infty}$
are PD$)$,
   \item[(iii)] $\triangle^2\gammab$ is a Hamburger
moment sequence $($resp., $\triangle^2\gammab$ is a
Stieltjes moment sequence$)$.
   \end{enumerate}
Moreover, if $\gammab$ is CPD and has a
representing triplet $(b,c,\nu)$ such that
$\supp{\nu} \subseteq \rbb_+$, then the sequence
$\triangle\gammab$ is monotonically increasing.
   \end{pro}
   \begin{proof}
(i)$\Rightarrow$(iii) Apply \eqref{rca1} to both
versions.

(iii)$\Rightarrow$(i) Suppose that
$\triangle^2\gammab$ is a Hamburger moment
sequence. Let $\mu$ be a representing measure of
$\triangle^2 \gammab$. Using Lemma~\ref{csmad}
and \eqref{li-le} we deduce that $\mu$ is
compactly supported. Note that
   \begin{align} \label{by-sie}
(\triangle^k \gammab)_n = (\triangle^{k-2}
(\triangle^2\gammab))_n = \int_{\rbb}
(x-1)^{k-2}x^n \D \mu(x), \quad n\in \zbb_+, k
\Ge 2.
   \end{align}
Let $\nu$ be the finite compactly supported Borel
measure on $\rbb$ given by
   \begin{align} \label{duzo-duzo}
\nu(\varDelta) = \mu(\varDelta \setminus \{1\}), \quad
\varDelta \in \borel{\rbb}.
   \end{align}
Applying Newton's binomial formula to
$\triangle$ (see \cite[(2.2)]{J-J-S20}), we
obtain
   \allowdisplaybreaks
   \begin{align*}
\gamma_n & = \sum_{k=0}^n \binom{n}{k}
(\triangle^{k} \gammab)_0
      \\
& \hspace{-2.2ex}
\overset{\eqref{by-sie}}=\gamma_0+ n (\gamma_1 -
\gamma_0) + \int_{\rbb} \sum_{k=2}^n
\binom{n}{k} (x-1)^{k-2} \D \mu(x)
   \\
& = \gamma_0+ n (\gamma_1 - \gamma_0) +
\frac{n(n-1)}{2} \mu(\{1\}) + \int_{\rbb}
\frac{\sum_{k=2}^n \binom{n}{k}
(x-1)^k}{(x-1)^2} \D \nu(x)
   \\
& \hspace{-1.7ex}\overset{\eqref{rnx-1}}=
\gamma_0 + n \Big(\gamma_1 - \gamma_0 -
\frac{1}{2}\mu(\{1\})\Big) + \frac{n^2}{2}
\mu(\{1\}) + \int_{\rbb} Q_n(x) \D \nu(x), \quad
n\Ge 2.
   \end{align*}
This implies that condition (iii) of
Theorem~\ref{cpd-expon} holds with $\nu$ as in
\eqref{duzo-duzo} and the parameters $b$ and $c$
defined by
   \begin{align*}
b=\gamma_1 - \gamma_0 - \frac{1}{2}\mu(\{1\})
\quad \text{and} \quad c=\frac{1}{2} \mu(\{1\}).
   \end{align*}
Thus $\gammab$ is CPD with the representing
triplet $(b,c,\nu)$. Clearly, by
\eqref{duzo-duzo}, $\supp{\mu} \subseteq \rbb_+$
if and only if $\supp{\nu} \subseteq \rbb_+$.
All this together proves both versions of the
implication (iii)$\Rightarrow$(i).

(ii)$\Leftrightarrow$(iii) Use
\cite[Theorem~6.2.2]{B-C-R} (resp.,
Theorem~\ref{Stiech}).

Since $\triangle\gammab$ is monotonically increasing
if and only if $\triangle^2\gammab \Ge 0$, the
``moreover'' part follows from \eqref{by-sie} applied
to $k=2$ and \eqref{duzo-duzo}.
   \end{proof}
Now we give necessary and sufficient conditions
for a CPD sequence to have a polynomial growth
of degree at most $2$.
   \begin{pro} \label{grown}
Let $\gammab=\{\gamma_n\}_{n=0}^{\infty}$ be a
CPD sequence with the representing triplet
$(b,c,\nu).$ Then the following conditions are
equivalent{\em :}
   \begin{enumerate}
   \item[(i)] there exists $\alpha\in \rbb_+$ such that
   \begin{align} \label{Vang0-skal}
|\gamma_n| \Le \alpha \cdot n^2, \quad n\in \nbb,
   \end{align}
   \item[(ii)] $\limsup_{n\to \infty}|\gamma_n|^{1/n} \Le
1,$
   \item[(iii)] $\supp \nu \subseteq [-1,1].$
   \end{enumerate}
Moreover, {\em (iii)} implies \eqref{Vang0-skal} with
$\alpha=|\gamma_0|+|b|+c+\nu(\rbb)$.
   \end{pro}
   \begin{proof}
(i)$\Rightarrow$(ii) This implication is obvious.

(ii)$\Rightarrow$(iii) It suffices to apply
\eqref{limsup}.

(iii)$\Rightarrow$(i) Arguing as in the proof of
\eqref{limsup1.6}, one can verify that
inequality \eqref{Vang0-skal} holds with
$\alpha:=|\gamma_0|+|b|+c+\nu(\rbb)$, which
completes the proof.
   \end{proof}
The above lemma enables us to prove the following.
   \begin{pro} \label{ojoj1}
Let $p$ be a polynomial in one indeterminate with real
coefficients. Then the following conditions are
equivalent{\em :}
   \begin{enumerate}
   \item[(i)] the sequence $\{p(n)\}_{n=0}^{\infty}$
is CPD,
   \item[(ii)] either $\deg p \Le 1$ or $\deg p = 2$ and the leading
coefficient of $p$ is positive.
   \end{enumerate}
   \end{pro}
   \begin{proof}
Suppose that the sequence
$\{p(n)\}_{n=0}^{\infty}$ is CPD. Since
   \begin{align*}
\limsup_{n\to \infty}|p(n)|^{1/n} \Le 1,
   \end{align*}
we infer from Proposition~\ref{grown} that there
exists $\alpha\in \rbb_+$ such that $|p(n)| \Le
\alpha \cdot n^2$ for all $n\in \nbb$. As a
consequence, $\deg p \Le 2$. Straightforward
computations complete the proof.
   \end{proof}
PD sequences of exponential growth can be
characterized by means of parameters describing
conditional positive definiteness given in
Theorem~ \ref{cpd-expon}(iii).
   \begin{thm} \label{dyszcz3}
Let $\gammab=\{\gamma_n\}_{n=0}^{\infty}$ be a
sequence of real numbers such that
$\limsup_{n\to \infty}|\gamma_n|^{1/n} <
\infty$. Then the following conditions are
equivalent\/\footnote{If the inequality in (iii)
holds, then by the Cauchy-Schwarz inequality,
\mbox{$\frac{1}{x-1}\in L^1(\nu).$}}{\em :}
   \begin{enumerate}
   \item[(i)] $\gammab$ is PD,
   \item[(ii)] $\gammab$ is a Hamburger moment sequence,
   \item[(iii)] $\gammab$ is CPD,
$\int_{\rbb} \frac{1}{(x-1)^2} \D \nu(x) \Le
\gamma_0$, $b=\int_{\rbb} \frac{1}{x-1} \D
\nu(x)$ and $c=0$, where $(b,c,\nu)$ is a
representing triplet of $\gammab$.
   \end{enumerate}
Moreover, if {\em (iii)} holds, then $\gammab$
is a determinate Hamburger moment sequence, its
unique representing measure $\mu$ is compactly
supported, and the following identities~hold{\em
:}
   \allowdisplaybreaks
   \begin{align} \label{zeg1}
\mu(\varDelta) & =\int_{\varDelta}
\frac{1}{(x-1)^2} \D \nu(x) + \Big(\gamma_0 -
\int_{\rbb} \frac{1}{(x-1)^2} \D \nu(x)\Big)
\delta_1(\varDelta), \;\; \varDelta \in
\borel{\rbb},
   \\ \label{zeg3}
b & = \int_{\rbb} (x-1) \D \mu(x),
   \\ \label{zeg2}
\nu(\varDelta) & = \int_{\varDelta} (x-1)^2 \D
\mu(x), \quad \varDelta \in \borel{\rbb}.
   \end{align}
   \end{thm}
   \begin{proof}
The equivalence (i)$\Leftrightarrow$(ii) follows from
\cite[Theorem~ 6.2.2]{B-C-R}.

(ii)$\Rightarrow$(iii) Clearly, $\gammab$ is
CPD. Denote by $(b,c,\nu)$ the representing
triplet of $\gammab$. Let $\mu$ be a
representing measure of $\gammab$, that is
   \begin{align} \label{csmad2}
\gamma_n = \int_{\rbb} x^n \D \mu(x), \quad n\in
\zbb_+.
   \end{align}
By Lemma~\ref{csmad}, $\gammab$ is determinate and
$\mu$ is compactly supported. Note that
   \begin{align*}
\int_{\rbb} x^n(x-1)^2 \D\mu(x)
\overset{\eqref{csmad2}}= (\triangle^2 \gammab)_n
\overset{\eqref{rca1}} = \int_{\rbb} x^n \D
(\nu+2c\delta_1)(x), \quad n\in \zbb_+.
   \end{align*}
Since the measure $\nu+2c\delta_1$ is compactly
supported, we infer from Lemma~\ref{csmad} that
   \begin{align} \label{csmad3}
\int_{\varDelta} (x-1)^2 \D \mu(x) =
(\nu+2c\delta_1)(\varDelta), \quad \varDelta \in
\borel{\rbb}.
   \end{align}
Substituting $\varDelta=\{1\}$ into \eqref{csmad3}, we
deduce that $c=0$. Combined with \eqref{csmad3}, this
implies \eqref{zeg2}. As a consequence of \eqref{zeg2}
and $\nu(\{1\})=0$, we have
   \begin{align} \label{zeg1+}
\mu(\varDelta) & =\int_{\varDelta} \frac{1}{(x-1)^2}
\D \nu(x) + \mu(\{1\}) \delta_1(\varDelta), \quad
\varDelta \in \borel{\rbb}.
   \end{align}
Since $1$ is a common root of the polynomials $X^n -1
- n (X-1)$, where $n\in \zbb_+$, and $\nu(\{1\})=0$,
it follows from Lemma~\ref{rnx} that
   \allowdisplaybreaks
   \begin{align} \notag
\gamma_n &\overset{\eqref{cdr4}}= \gamma_0 + b n +
\int_{\rbb} \frac{x^n-1 - n (x-1)}{(x-1)^2} \D \nu(x)
   \\ \notag
& \overset{\eqref{zeg2}} = \gamma_0 + b n +
\int_{\rbb} \Big(x^n-1 - n (x-1)\Big)\D \mu(x)
   \\ \notag
& \hspace{2.2ex}= (\gamma_0 - \mu(\rbb)) +
\Big(b-\int_{\rbb}(x-1)\D \mu(x)\Big) n + \int_{\rbb}
x^n \D\mu(x)
      \\ \label{gamm}
&\overset{\eqref{csmad2}}= (\gamma_0 - \mu(\rbb)) +
\Big(b-\int_{\rbb}(x-1)\D \mu(x)\Big) n + \gamma_n,
\quad n\in \zbb_+.
   \end{align}
Hence, we have
   \begin{gather} \label{zeg4}
\gamma_0 = \mu(\rbb) \overset{\eqref{zeg1+}}=
\int_{\rbb} \frac{1}{(x-1)^2} \D \nu(x) + \mu(\{1\})
\Ge \int_{\rbb} \frac{1}{(x-1)^2} \D \nu(x),
   \end{gather}
and $b=\int_{\rbb}(x-1)\D \mu(x),$ which yields
\eqref{zeg3} and the inequality in (iii). Using
\eqref{zeg3}, \eqref{zeg1+} and
\cite[Theorem~1.29]{Rud87}, we deduce that $b =
\int_{\rbb} \frac{1}{x-1} \D \nu(x)$.
Summarizing, we have proved that (iii) holds. It
follows from \eqref{zeg4} that
   \begin{align*}
\mu(\{1\}) = \gamma_0 - \int_{\rbb} \frac{1}{(x-1)^2}
\D \nu(x).
   \end{align*}
Combined with \eqref{zeg1+}, this implies
\eqref{zeg1}. This also justifies the ``moreover''
part.

(iii)$\Rightarrow$(ii) It follows from the inequality
in (iii) that the formula
   \begin{align} \label{zim1}
\mu(\varDelta) =\int_{\varDelta} \frac{1}{(x-1)^2} \D
\nu(x) + \bigg(\gamma_0 - \int_{\rbb}
\frac{1}{(x-1)^2} \D \nu(x)\bigg) \delta_1(\varDelta),
\quad \varDelta \in \borel{\rbb},
   \end{align}
defines a finite compactly supported Borel measure
$\mu$ on $\rbb$. Arguing as in the first three lines
of \eqref{gamm} and using \eqref{zim1} instead of
\eqref{zeg2}, we verify that \eqref{csmad2} is
satisfied. This completes the proof.
   \end{proof}
In view of the Schur product theorem (see
\cite[p.\ ~14]{sch} or \cite[Theorem
7.5.3]{Hor-Joh}), the product of two PD
sequences is PD; this is no longer true for CPD
sequences, e.g., the powers
$\{n^{2k}\}_{n=0}^{\infty}$, $k=2,3, \ldots$, of
the CPD sequence $\{n^2\}_{n=0}^{\infty}$ are
not CPD (see Proposition~\ref{ojoj1}). As a
consequence, if
$\gammab=\{\gamma_n\}_{n=0}^{\infty}$ is a PD
sequence, then the product sequence
$\{\xi_n\gamma_n\}_{n=0}^{\infty}$ is CPD for
every PD sequence $\{\xi_n\}_{n=0}^{\infty}$.
Below, we show that the converse implication is
true for sequences $\gammab$ of exponential
growth. What is more, the above equivalence
remains true if the class of all PD sequences
$\{\xi_n\}_{n=0}^{\infty}$ is reduced
drastically to the class of the sequences of the
form $\{\theta^n\}_{n=0}^{\infty}$, where
$\theta \in \rbb$.
   \begin{thm}\label{Gyeon}
Suppose that $\{\gamma_n\}_{n=0}^{\infty}$ is a
sequence of real numbers such that
$\limsup_{n\to\infty} |\gamma_n|^{1/n} < \infty$. Then
the following conditions are equivalent{\em :}
   \begin{enumerate}
   \item[(i)] the sequence $\{\gamma_n\}_{n=0}^{\infty}$ is PD,
   \item[(ii)] the sequence  $\{\theta^n \gamma_n\}_{n=0}^{\infty}$ is CPD for all $\theta \in \rbb$,
   \item[(iii)] zero is an accumulation point
of the set of all $\theta \in \rbb\setminus
\{0\}$ for which the sequence $\{\theta^n
\gamma_n\}_{n=0}^{\infty}$ is CPD,
   \item[(iv)] there exists $\theta \in \rbb\setminus \{0\}$
such that $|\theta| \cdot \limsup_{n\to\infty}
|\gamma_n|^{1/n}< 1$ and the sequence
$\{\theta^n \gamma_n\}_{n=0}^{\infty}$ is CPD.
   \end{enumerate}
   \end{thm}
   \begin{proof}
The implication (i)$\Rightarrow$(ii) is a direct
consequence of the Schur product theorem. The implications
(ii)$\Rightarrow$(iii) and (iii)$\Rightarrow$(iv) are
obvious.

(iv)$\Rightarrow$(i) Replacing
$\{\gamma_n\}_{n=0}^{\infty}$ by
$\{\theta^n\gamma_n\}_{n=0}^{\infty}$ if
necessary, we can assume that
$\{\gamma_n\}_{n=0}^{\infty}$ is CPD and
   \begin{align*}
r:=\limsup_{n\to\infty} |\gamma_n|^{1/n} < 1.
   \end{align*}
Then $\lim_{n\to\infty} \gamma_n=0$ and consequently
   \begin{align} \label{ditn}
\lim_{n\to\infty} (\triangle^j\gammab)_n=0, \quad j\in
\zbb_+,
   \end{align}
where $\gammab:=\{\gamma_n\}_{n=0}^{\infty}$. Let
$(b,c,\nu)$ be the representing triplet of $\gammab$. It
follows from \eqref{limsup} that $\supp{\nu} \subseteq
[-r,r]$. Thus, by \eqref{rca1}, we have
   \begin{align*}
(\triangle^2\gammab)_n = 2c + \int_{[-r,r]} x^n \D
\nu(x), \quad n\in \zbb_+.
   \end{align*}
Using \eqref{ditn} for $j=2$ and Lebesgue's dominated
convergence theorem, we deduce that $c=0$. In view of
\eqref{del1} and \eqref{cdr4}, we get
   \begin{align}  \label{dziadzia}
(\triangle\gammab)_n = b + \int_{[-r,r]}
\frac{1-x^n}{1-x} \D \nu(x), \quad n\in \zbb_+.
   \end{align}
Since $r<1$, we see that $\frac{1}{(1-x)^j} \in
L^{\infty}(\nu) \subseteq L^1(\nu)$ for all $j\in
\zbb_+$. Hence, it follows from \eqref{ditn} for
$j=1$, \eqref{dziadzia} and Lebesgue's dominated
convergence theorem that $b=\int_{[-r,r]}
\frac{1}{x-1} \D \nu(x)$. According to \eqref{rnx-1}
and \eqref{cdr4}, we have \allowdisplaybreaks
   \begin{align*}
\gamma_n & = \gamma_0 + \int_{[-r,r]}
\bigg(\frac{n}{x-1} + \frac{x^n-1 - n
(x-1)}{(x-1)^2}\bigg) \D\nu(x)
   \\
& = \gamma_0 + \int_{[-r,r]} \frac{x^n-1}{(x-1)^2}
\D\nu(x), \quad \quad n\in \zbb_+.
   \end{align*}
Using \eqref{ditn} for $j=0$ and Lebesgue's
dominated convergence theorem, we conclude that
$\int_{[-r,r]} \frac{1}{(x-1)^2} \D\nu(x) =
\gamma_0$. Applying Theorem~\ref{dyszcz3} shows
that $\gammab$ is PD. This completes the proof.
   \end{proof}
   \subsection{\label{Sec2.3}Additional constraints}
In this subsection we characterize CPD sequences
$\gammab=\{\gamma_n\}_{n=0}^{\infty}$ of
exponential growth for which the sequence of
consecutive differences $\triangle \gammab$ is
either convergent (see
Theorem~\ref{boundiff-scalar}) or bounded from
above plus some additional constraints (see
Theorem~\ref{boundiff-scalar2}). As a
consequence, under slightly stronger hypotheses
than those of Theorem~\ref{boundiff-scalar2}, we
show that CPD sequences $\gammab$ of exponential
growth with $\lim_{n\to \infty} (\triangle
\gammab)_n = 0$ are PD (see
Corollary~\ref{pd2cpd}).

   We begin by proving a simple lemma on backward
growth estimates for powers of the difference
transformation $\triangle$.
   \begin{lem} \label{backest}
Let $\gammab=\{\gamma_n\}_{n=0}^{\infty}$ be a sequence of
real $($resp., complex$)$ numbers and $k\in \nbb$ be such
that
   \begin{align} \label{supk-a}
\sup_{n\in \nbb} (\triangle^k\gammab)_n < \infty \quad
\bigg(\text{resp., }\sup_{n\in \nbb}
|(\triangle^k\gammab)_n| < \infty\bigg).
   \end{align}
Then
   \begin{align} \label{supk}
\sup_{n\in \nbb} \frac{(\triangle^j \gammab)_n}{n^{k-j}} <
\infty \quad \bigg(\text{resp., }\sup_{n\in \nbb}
\frac{|(\triangle^j \gammab)_n|}{n^{k-j}} < \infty\bigg),
\quad j=0, \ldots, k.
   \end{align}
   \end{lem}
   \begin{proof}
Because of the similarity of proofs, we
concentrate on the real case. We use the
backward induction on $j$. By the first
inequality in \eqref{supk-a}, the first
inequality in \eqref{supk} holds for $j=k$. If
the first inequality in \eqref{supk} holds for a
fixed $j\in \{1,\ldots,k\}$, there exists
$\eta\in \rbb_+$ such that
   \allowdisplaybreaks
   \begin{align} \notag
(\triangle^{j-1}\gammab)_n =
(\triangle^{j-1}\gammab)_0 + \sum_{m=0}^{n-1}
(\triangle^{j}\gammab)_{m} & \Le
(\triangle^{j-1}\gammab)_0 + \eta \sum_{m=0}^{n-1}
(m+1)^{k-j}
   \\  \label{te-skop}
& \Le (\triangle^{j-1}\gammab)_0 + \eta \, n^{k-j+1},
\quad n\in \nbb.
   \end{align}
Hence the first inequality in \eqref{supk} holds for
$j-1$ in place of $j$.
   \end{proof}
Next, we characterize CPD sequences $\gammab$
for which the sequence $\triangle\gammab$ is
convergent.
   \begin{thm} \label{boundiff-scalar}
Let $\gammab=\{\gamma_n\}_{n=0}^{\infty}$ be a
sequence of real numbers. Then the following
statements are equivalent\/\footnote{\label{stop-a1}By
Lemma~\ref{backest} applied to $k=1$ and $j=0$,
\eqref{adycto} implies that $\limsup_{n\to
\infty}|\gamma_n|^{1/n} \Le 1$.}{\em :}
   \begin{enumerate}
   \item[(i)] $\gammab$ is CPD
and
   \begin{align} \label{adycto}
\text{the sequence $\triangle\gammab$ is convergent in
$\rbb$,}
   \end{align}
   \item[(ii)]  there exist a finite Borel measure $\nu$ on
$\rbb$ and $d\in \rbb$ such that
   \begin{enumerate}
   \item[(ii-a)] $\nu(\rbb\setminus (-1,1))=0$,
   \item[(ii-b)] $\frac{1}{1-x}\in L^1(\nu)$,
   \item[(ii-c)] $\gamma_n = \gamma_0 +  n d  -  \int_{(-1,1)}
\frac{1-x^n}{(1-x)^2} \D \nu(x)$ for all $n\in \zbb_+$.
   \end{enumerate}
   \end{enumerate}
Moreover, the following statements are satisfied{\em :}
   \begin{enumerate}
   \item[(iii)] if {\em (i)} holds
and $(b,c,\nu)$ represents $\gammab$, then $c=0$,
$\frac{1}{1-x}\in L^1(\nu)$ and the pair $(d,\nu)$
with $d=b+\int_{(-1,1)} \frac{1}{1-x} \D \nu(x)$ is a
unique pair satisfying {\em (ii)},
   \item[(iv)] if {\em (ii)} holds, then $d=\lim_{n\to
\infty}(\triangle{\gammab})_n$ and $(b,0,\nu)$
represents $\gammab$ with $b=d-\int_{(-1,1)}
\frac{1}{1-x} \D \nu(x)$.
   \end{enumerate}
   \end{thm}
   \begin{proof}
(i)$\Rightarrow$(ii) It follows from
footnote~\ref{stop-a1} and \eqref{limsup} that
$\supp{\nu} \subseteq [-1,1]$, where $(b,c,\nu)$
represents $\gammab$. By using \eqref{rca1}, we
get
   \begin{align} \label{backinDa}
(\triangle^2\gammab)_n = 2c + (-1)^n \nu(\{-1\}) +
\int_{(-1,1)} x^n \D \nu(x), \quad n\in \zbb_+.
   \end{align}
It follows from \eqref{adycto} that $\lim_{n\to
\infty} \triangle^2\gammab = 0$. By Lebesgue's
dominated convergence theorem, the third term on
the right-hand side of the equality in
\eqref{backinDa} converges to $0$. This together
with \eqref{backinDa} implies that $c=0$ and
$\nu(\{-1\})=0$, which gives \mbox{(ii-a)}. Now,
using \eqref{del1} and \eqref{cdr4}, we obtain
   \begin{align} \label{backinDa2}
(\triangle \gammab)_n = b + \int_{(-1,0)} \frac{1-x^n}{1-x}
\D \nu(x) + \int_{[0,1)} \frac{1-x^n}{1-x} \D \nu(x), \quad
n\in \zbb_+.
   \end{align}
Applying Lebesgue's dominated and monotone
convergence theorems to the second and the third
terms on the right-hand side of the equality in
\eqref{backinDa2} respectively, we infer from
\eqref{adycto} that \mbox{(ii-b)} holds and
   \begin{align} \label{adycto2}
b=d-\int_{(-1,1)} \frac{1}{1-x} \D \nu(x),
   \end{align}
where $d:=\lim_{n\to \infty}(\triangle{\gammab})_n$. Using
again \eqref{cdr4}, we get
   \allowdisplaybreaks
   \begin{multline}  \label{adycto3}
\gamma_n = \gamma_0 + bn + \int_{(-1,1)} Q_n(x) \D\nu(x)
   \\
\overset{(*)} = \gamma_0 + n d + \int_{(-1,1)}
\bigg(\frac{n}{x-1} + \frac{x^n-1 - n
(x-1)}{(x-1)^2} \bigg) \D\nu(x)
   \\
= \gamma_0 + n d - \int_{(-1,1)}
\frac{1-x^n}{(1-x)^2} \D \nu(x), \quad n\in
\zbb_+,
   \end{multline}
where $(*)$ follows from \eqref{adycto2} and
\eqref{rnx-1}. This implies (ii) and (iii) except for
the uniqueness of $(d,\nu)$.

(ii)$\Rightarrow$(i) Using (ii-b) and (ii-c) and arguing as
in \eqref{adycto3}, we see that
   \begin{align} \label{adycto5}
\gamma_n = \gamma_0 + bn + \int_{(-1,1)} Q_n(x) \D\nu(x),
\quad n\in \zbb_+,
   \end{align}
where $b$ is as in \eqref{adycto2}. Hence, by
Theorem \ref{cpd-expon} the sequence $\gammab$
is CPD and $\limsup_{n\to
\infty}|\gamma_n|^{1/n} < \infty$. By
\mbox{(ii-a)}, $(b,0,\nu)$ represents $\gammab$.
It follows from \eqref{del1} and \eqref{adycto5}
that
   \begin{align} \label{adycto6}
(\triangle \gammab)_n = b + \int_{(-1,1)} \frac{1-x^n}{1-x}
\D \nu(x), \quad n\in \zbb_+.
   \end{align}
Using \mbox{(ii-b)} and applying Lebesgue's dominated
convergence theorem to \eqref{adycto6}, we see that
\eqref{adycto} holds and $d=\lim_{n\to
\infty}(\triangle{\gammab})_n$. Summarizing, we have proved
that (i) and (iv) hold. As a consequence, this yields the
uniqueness of $(d,\nu)$ in (iii), which completes the
proof.
   \end{proof}
Under some additional constraints, CPD sequences
$\gammab$ for which the sequence
$\triangle\gammab$ is bounded from above can be
characterized as follows.
   \begin{thm} \label{boundiff-scalar2}
Let $\gammab=\{\gamma_n\}_{n=0}^{\infty}$ be a sequence of
real numbers such that
   \begin{align} \label{zal-scalar3}
\inf_{n\in\zbb_+} \gamma_n > -\infty.
   \end{align}
Then the following statements are
equivalent\/\footnote{\label{stop-a2} Applying
Lemma~\ref{backest} to $k=1$ and $j=0$, we verify that
\eqref{zal-scalar3} and \eqref{zal-scalar2} imply that
$\limsup_{n\to \infty}|\gamma_n|^{1/n} \Le 1$.}{\em :}
   \begin{enumerate}
   \item[(i)] $\gammab$ is CPD and
   \allowdisplaybreaks
   \begin{gather} \label{zal-scalar2}
\sup_{n\in \zbb_+} (\triangle\gammab)_n < \infty,
   \\ \label{zal-scalar4}
\supp{\nu} \subseteq \rbb_+,
 \end{gather}
where $\nu$ is the measure appearing in the representing
triplet of $\gammab$,
   \item[(ii)] there exist a finite Borel measure $\nu$ on $\rbb$
and $d\in \rbb$ such that
   \begin{enumerate}
   \item[(ii-a)] $\nu(\rbb\setminus [0,1))=0$,
   \item[(ii-b)] $\frac{1}{1-x}\in L^1(\nu)$,
   \item[(ii-c)] $\gamma_n = \gamma_0 +  n d   -  \int_{[0,1)}
\frac{1-x^n}{(1-x)^2} \D \nu(x)$ for all $n\in \zbb_+$.
   \end{enumerate}
   \end{enumerate}
Moreover, the following statements are satisfied{\em :}
   \begin{enumerate}
   \item[(iii)] if {\em (i)} holds and $(b,c,\nu)$
represents $\gammab$, then $c=0$, $\frac{1}{1-x}\in
L^1(\nu)$ and the pair $(d,\nu)$ with
$d=b+\int_{[0,1)} \frac{1}{1-x} \D \nu(x)$ is a unique
pair satisfying {\em (ii)},
   \item[(iv)] if {\em (ii)} holds, then
$d\Ge 0$, the sequence $\triangle\gammab$ is
monotonically increasing to $d$ and $(b,0,\nu)$
represents $\gammab$ with $b=d-\int_{[0,1)}
\frac{1}{1-x} \D \nu(x)$.
   \end{enumerate}
   \end{thm}
   \begin{proof}
We begin by proving the implication
(i)$\Rightarrow$(ii). Suppose (i) holds. By
footnote~\ref{stop-a2}, $\limsup_{n\to\infty}
|\gamma_n|^{1/n}\Le 1$. Hence by \eqref{limsup}
and \eqref{zal-scalar4}, \mbox{(ii-a)} holds.
Applying \eqref{del1} and \eqref{cdr4}, we
obtain
   \begin{align} \label{plagg}
(\triangle\gammab)_n = b + c(2n+1) + \int_{[0,1)}
\frac{1-x^n}{1-x} \D \nu(x), \quad n\in \zbb_+,
   \end{align}
where $(b,c,\nu)$ represents $\gammab$. By
Lebesgue's monotone convergence theorem, the
third term on the right-hand side of the
equality in \eqref{plagg} is monotonically
increasing to $\int_{[0,1)} \frac{1}{1-x} \D
\nu(x)$. Since $c\Ge 0$, we deduce from
\eqref{zal-scalar2} and \eqref{plagg} that
$c=0,$ $\frac{1}{1-x} \in L^1(\nu)$ (which
yields \mbox{(ii-b)}), the sequence
$\triangle\gammab$ is monotonically increasing
and convergent in $\rbb$ and
   \begin{align} \label{adycto8}
b=d-\int_{[0,1)} \frac{1}{1-x} \D \nu(x),
   \end{align}
where $d=\lim_{n\to \infty}(\triangle{\gammab})_n$.
Using \eqref{cdr4} and \eqref{adycto8} and arguing as
in \eqref{adycto3}, we deduce that \mbox{(ii-c)}
holds. Since $d=\sup_{n\in
\zbb_+}(\triangle{\gammab})_n$, the telescopic
argument (cf.\ \eqref{te-skop}) shows that
   \begin{align} \label{wyr-ic}
\gamma_n \Le \gamma_0 + n d, \quad n\in \nbb.
   \end{align}
Applying \eqref{zal-scalar3}, we conclude that $d\Ge
0$. This proves (ii) and (iii) except for the
uniqueness of $(d,\nu)$.

A close inspection of the proof of the
implication (ii)$\Rightarrow$(i) of
Theorem~\ref{boundiff-scalar} shows that (ii)
implies (i) and that $(d,\nu)$ in (iii) is
unique. By this uniqueness, statement (iv)
follows from the proof of the implication
(i)$\Rightarrow$(ii).
   \end{proof}
   \begin{cor} \label{pd2cpd}
Let $\gammab=\{\gamma_n\}_{n=0}^{\infty}$ be a
CPD sequence such that
   \begin{align*}
\limsup_{n\to \infty}|\gamma_n|^{1/n} < \infty
   \end{align*}
and let $(b,c,\nu)$ be the representing triplet
of $\gammab$. Suppose $\supp{\nu} \subseteq
\rbb_+$ and $\gamma_n \Ge 0$ for $n$ large
enough. Then the following conditions are
equivalent{\em :}
   \begin{enumerate}
   \item[(i)] $\lim_{n\to \infty} (\triangle \gammab)_n =
0$,
   \item[(ii)] the sequence $\gammab$ is monotonically decreasing,
   \item[(iii)] the sequence $\gammab$ is convergent in $\rbb$.
   \end{enumerate}
Moreover, if {\em (i)} holds, then $\gammab$ is
PD and $\gamma_n \Ge 0$ for all $n\in\zbb_+$.
   \end{cor}
   \begin{proof}
(i)$\Rightarrow$(ii) It follows from
Theorem~\ref{boundiff-scalar2} that $\nu(\rbb\setminus
[0,1))=0$, $b=\int_{\rbb} \frac{1}{x-1} \D \nu(x)$,
$c=0$ and
   \begin{align*}
\gamma_n = \gamma_0 - \int_{[0,1)} \frac{1-x^n}{(1-x)^2} \D
\nu(x), \quad n\in \zbb_+,
   \end{align*}
which yields (ii) and consequently implies that
$\gamma_n \Ge 0$ for all $n\in\zbb_+$.
Lebesgue's monotone convergence theorem gives
$\int_{\rbb} \frac{1}{(1-x)^2} \D \nu(x) \Le
\gamma_0$, so by Theorem~\ref{dyszcz3},
$\gammab$ is PD. This proves the ``moreover''
part.

The implications (ii)$\Rightarrow$(iii) and
(iii)$\Rightarrow$(i) are obvious.
   \end{proof}
   \section{\label{Sec3}Representations of conditionally positive definite operators}
   \subsection{\label{Sec3.1}Semispectral integral representations}
Recall that an operator $T\in \ogr{\hh}$ is said
to be {\em CPD} if the sequence $\{\|T^n
h\|^2\}_{n=0}^{\infty}$ is CPD for every $h\in
\hh$. Occasionally, we will use a concise
notation:
   \begin{align} \label{gth}
(\gammab_{T,h})_n:=\|T^nh\|^2, \quad n\in \zbb_+, \,
h\in \hh.
   \end{align}
The class of CPD operators contains the class of
complete hypercontractions of order $2$
introduced by Chavan and Sholapurkar in
\cite{Cha-Sh} (see the paragraph preceding
Proposition~\ref{traj-pd-op} for a more detailed
discussion). The main difference between these
two concepts is that the representing
semispectral measures of complete
hypercontractions of order $2$ are concentrated
on the closed interval $[0,1]$ (see
\cite[Theorem~4.11]{Cha-Sh}), while the
representing semispectral measures of CPD
operators can be concentrated on an arbitrary
finite subinterval of $\rbb_+$ (see
Theorem~\ref{cpdops}). Let us point out that CPD
operators are not scalable in general (see
Corollary~\ref{scalcpd}). We also refer the
reader to \cite{Ja02} for semispectral integral
representations and the corresponding dilations
for completely hypercontractive and completely
hyperexpansive operators (still on $[0,1]$). The
article \cite{Ja02} was an inspiration for the
research carried out in \cite{Cha-Sh,Cha-Sh17}.
It is also worth mentioning that in view of
\cite[Theorem~2]{At2}, an operator $T\in
\ogr{\hh}$ is completely hyperexpansive if and
only if the sequence $\{-\|T^n
h\|^2\}_{n=0}^{\infty}$ is CPD for every $h \in
\hh$.
   \begin{thm} \label{cpdops}
   Let $T\in \ogr{\hh}$. Then the following statements are
equivalent{\em :}
   \begin{enumerate}
   \item[(i)] $T$ is CPD,
   \item[(ii)] there exist operators $B,C\in
\ogr{\hh}$ and a compactly supported semispectral measure
$F\colon \borel{\rbb_+} \to \ogr{\hh}$ such that $B=B^*$,
$C\Ge 0$, $F(\{1\})=0$ and
   \begin{align} \label{cdr5}
T^{*n}T^n = I + n B + n^2 C + \int_{\rbb_+} Q_n(x) F(\D x),
\quad n\in \zbb_+.
   \end{align}
   \end{enumerate}
Moreover, if {\em (ii)} holds, then the triplet
$(B,C,F)$ is unique and
   \begin{gather} \label{fontan}
\supp{F} \subseteq [0,r(T)^2],
   \\ \label{fontan5}
C\neq 0 \implies r(T) \Ge 1,
   \\ \label{fontan6}
\sup \supp{F} \Ge 1 \implies r(T)^2 = \sup \supp{F}.
   \end{gather}
Furthermore, $(\is{Bh}h, \is{Ch}h,
\is{F(\cdot)h}h)$ is the representing triplet of
the CPD sequence $\{\|T^n
h\|^2\}_{n=0}^{\infty}$ for every $h\in \hh$.
   \end{thm}
   \begin{proof}
(i)$\Rightarrow$(ii) By Theorem~ \ref{cpd-expon}, for
every $h\in \hh$ there exists a unique triplet
$(b_h,c_h,\nu_h)$ consisting of a real number $b_h$,
nonnegative real number $c_h$ and a finite compactly
supported Borel measure $\nu_h$ on $\rbb$ such that
$\nu_h(\{1\})=0$ and
   \begin{align} \label{bar1}
(\gammab_{T,h})_n \overset{\eqref{gth}}= \|T^n h\|^2 =
\|h\|^2 + b_hn + c_h n^2 + \int_{\rbb} Q_n(x)
\D\nu_h(x), \quad n\in \zbb_+.
   \end{align}
First we show that
   \begin{align} \label{buff2}
\supp{\nu_h} \subseteq \rbb_+, \quad h\in\hh.
   \end{align}
For this, note that by \eqref{bar1} and \eqref{rca1}
we have
   \begin{align} \label{buf1}
(\triangle^2 \gammab_{T,h})_n = \int_{\rbb} x^n \D
(\nu_h+2c_h\delta_1)(x), \quad n\in \zbb_+, \, \quad
h\in \hh.
   \end{align}
It is a simple matter to verify that the following identity
holds
   \begin{align} \label{buf2}
(\triangle^2 \gammab_{T,h})_{n+1} = (\triangle^2
\gammab_{T,Th})_n, \quad n\in\zbb_+, \, h \in \hh.
   \end{align}
It follows from \eqref{buf1} and \eqref{buf2}
that the sequences $\triangle^2 \gammab_{T,h}$
and $\{(\triangle^2
\gammab_{T,h})_{n+1}\}_{n=0}^{\infty}$ are PD.
Hence, by Theorem~\ref{Stiech}, $\triangle^2
\gammab_{T,h}$ is a Stieltjes moment sequence.
Since the measure $\nu_h+2c_h\delta_1$ is
compactly supported, we infer from \eqref{buf1}
and Lemma~\ref{csmad} that the Stieltjes moment
sequence $\triangle^2 \gammab_{T,h}$ is
determinate (as a Hamburger moment sequence).
Therefore, $\supp{\nu_h+2c_h\delta_1} \subseteq
\rbb_+$ for every $h\in\hh$, which implies
\eqref{buff2}.

Define the functions $\hat b, \hat c\colon \hh \times
\hh \to \cbb$ and $\hat \nu\colon \borel{\rbb_+}
\times \hh \times \hh \to \cbb$ by
   \begin{align*}
\hat b(f,g) = \frac{1}{4}\sum_{k=0}^3 \I^kb_{f+\I^kg},
\; \hat c(f,g) = \frac{1}{4}\sum_{k=0}^3
\I^kc_{f+\I^kg}, \; \hat \nu(\varDelta;f,g) =
\frac{1}{4}\sum_{k=0}^3 \I^k \nu_{f+\I^kg}(\varDelta),
   \end{align*}
where $f,g \in \hh$ and $\varDelta \in
\borel{\rbb_+}$. Clearly, $\hat
\nu(\,\cdot\,;f,g)$ is a complex
measure for all $f,g\in \hh$. It
follows from \eqref{bar1},
\eqref{buff2} and the polarization
formula that
   \begin{align}  \notag
\is{T^n f}{T^n g}& = \is{f}{g} + \hat b(f,g) n + \hat
c(f,g) n^2
   \\ \label{tnf}
& \hspace{3ex} + \int_{\rbb_+} Q_n(x) \, \hat\nu(\D x;
f,g), \quad n\in \zbb_+, \, f,g \in \hh.
   \end{align}
Using \eqref{tnf} and Lemma~\ref{uniq}, one can verify
that $\hat b$ is a Hermitian symmetric sesquilinear
form and the functions $\hat c$ and $\hat
\nu(\varDelta;\cdot,\mbox{-})$, where $\varDelta \in
\borel{\rbb_+}$, are semi-inner products such that for
all $h\in \hh$ and $\varDelta \in \borel{\rbb_+}$,
   \begin{align} \label{Ber1}
\hat b(h,h) = b_h, \; \hat c(h,h) = c_h, \; \hat
\nu(\varDelta;h,h) = \nu_h(\varDelta).
   \end{align}
(cf.\ the proofs of \cite[Proposition~
1]{Sto0} and \cite[Theorem~
4.2]{Ja02}). By \eqref{bar1}, we have
   \begin{align*}
\hat\nu(\varDelta;h,h)+2\hat c(h,h) &
\overset{\eqref{Ber1}}\Le \nu_h(\rbb)+2c_h
   \\
& \hspace{.5ex}\overset{\eqref{buf1}}= (\triangle^2
\gammab_{T,h})_0
   \\
& \hspace{.5ex} \overset{\eqref{bmt}}=
\is{\bscr_2(T)h}{h}
   \\
& \hspace{6ex} \Le \|\bscr_2(T)\| \|h\|^2, \quad h \in
\hh, \, \varDelta \in \borel{\rbb_+}.
   \end{align*}
This implies that the sesquilinear forms $\hat c$ and
$\hat \nu(\varDelta;\cdot,\mbox{-})$, where $\varDelta
\in \borel{\rbb_+}$, are bounded. Hence, there exist
$C, F(\varDelta) \in \ogr{\hh}_+$, where $\varDelta
\in \borel{\rbb_+}$, such that
   \begin{gather} \label{cyk1}
\is{Ch}{h}=\hat c(h,h)\overset{\eqref{Ber1}}=c_h,
\quad h\in \hh,
   \\ \label{cyk2}
\is{F(\varDelta)h}{h}=\hat\nu(\varDelta;h,h)
\overset{\eqref{Ber1}} = \nu_h(\varDelta), \quad
\varDelta \in \borel{\rbb_+}, \, h\in \hh.
   \end{gather}
In view of \eqref{cyk2}, $F$ is a Borel semispectral
measure on $\rbb_+$.

Now we show that the so-constructed $F$ satisfies
\eqref{fontan} and \eqref{fontan6}. It follows from
Gelfand's formula for spectral radius that
   \begin{align} \label{koron-av}
\limsup_{n\to \infty}\|T^nh\|^{1/n} \Le r(T), \quad h\in
\hh.
   \end{align}
This together with \eqref{bar1}, \eqref{cyk2} and Theorem~
\ref{cpd-expon} applied to $\gammab_{T,h}$ yields
   \begin{multline*}
\Big\langle F\Big(\big(r(T)^2,\infty\big)\Big)h,h
\Big\rangle \Le \Big\langle F\Big(\big(\limsup_{n\to
\infty}\|T^nh\|^{2/n},\infty\big)\Big)h, h\Big\rangle
\overset{\eqref{limsup}}= 0, \quad h\in \hh,
   \end{multline*}
which, when combined with \eqref{buff2}, implies
\eqref{fontan}. Hence, we have
   \begin{align} \label{pasr2}
\sup \supp{F} \Le r(T)^2.
   \end{align}
Observing that
   \begin{align*}
\supp{\is{F(\cdot)h}{h}} \subseteq \supp{F}, \quad h\in
\hh,
   \end{align*}
we obtain
   \begin{align*}
\limsup_{n\to \infty} \|T^nh\|^{2/n} &
\overset{\eqref{limsup1.6}} \Le \max\bigg\{1,
\sup\supp{\is{F(\cdot)h}{h}}\bigg\}
   \\
& \hspace{2.1ex} \Le \max\big\{1, \sup\supp{F}\big\}, \quad
h\in \hh.
   \end{align*}
It follows from \cite[Corollary~3]{Dan} that $r(T)^2 \Le
\max\bigg\{1, \sup\supp{F}\bigg\}$, which together with
\eqref{pasr2} gives \eqref{fontan6}.

Our next goal is to construct the operator $B$. By
\eqref{bar1} and \eqref{cyk1}, we have
   \begin{align*}
\|Th\|^2-\|h\|^2=(\triangle \gammab_{T,h})_0
\overset{\eqref{Ber1}}= \hat b(h,h) + \is{Ch}{h},
\quad h\in \hh.
   \end{align*}
As a consequence, $\hat b$ is a bounded
Hermitian symmetric sesquilinear form. This
implies that there exists a selfadjoint operator
$B\in \ogr{\hh}$ such that
   \begin{align} \label{cyk3}
\is{Bh}{h} = \hat b(h,h) \overset{\eqref{Ber1}}= b_h,
\quad h\in \hh.
   \end{align}
Combining \eqref{bar1} with \eqref{cyk1}, \eqref{cyk2}
and \eqref{cyk3} gives (ii).

(ii)$\Rightarrow$(i) This implication is a direct
consequence of Theorem~ \ref{cpd-expon} applied to the
sequences $\gammab_{T,h}$, $h\in \hh$.

It remains to justify the ``moreover'' part.
Suppose (ii) holds. The uniqueness of the
triplet $(B,C,F)$ follows from Theorem~
\ref{cpd-expon}. Assertions \eqref{fontan} and
\eqref{fontan6} were proved above. To show
\eqref{fontan5}, assume that $C\neq 0$. Then the
set $U:= \{h\in \hh\colon \is{Ch}h > 0\}$ is
nonempty. By the ``moreover'' part of
Theorem~\ref{cpd-expon} and \eqref{koron-av}, we
have
   \begin{align*}
r(T) \Ge \limsup_{n\to\infty} \|T^n h\|^{1/n} \Ge 1,
\quad h \in U,
   \end{align*}
which implies \eqref{fontan5}. The last statement of
the theorem is easily seen to be true. This completes
the proof.
   \end{proof}
The following definition is an operator counterpart of
Definition~ \ref{deftryp}.
   \begin{dfn}
If $T\in\ogr{\hh}$ is a CPD operator and $B$,
$C$ and $F$ are as in statement (ii) of
Theorem~\ref{cpdops}, we call $(B,C,F)$ the {\em
representing triplet} of $T$, or we simply say
that $(B,C,F)$ {\em represents} $T$.
   \end{dfn}
   \begin{rem}
Note that if $(B,C,F)$ represents a CPD operator
$T$ on $\hh\neq \{0\}$ and $B\Ge 0$, then by
\eqref{cdr5}, $T^{*n}T^n \Ge I$ for every $n \in
\nbb$ which together with Gelfand's formula for
spectral radius yields $r(T)\Ge 1$.
   \hfill $\diamondsuit$
   \end{rem}
Proposition~\ref{traj-pd-op} below which gives
characterizations of CPD operators is closely
related to Proposition~\ref{traj-pd} (see also
Theorem~\ref{dyl-an} for an alternative
approach). The most important fact we need in
its proof is that a sequence
$\{\gamma_n\}_{n=0}^{\infty} \subseteq
\ogr{\hh}$ of exponential growth is a Hamburger
moment sequence (that is, \eqref{hamb} holds for
some semispectral measure $\mu\colon
\borel{\rbb} \to \ogr{\hh}$) if and only if
$\{\is{\gamma_n h}{h}\}_{n=0}^{\infty}$ is a
Hamburger moment sequence for all $h\in \hh$
(see \cite[Theorem~2]{Bi94}). Similar assertions
are true for Stieltjes and Hausdorff operator
moment sequences. In view of \cite{Sz-N53},
operators $T\in \ogr{\hh}$ for which the
sequence
$\{T^{*n}\bscr_2(T)T^n\}_{n=0}^{\infty}$ is a
Hausdorff moment sequence coincide with complete
hypercontractions of order $2$ introduced in
\cite{Cha-Sh}. On the other hand, by
\cite[Corollary]{Emb73} (see also
Theorem~\ref{lamb}), an operator $T\in
\ogr{\hh}$ is subnormal if and only if the
sequence $\{T^{*n}T^n\}_{n=0}^{\infty}$ is a
Stieltjes moment sequence. We refer the reader
to \cite{Sz-N53,Bi94} for necessary definitions
and facts related to the aforesaid operator
moment problems.
   \begin{pro} \label{traj-pd-op}
For $T\in \ogr{\hh}$, the following conditions are
equivalent{\em :}
   \begin{enumerate}
   \item[(i)] $T$ is CPD,
   \item[(ii)]
$\{T^{*n}\bscr_2(T)T^n\}_{n=0}^{\infty}$ is PD,
   \item[(iii)] $\{T^{*n}\bscr_2(T)T^n\}_{n=0}^{\infty}$
is a Stieltjes moment sequence.
   \end{enumerate}
Moreover, if $T$ is CPD, then
$\{T^{*(n+1)}T^{n+1} -
T^{*n}T^n\}_{n=0}^{\infty}$ is monotonically
increasing and
   \begin{align*}
\inf_{n\in \zbb_+} (\|T^{n+1}h\|^2 - \|T^nh\|^2) =
-\is{\bscr_1(T) h}{h} = \is{Bh}{h} + \is{Ch}{h}, \quad
h \in \hh,
   \end{align*}
where $B$ and $C$ are as in Theorem~{\em
\ref{cpdops}(ii)}.
   \end{pro}
   \begin{proof}
Clearly, the sequence
$\{T^{*n}\bscr_2(T)T^n\}_{n=0}^{\infty}$ is of
exponential growth and
$(\triangle^2\gammab_{T,h})_n =
\is{T^{*n}\bscr_2(T)T^n h}{h}$ for all $n\in
\zbb_+$ and $h \in \hh$. This implies that
$\{T^{*n}\bscr_2(T)T^n\}_{n=0}^{\infty}$ is PD
(resp., a Stieltjes moment sequence) if and only
if $\triangle^2\gammab_{T,h}$ is PD (resp., a
Stieltjes moment sequence) for every $h\in \hh$.
Applying Proposition~\ref{traj-pd} to the
sequences $\gammab_{T,h}$ and using
Theorem~\ref{cpdops}, we deduce that conditions
(i)-(iii) are equivalent. The ``moreover'' part
is a direct consequence of the corresponding
part of Proposition~\ref{traj-pd},
Theorem~\ref{cpdops} and \eqref{cdr5} applied to
$n=1$.
   \end{proof}
Theorem~\ref{boundiff} below can be thought of
as an operator counterpart of
Theorem~\ref{boundiff-scalar2}. Before stating
it, we will discuss the role played by condition
\eqref{zal}, which is an operator counterpart of
\eqref{zal-scalar2}.
   \begin{pro} \label{unif-bund}
Let $T\in \ogr{\hh}$ be a CPD operator. Then the
following conditions are equivalent{\em :}
   \begin{enumerate}
   \item[(i)] the sequence $\{T^{*(n+1)}T^{n+1} -
T^{*n}T^n\}_{n=0}^{\infty}$ is convergent in
{\sc wot},
   \item[(ii)] $\sup_{n\in \zbb_+} \|T^{*(n+1)}T^{n+1} -
T^{*n}T^n\| < \infty$,
   \item[(iii)] $T$ satisfies the following estimate{\em
:}
   \begin{align} \label{zal}
\sup_{n\in \zbb_+} (\|T^{n+1}h\|^2 - \|T^nh\|^2)
< \infty, \quad h\in \hh.
   \end{align}
   \end{enumerate}
   \end{pro}
   \begin{proof}
(i)$\Rightarrow$(ii) Apply the uniform
boundedness principle.

(ii)$\Rightarrow$(iii) Obvious.

(iii)$\Rightarrow$(i) Set $D_n=T^{*(n+1)}T^{n+1}
- T^{*n}T^n$ for $n\in \zbb_+$. It follows from
the ``moreover'' part of
Proposition~\ref{traj-pd-op} that the sequence
$\{D_n\}_{n=0}^{\infty}$ is monotonically
increasing. Hence, by \eqref{zal} and the
polarization formula, the sequence $\{\is{D_n
f}{g}\}_{n=0}^{\infty}$ is convergent in $\cbb$
for all $f,g\in \hh$. Using the uniform
boundedness principle again and the Riesz
representation theorem, we deduce that (i) is
valid.
   \end{proof}
Regarding Proposition~\ref{unif-bund}, note that
if an operator $T\in \ogr{\hh}$ is CPD, then by
Proposition~\ref{traj-pd-op}, ``$\sup_{n\in
\zbb_+}$'' in \eqref{zal} can be replaced by
``$\lim_{n\to\infty}$'' (in the extended real
line). If $T$ is CPD and satisfies \eqref{zal},
then by Proposition~\ref{unif-bund}, the
sequence $\{T^{*(n+1)}T^{n+1} -
T^{*n}T^n\}_{n=0}^{\infty}$ is convergent in
{\sc wot}, say to $D\in \ogr{\hh}$. It is worth
mentioning that in view of
Remark~\ref{manyrem}a) and \cite[Proposition~
8]{At91} (see also \cite[Lemma~
6.1(ii)]{Ja-St}), there are CPD unilateral
weighted shifts $T$ such that
   \begin{align*}
\is{Dh}{h}=\sup_{n\in \zbb_+} (\|T^{n+1}h\|^2 -
\|T^nh\|^2) = - \is{\bscr_1(T)h}{h}> 0, \quad h\in \hh
\setminus \{0\}.
   \end{align*}
In turn, there are CPD operators $T$ for which
the only vector $h$ satisfying \eqref{zal} is
the zero vector (see Remark~\ref{manyrem}c)).
   \begin{thm} \label{boundiff}
Let $T\in\ogr{\hh}$. Then the following are equivalent{\em
:}
   \begin{enumerate}
   \item[(i)] $T$ is CPD and
satisfies \eqref{zal},
   \item[(ii)] there exist a semispectral measure $F\colon
\borel{\rbb_+} \to \ogr{\hh}$ and a selfadjoint
operator $D\in \ogr{\hh}$ such that
   \begin{enumerate}
   \item[(ii-a)] $F([1,\infty))=0$,
   \item[(ii-b)] $\frac{1}{1-x} \in L^1(F)$,
   \item[(ii-c)] $T^{*n}T^n = I + n D  -  \int_{[0,1)}
\frac{1-x^n}{(1-x)^2} F(\D x)$ for all $n\in \zbb_+$.
   \end{enumerate}
   \end{enumerate}
Moreover, the following statements are satisfied{\em
:}
   \begin{enumerate}
   \item[(iii)] if {\em (i)} holds and $(B,C,F)$
represents $T$, then $C=0$, $\frac{1}{1-x}\in L^1(F)$
and the pair $(D,F)$ with $D=B+\int_{[0,1)}
\frac{1}{1-x} F(\D x)$ is a unique pair satisfying
{\em (ii)},
   \item[(iv)] if {\em (ii)} holds,
then $D\Ge 0$, $\{T^{*(n+1)}T^{n+1} -
T^{*n}T^n\}_{n=0}^{\infty}$ converg\-es in {\sc
wot} to $D$ and $(B,0,F)$ represents $T$ with
$B=D-\int_{[0,1)} \frac{1}{1-x} F(\D x)$.
   \end{enumerate}
   \end{thm}
   \begin{proof}
Using Theorem~\ref{boundiff-scalar2} (together
with its ``moreover'' part) and Theorem~
\ref{cpdops} (together with its ``furthermore''
part), we deduce that statements (i) and (ii)
are equivalent and statements (iii) and (iv) are
valid.
   \end{proof}
The simple argument given below shows that
condition \eqref{zal} is strong enough to
guarantee that $r(T) \Le 1$.
   \begin{pro} \label{zen-jesz-nie}
If an operator $T\in\ogr{\hh}$ satisfies
condition {\em \eqref{zal}}, then
$\alpha_T(h):=\sup_{n\in \zbb_+} (\|T^{n+1}h\|^2
- \|T^nh\|^2) \Ge 0$ for all $h\in \hh$ and
$r(T)\Le 1$.
   \end{pro}
   \begin{proof}
Using the telescopic argument (cf.\ \eqref{wyr-ic})
yields
   \begin{align*}
\|T^nh\|^2 \Le \|h\|^2 + n\alpha_T(h), \quad n \in
\nbb, h \in \hh.
   \end{align*}
Hence, $\alpha_T(h) \Ge 0$ and $\limsup_{n\to\infty}
\|T^nh\|^{1/n} \Le 1$ for all $h\in \hh$. Applying
\cite[Corollary~3]{Dan}, we conclude that $r(T)\Le 1$.
   \end{proof}
We show below that if the operator $D$ in
Theorem~\ref{boundiff} is nonzero, then the spectral
radius of $T$ is equal to $1$. The case $D=0$ is
discussed in Theorem~\ref{glow-main}.
   \begin{thm}\label{rt=1}
Let $T\in \ogr{\hh}$ be a CPD operator
satisfying \eqref{zal}. Suppose that
$\frac{1}{(1-x)^2} \in L^1(F)$ and $D\neq 0$,
where $F$ is as in Theo\-rem~{\em
\ref{cpdops}(ii)} and $D:=\mbox{{\sc
(wot)}}\lim_{n\to \infty}(T^{*(n+1)}T^{n+1} -
T^{*n}T^n)$. Then~$r(T)=1$.
   \end{thm}
   \begin{proof}
By Proposition~\ref{unif-bund}, the limit
$\mbox{{\sc (wot)}}\lim_{n\to
\infty}(T^{*(n+1)}T^{n+1} - T^{*n}T^n)$ exists.
It follows from statements (iii) and (iv) of
Theorem~\ref{boundiff} that the pair $(D,F)$
satisfies condition (ii) of this theorem and
$D\Ge 0$. Hence $\is{Dh_0}{h_0}
> 0$ for some $h_0\in \hh$. In view of
Lebesgue's monotone convergence theorem, the
sequence $\{\int_{[0,1)}\frac{1-x^n}{(1-x)^2}
\is{F(\D x)h_0}{h_0}\}_{n=0}^{\infty}$ converges
to $\int_{[0,1)}\frac{1}{(1-x)^2} \is{F(\D
x)h_0}{h_0}$. Because the last integral is
finite, we infer from equality \mbox{(ii-c)} of
Theorem~\ref{boundiff} that there exists $n_0\in
\nbb$ such that
   \begin{align*}
\frac{\|T^n h_0\|^2}{n} = \frac{\|h_0\|^2 -
\int_{[0,1)} \frac{1-x^n}{(1-x)^2} \is{F(\D
x)h_0}{h_0}}{n} & + \is{Dh_0}{h_0}
   \\
& \Ge \frac{1}{2}\is{Dh_0}{h_0}, \quad n\Ge n_0.
   \end{align*}
Combined with Gelfand's formula for spectral radius,
this implies that
   \begin{align*}
r(T) \Ge \limsup_{n\to \infty} \frac{\|T^n
h_0\|^{1/n}}{n^{1/2n}} \Ge 1.
   \end{align*}
Therefore applying Proposition~\ref{zen-jesz-nie}
yields $r(T)=1$.
   \end{proof}
It follows from
Theorem~\ref{boundiff}\mbox{(ii-c)} that if
$T\in \ogr{\hh}$ is a CPD operator satisfying
\eqref{zal}, then there exists $\alpha \in
\rbb_+$ such that $\|T^n\| \Le \alpha \sqrt{n}$
for all $n\in \nbb$. The next proposition shows
that the powers of a CPD operator with spectral
radius less than or equal to~$1$ have polynomial
growth of degree at most $1$.
   \begin{pro}\label{wzrostkwadr}
Let $T\in \ogr{\hh}$ be a CPD operator with the
representing triplet $(B,C,F).$ Then the
following conditions are equivalent{\em :}
   \begin{enumerate}
   \item[(i)] there exists $\alpha\in \rbb_+$ such that
   \begin{align} \label{Vang01}
\|T^n\| \Le \alpha \cdot n, \quad n\in \nbb,
   \end{align}
   \item[(ii)] $r(T) \Le 1,$
   \item[(iii)] $\supp F \subseteq [0,1].$
   \end{enumerate}
Moreover, {\em (iii)} implies \eqref{Vang01}
with $\alpha=\sqrt{1+\|B\|+\|C\|+\|F([0,1])\|}$.
   \end{pro}
Proposition~\ref{wzrostkwadr} can be deduced
from Proposition~\ref{grown} and Gelfand's
formula for spectral radius. Its proof is
omitted. According to
Proposition~\ref{sub-mzero-n} and
Remark~\ref{manyrem}d) below, there are CPD
operators having exactly polynomial growth of
degree $1$. Observe that a subnormal operator of
polynomial growth of arbitrary degree, being
normaloid (see \eqref{subn-norm}) is a
contraction, that is, it has polynomial growth
of degree zero.
   \subsection{\label{Sec3.2}A dilation representation}
First, we adapt Agler's hereditary
functional calculus
\cite{Ag85,MP89,Cu-Pu93} to our needs.
For $T\in \ogr{\hh}$, we set
   \begin{align} \label{Il111}
p\lrangle{T} = \sum_{i\Ge 0} \alpha_i T^{*i}T^i\quad
\text{for } p=\sum_{i\Ge 0} \alpha_i X^i \in \cbb[X].
   \end{align}
In particular, we have (see \eqref{bmt})
   \begin{align} \label{nab-bla2}
\bscr_m(T) = (1-X)^m\lrangle{T}, \quad m\in \zbb_+.
   \end{align}
The map $\cbb[X] \ni p \mapsto p\lrangle{T} \in
\ogr{\hh}$ is linear but in general not multiplicative
(e.g., if $T\in \ogr{\hh}$ is a nilpotent operator
with index of nilpotency $2$ and $p=X-1$, then
$p\lrangle{T}^2 \neq (p^2)\lrangle{T}$). However, it
has the following property.
   \begin{align} \label{stand-in}
   \begin{minipage}{70ex}
{\em The map $p \mapsto p\lrangle{T}$ is a unique
linear map from $\cbb[X]$ to $\ogr{\hh}$ such that
$X^0\lrangle{T}=I$ and $(X p)\lrangle{T}
=T^*p\lrangle{T}T$ for all $p \in \cbb[X]$.}
   \end{minipage}
   \end{align}
There is another way of defining $p\lrangle{T}$.
Namely, let us consider the elementary operator
$\nabla_T \colon \ogr{\hh} \to \ogr{\hh}$
defined by
   \begin{align} \label{nubile}
\nabla_T(A)=T^*AT, \quad A\in \ogr{\hh}.
   \end{align}
It is then easily seen that $p\lrangle{T} =
p(\nabla_T)(I)$ for any $p\in \cbb[X]$ and by
\eqref{stand-in},
   \begin{align} \label{nab-bla}
p(\nabla_T)(q\lrangle{T}) = \left((pq)(\nabla_T)\right)(I)
= (pq)\lrangle{T}, \quad p,q\in \cbb[X].
   \end{align}
Although the map $\cbb[X] \ni p \mapsto
p\lrangle{T} \in \ogr{\hh}$ is not
multiplicative, it does have a property
that resembles multiplicativity.
   \begin{lem} \label{ide-fin}
Let $T\in \ogr{\hh}$ and $q_0 \in \cbb[X]$. Then the
set
   \begin{align*}
\idealo=\{q\in \cbb[X]\colon (q_0 q)\lrangle{T}=0\}
   \end{align*}
is a principal ideal in $\cbb[X]$, that is, $\idealo =
\{p w \colon p \in \cbb[X]\}$ for some $w\in \cbb[X]$.
Moreover, if $q_0 = q_0^*$, then $w$ can be chosen to
satisfy $w^*=w$.
   \end{lem}
   \begin{proof}
First note that $\idealo$ is an ideal. Indeed, if
$q\in \idealo$ and $p\in \cbb[X]$, then
   \begin{align*}
0=p(\nabla_T)\Big(\big(q_0 q\big)\lrangle{T}\Big)
\overset{\eqref{nab-bla}}= (q_0pq)\lrangle{T}.
   \end{align*}
By \cite[Theorem~III.3.9]{hun74}, $\idealo$ is a
principal ideal in $\cbb[X]$. That $w$ can be chosen
to satisfy $w^*=w$, follows from the fact that
$(p\lrangle{T})^*=(p^{*}\lrangle{T})$ for all $p\in
\cbb[X]$.
   \end{proof}
   \begin{rem} \label{m-plus-k}
It follows from \eqref{nab-bla2} and
Lemma~\ref{ide-fin} that if $T\in\ogr{\hh}$ is an
$m$-isometry, that is $(1-X)^m\lrangle{T} =
\bscr_m(T)=0$, then $((1-X)^m q)\lrangle{T} =0$ for
every $q\in \cbb[X]$; in particular,
$\bscr_k(T)=(1-X)^k\lrangle{T} = 0$ for all $k\Ge m$,
which means that $T$ is a $k$-isometry for all $k\Ge
m$ (see \cite[p.\ 389, line 6]{Ag-St1}).
   \hfill $\diamondsuit$
   \end{rem}
Given $a\in \cbb$, we define the linear transformation
$\mathfrak{D}_a \colon \cbb[X] \to \cbb[X]$ by
   \begin{align*}
\mathfrak{D}_a p &=\frac{p - p(a)}{X-a}, \quad p \in
\cbb[X].
   \end{align*}
Using the Taylor series expansion about the
point $a$, it is easily seen that the
transformation $\mathfrak{D}_a$ is well defined
and ($p^{\prime}(a)$ stands for the derivative
of $p$ at $a$)
   \begin{align} \label{gim-11}
\mathfrak{D}_a^2 p&=\frac{p - p(a) - p^{\prime}(a)
(X-a)}{(X-a)^2}, \quad a\in \cbb, \, p\in \cbb[X].
   \end{align}
It is a simple matter to verify that for each $p\in
\cbb[X]$, $\mathfrak{D}_a^n p=0$ whenever $n > \deg
p$.

The following lemma will be used in some proofs of
subsequent results.
   \begin{lem} \label{minim-2}
The following assertions hold{\em :}
   \begin{enumerate}
   \item[(a)] if  $\kk$ is a
Hilbert space, $S\in \ogr{\kk}_+$, $E$ is the spectral
measure of $S$ and $\mscr$ is a vector subspace of
$\kk$, then
   \begin{align} \label{bigubigu}
\bigvee \Big\{E(\varDelta)\mscr \colon \varDelta\in
\borel{\rbb_+}\Big\} = \bigvee \{S^n\mscr \colon n\in
\zbb_+\},
   \end{align}
   \item[(b)] if for $i=1,2$,
$(\kk_i,R_i,S_i)$ consists of a
Hilbert space $\kk_i$ and operators $R_i
\in \ogr{\hh,\kk_i}$ and $S_i\in
\ogr{\kk_i}_+$ such that $\kk_i=\bigvee
\{S_i^n\ob{R_i}\colon n\in \zbb_+\}$ and
$R_1^*S_1^n R_1 = R_2^*S_2^n R_2$ for all
$n\in\zbb_+$, then there exists a
$($unique$)$ unitary isomorphism $U\in
\ogr{\kk_1,\kk_2}$ such that $UR_1=R_2$
and $US_1=S_2U${\em ;} in particular,
$\sigma(S_1) = \sigma(S_2)$ and
$\|S_1\|=\|S_2\|$.
   \end{enumerate}
   \end{lem}
   \begin{proof}
(a) Since $S$ is bounded, $E(\rbb_+\setminus
[0,r])=0$, where $r:=\|S\|$. Take a vector $g\in \kk$.
Then $g$ is orthogonal to the right-hand side of
\eqref{bigubigu} if and only if
   \begin{align} \label{sentak}
0= \is{S^nh}{g} = \int_{[0,r]} x^n \is{E(\D x)h}{g},
\quad n \in \zbb_+, \, h \in \mscr.
   \end{align}
It follows from the Weierstrass approximation theorem
and the uniqueness part of the Riesz representation
theorem (see \cite[Theorem~6.19]{Rud87}) that
\eqref{sentak} holds if and only if
$\is{E(\varDelta)h}{g}=0$ for all $\varDelta \in
\borel{\rbb_+}$ and $h\in \mscr$, or equivalently if
and only if $g$ is orthogonal to the left-hand side of
\eqref{bigubigu}. This implies \eqref{bigubigu}.

(b) It is easily seen that there exists a unique unitary
isomorphism $U\in \ogr{\kk_1,\kk_2}$ such that $US_1^nR_1 h
= S_2^nR_2 h$ for all $h \in \hh$ and $n\in \zbb_+$. It is
a matter of routine to verify that $U$ has the desired
properties. This completes the proof.
   \end{proof}
For the reader's convenience, we recall a version of the
Naimark dilation theorem needed in this paper.
   \begin{thm}[\mbox{\cite[Theorem~6.4]{Ml78}}] \label{Naim-ark}
If $M\colon\borel{\rbb_+} \to \ogr{\hh}$ is a
semispectral measure, then there exist a Hilbert space
$\kk$, an operator $R\in \ogr{\hh,\kk}$ and a spectral
measure $E\colon \borel{\rbb_+} \to \ogr{\kk}$ such
that
   \allowdisplaybreaks
   \begin{gather} \label{Nai1}
M(\varDelta) = R^*E(\varDelta) R, \quad \varDelta \in
\borel{\rbb_+},
   \\  \label{Nai2}
\kk = \bigvee
\big\{E(\varDelta)\ob{R}\colon
\varDelta\in \borel{\rbb_+}\big\}.
   \end{gather}
   \end{thm}
We are now ready to give a dilation
representation for CPD operators and relate
their spectral radii to the norms of positive
operators appearing in this representation.
Dilation representations for complete
hypercontractions and complete hyperexpansions
were given in \cite{Ja02} and afterwards
generalized to the case of complete
hypercontractions of finite order in
\cite{Cha-Sh}. All aforesaid representations
were built over the closed interval $[0,1]$.
What is more, the dilation representation for
complete hypercontractions of order $2$ (which
are very particular instances of CPD operators)
was proved under a restrictive assumption on the
representing semispectral measure (see
\cite[Theorem~4.20]{Cha-Sh}). Below we use the
convention \eqref{konw-1}.
   \begin{thm} \label{dyltyprep}
Let $T\in \ogr{\hh}$. Then the following conditions are
equivalent{\em :}
   \begin{enumerate}
   \item[(i)] $T$ is CPD,
   \item[(ii)] there exists a
semispectral measure $M\colon \borel{\rbb_+} \to \ogr{\hh}$
with compact support such that
   \begin{align} \label{checpt-2}
p\lrangle{T} = p(1) I - p^{\prime}(1) \bscr_1(T) +
\int_{\rbb_+} (\mathfrak{D}_1^2 p)(x) M(\D x), \quad
p\in \cbb[X],
   \end{align}
   \item[(iii)] there exist a Hilbert space
$\kk$, $R\in \ogr{\hh,\kk}$ and $S\in \ogr{\kk}_+$
such that
   \begin{gather} \label{checpt-3}
p\lrangle{T} = p(1) I - p^{\prime}(1) \bscr_1(T) + R^*
(\mathfrak{D}_1^2 p)(S) R, \quad p\in \cbb[X],
   \end{gather}
   \item[(iv)] there exist a Hilbert space
$\kk$, $R\in \ogr{\hh,\kk}$ and $S\in \ogr{\kk}_+$
such that {\em \eqref{checpt-3}} holds and
   \begin{gather} \label{mini-tr}
\kk =\bigvee \{S^n\ob{R}\colon n\in
\zbb_+\}.
   \end{gather}
   \end{enumerate}
Moreover, if any of conditions {\em (i)-(iv)}
holds, then
   \begin{enumerate}
   \item[(a)] the semispectral measure $M$ in {\em (ii)} is
unique,
   \item[(b)] if $(B,C,F)$ represents
$T$, then $B+C=-\bscr_1(T)$, $C=\frac{1}{2} M(\{1\})$ and
   \begin{align} \label{f2m-semi}
F(\varDelta)=(1-\chi_{\varDelta}(1))M(\varDelta), \quad
\varDelta \in \borel{\rbb_+},
   \end{align}
   \item[(c)] if $(\kk,R,S)$ is as in {\em (iv)}, then
   \begin{align} \label{wid-sup}
\text{$\sigma(S)=\supp{M}$ and
$\|S\|=\max\big\{0,\sup{\supp{M}}\big\} \Le r(T)^2$.}
   \end{align}
   \end{enumerate}
   \end{thm}
   \begin{proof}
(i)$\Rightarrow$(ii) By Theorem~\ref{cpdops}, $T$ has a
representing triplet $(B,C,F)$. Define the compactly
supported semispectral measure $M\colon \borel{\rbb_+} \to
\ogr{\hh}$ by
   \begin{align*}
M(\varDelta) = F(\varDelta) + 2\chi_{\varDelta}(1) C,
\quad \varDelta\in \borel{\rbb_+}.
   \end{align*}
Using \eqref{cdr5} and the fact that
$Q_n(1)=\frac{n(n-1)}{2}$ (see \eqref{klaud}), we deduce
that
   \begin{align} \label{anielgl-1}
T^{*n}T^n = I + n (B+C) + \int_{\rbb_+} Q_n(x) M(\D x),
\quad n\in \zbb_+.
   \end{align}
Since $Q_1=0$, substituting $n=1$ into
\eqref{anielgl-1} yields $B+C=-\bscr_1(T)$. Hence
   \begin{align} \label{anielgl}
T^{*n}T^n = I - n \bscr_1(T) + \int_{\rbb_+} Q_n(x) M(\D
x), \quad n\in \zbb_+.
   \end{align}
Suppose $p\in \cbb[X]$ is of the form $p=\sum_{n\Ge 0}
\alpha_n X^n$, where $\alpha_n\in \cbb$. Multiplying
\eqref{anielgl} by $\alpha_n$ and summing with respect
to $n$, gives
   \begin{align} \label{chpt}
p\lrangle{T} \overset{\eqref{Il111}}= p(1) I -
p^{\prime}(1) \bscr_1(T) + \int_{\rbb_+} \sum_{n\Ge 0}
\alpha_n Q_n(x) M(\D x).
   \end{align}
Notice that
   \begin{align}  \notag
\sum_{n\Ge 0} \alpha_n Q_n(x) &
\overset{\eqref{rnx-1}}= \sum_{n\Ge 0} \alpha_n
\frac{x^n-1 - n (x-1)}{(x-1)^2}
   \\ \label{chpt2}
& \overset{\eqref{gim-11}} = (\mathfrak{D}_1^2 p) (x),
\quad x\in \rbb \setminus \{1\}.
   \end{align}
Combining \eqref{chpt} with \eqref{chpt2} gives
\eqref{checpt-2}.

(ii)$\Rightarrow$(i) Substituting the polynomial
$p=X^n$ into \eqref{checpt-2} and using \eqref{rnx-1},
\eqref{gim-11}, \eqref{Il111} and the fact that
$Q_n(1)=\frac{n(n-1)}{2}$, we get
   \allowdisplaybreaks
   \begin{align} \notag
T^{*n}T^n &= I - n \bscr_1(T) + \int_{\rbb_+} Q_n(x) M(\D
x)
   \\ \notag
&= I - n \bscr_1(T) + Q_n(1) M(\{1\}) +
\int_{\rbb_+} Q_n(x) F(\D x)
   \\ \notag
&= I - n\bigg(\bscr_1(T) + \frac{1}{2} M(\{1\})\bigg)
   \\  \label{uniq-rep}
& \hspace{17ex} + \frac{n^2}{2} M(\{1\}) + \int_{\rbb_+}
Q_n(x) F(\D x), \quad n \in \zbb_+,
   \end{align}
where $F\colon \borel{\rbb_+} \to \ogr{\hh}$ is
the compactly supported semispectral measure
given by \eqref{f2m-semi}. Hence, by
Theorem~\ref{cpdops}, $T$ is CPD. What is more,
using \eqref{f2m-semi}, \eqref{uniq-rep} and the
uniqueness of representing triplets, we easily
verify that (a) and (b) hold.

(ii)$\Rightarrow$(iv) By Theorem~\ref{Naim-ark}, there
exists a triplet $(\kk,R,E)$ satisfying \eqref{Nai1}
and \eqref{Nai2}. Notice that the measure $E$ is
compactly supported. This is a direct consequence of
the identity $\supp{M} = \supp{E}$, which follows from
\eqref{Nai1} and \eqref{Nai2} (see the proof of
\cite[Theorem~4.4]{Ja02}). Set $S=\int_{\rbb_+} x E(\D
x)$. Since $E$ is compactly supported in $\rbb_+$, the
operator $S$ is bounded and positive (see
\cite[Theorem~5.9]{Sch12}). Applying the Stone-von
Neumann functional calculus (cf.\ \cite{B-S87,Sch12}
and \cite{Sto3}), we deduce from \eqref{checpt-2} that
the triplet $(\kk,R,S)$ satisfies \eqref{checpt-3}.
Using \eqref{Nai2} and Lemma~\ref{minim-2}(a), we get
\eqref{mini-tr}, which yields (iv).

(iv)$\Rightarrow$(iii) This is obvious.

(iii)$\Rightarrow$(ii) Applying the Stone-von Neumann
functional calculus to \eqref{checpt-3} yields (ii) with
the semispectral measure $M$ defined by \eqref{Nai1}, where
$E$ is the spectral measure of $S$.

It remains to prove (c). Suppose that $(\kk,R,S)$ is
as in (iv). If $\kk = \{0\},$ then by
\eqref{checpt-2}, \eqref{checpt-3} and (a), we deduce
that $M=0$ which gives \eqref{wid-sup}. Therefore we
can assume that $\kk \neq \{0\}.$ According to the
proof of the implication (iv)$\Rightarrow$(ii), we see
that the semispectral measure $M$ is given by
\eqref{Nai1}, where $E$ is the spectral measure of
$S$. In view of \eqref{mini-tr} and
Lemma~\ref{minim-2}(a), $(\kk,R,E)$ satisfies
\eqref{Nai2}, and so $M\neq 0$. As mentioned above,
$\supp{M} = \supp{E}$. Combined with
\cite[Theorem~5.9, Proposition~5.10]{Sch12}, this
implies~that
   \begin{align*}
\text{$\sigma(S)=\supp{M}$ and
$\|S\|=\sup{\sigma(S)}=\sup{\supp{M}}$.}
   \end{align*}
Thus it suffices to show that $\sup{\supp{M}} \Le r(T)^2$.
Let $(B,C,F)$ be the representing triplet of $T$. Set
$\vartheta=\sup\supp{F}$. Since $\supp{F}$ is compact, we
see that $\vartheta \in \{-\infty\} \cup \rbb_+$. We now
consider three possible cases that are logically disjoint.
The possibility that $\vartheta=-\infty$ (equivalently,
$\supp{F}=\emptyset$) may happen only in Case 2 (see
Corollary~\ref{nofs-sup2} and Remark~\ref{manyrem}d)).

{\sc Case 1.} $\vartheta \Ge 1$.

Then the following equalities hold
   \begin{align*}
r(T)^2 \overset{\eqref{fontan6}} = \sup \supp{F}
\overset{\mathrm{(b)}}= \sup \supp{M}.
   \end{align*}

{\sc Case 2.} $\vartheta < 1$ and $C\neq 0$.

According to \eqref{fontan5}, $r(T) \Ge 1$. Since
$\vartheta < 1$, we obtain
   \begin{align*}
\sup\supp{M} \overset{\mathrm{(b)}} = 1 \Le r(T)^2.
   \end{align*}

{\sc Case 3.} $\vartheta < 1$ and $C = 0$.

First observe that by (b), $\vartheta=\sup{\supp{M}}$.
It follows from \eqref{fontan} that $\vartheta \Le
r(T)^2$, hence $\sup{\supp{M}} \Le r(T)^2$. This
completes the proof.
   \end{proof}
   \begin{cor}
Suppose $T\in \ogr{\hh}$ is a CPD operator and
$(\kk,R,S)$ is as in Theorem~{\em
\ref{dyltyprep}(iv)}. If $1$ is an accumulation
point of $\sigma(S) \cap (0,1)$ or if $\sigma(S)
\cap (1,\infty)\neq \emptyset$, then
$r(T)^2=\|S\|$.
   \end{cor}
   \begin{proof}
Let $(B,C,F)$ be the representing triplet
of $T$ and let $M$ be as in
Theorem~\ref{dyltyprep}(ii). Suppose
first that $1$ is an accumulation point
of $\sigma(S) \cap (0,1)$. It follows
from the first equality in
\eqref{wid-sup} and \eqref{f2m-semi} that
$\supp{M}=\supp{F}$ and $1\in \supp{F}$,
so by the second equality in
\eqref{wid-sup}, we have
   \begin{align*}
\|S\|=\sup \supp{F} \Ge 1.
   \end{align*}
In turn, if $\sigma(S) \cap
(1,\infty)\neq \emptyset$, then again by
the first equality in \eqref{wid-sup} and
\eqref{f2m-semi},
   \begin{align*}
1 < \sup\sigma(S) = \sup\supp{M} =
\sup\supp{F}.
   \end{align*}
In both cases, an application of
\eqref{fontan6} and $\sup\sigma(S) =
\|S\|$ yields $r(T)^2 = \|S\|$.
   \end{proof}
The next corollary enables as to determine the
mass of the measure $M$ at the point $0$
provided the CPD operator has the spectral
radius less than or equal to $1$.
   \begin{cor} \label{nofs-sup2}
Suppose $T\in \ogr{\hh}$ is a CPD operator and
$M$ is as in Theorem~{\em \ref{dyltyprep}(ii)}.
Then
   \begin{align} \label{stary-num}
\bscr_m(T) = \int_{\rbb_+} (1-x)^{m-2} M(\D x), \quad m\ge
2.
   \end{align}
In particular, the following assertions hold{\em :}
   \begin{enumerate}
   \item[(i)] $\bscr_{2k}(T)\Ge 0$ for all $k\in \zbb_+,$
   \item[(ii)]  if $r(T) \Le
1,$ then $\bscr_m(T)\Ge 0$ for all $m\in
\zbb_+\setminus\{1\}$,
   \item[(iii)] if $r(T) \Le
1,$ then the sequence $\{\bscr_m(T)\}_{m=2}^{\infty}$ is
monotonically decreasing and convergent to $M(\{0\})$ in
the strong operator topology.
   \end{enumerate}
   \end{cor}
   \begin{proof}
Fix an integer $m\Ge 2$ and set $p=(1-X)^m$.
Then by \eqref{gim-11} we have $\mathfrak{D}_1^2
p = (1-X)^{m-2}$. Applying \eqref{nab-bla2} and
Theorem~\ref{dyltyprep}(ii), we get
\eqref{stary-num}. Assertion (i) is immediate
from \eqref{stary-num}, while assertions (ii)
and (iii) can be deduced from \eqref{fontan},
\eqref{f2m-semi} and Lebesgue's dominated
convergence theorem.
   \end{proof}
Concluding this subsection, we make a few
remarks related to
Theorem~\ref{dyltyprep} and
Corollary~\ref{nofs-sup2}.
   \begin{rem}  \label{waz-rem}
a) Let us begin by discussing in more detail the
relationship between $r(T)$ and $\vartheta =
\sup\supp{F}$, where $T\in \ogr{\hh}$ is a CPD
operator and $(B,C,F)$ represents $T$. As in the
proof of Theorem~\ref{dyltyprep}(c), we consider
three cases. If $\vartheta \Ge 1$, then by
\eqref{fontan6}, $1\Le \vartheta = r(T)^2$. If
$\vartheta < 1$ and $C\neq 0$, then by
\eqref{fontan5}, $\vartheta < 1 \Le r(T)^2$.

Suppose now that $\vartheta < 1$ and $C = 0$. First,
we consider the subcase when $D:=B+\int_{\rbb_+}
\frac{1}{1-x} F(\D x)\neq 0$. Then, there exists
$h_0\in \hh$ such that $\eta(h_0):= \is{Dh_0}{h_0}
\neq 0$. According to \eqref{cdr5}, we have
   \begin{align} \label{monot-3}
\|T^n h_0\|^2 = n\bigg(\frac{\|h_0\|^2}{n} +
\is{Bh_0}{h_0} + \int_{\rbb_+} \frac{Q_n(x)}{n}
\is{F(\D x)h_0}{h_0}\bigg), \quad n\in \nbb.
   \end{align}
By assumption that $\vartheta < 1$, we infer from
\eqref{monot-1}, \eqref{monot-2} and Lebesgue's
monotone convergence theorem that
   \begin{align*}
\text{$\is{Bh_0}{h_0} + \int_{\rbb_+} \frac{Q_n(x)}{n}
\is{F(\D x)h_0}{h_0} \longrightarrow \eta(h_0)$ as
$n\to \infty$.}
   \end{align*}
Since $\eta(h_0) \neq 0$, we deduce from
\eqref{monot-3} that $\eta(h_0) > 0$ and so by
Gelfand's formula for spectral radius we obtain
   \begin{align*}
r(T)^2 \Ge \limsup_{n\to\infty} \|T^n h_0\|^{2/n} \Ge
1 > \vartheta .
   \end{align*}
It remains to consider the subcase when $D=0$. Then by
\eqref{subn-norm}, \eqref{fontan} and
Theorem~\ref{glow-main}(v), $T$ is subnormal and
$\vartheta \Le r(T)^2 = \|T\|^2 \Le 1$ (see
Example~\ref{prz-do-na} and Remark~ \ref{rem-t0-ex}
for the continuation of this discussion).

b) It follows from assertions (a) and (c) of
Theorem~\ref{dyltyprep} that $\sigma(S)$ does
not depend on a triplet $(\kk,R,S)$ satisfying
\eqref{checpt-3} and \eqref{mini-tr}. This fact
can be also deduced from
Lem\-ma~\ref{minim-2}(b) by applying
\eqref{checpt-3} and \eqref{gim-11} to the
polynomials $p=(X-1)^2X^n$, where $n\in \zbb_+$.

c) Concerning \eqref{wid-sup}, observe that if
$T\in \ogr{\hh}$ is a $2$-isometry and $\hh \neq
\{0\}$, then $T$ is CPD and $1=r(T)^2 > \|S\|=0$
(use Proposition~\ref{sub-mzero-n},
\eqref{wid-sup} and \cite[Lemma~1.21]{Ag-St1}).

d) Regarding Corollary~\ref{nofs-sup2} (see also
Theorem~\ref{cpd-q}), it is worth recalling a
result due to Agler saying that an operator
$T\in \ogr{\hh}$ is a subnormal contraction if
and only if $\bscr_m(T)\Ge 0$ for all $m\in
\zbb_+$ (see \cite[Theorem~3.1]{Ag85}). Recently
Gu has shown that $\bscr_m(T)\Ge 0$ implies
$\bscr_{m-1}(T)\Ge 0$ for all positive odd
integers $m$ (see \cite[Theorem~2.5]{Gu15}). It
turns out that there are non-subnormal CPD
operators $T$ with $r(T)=1,$ so by
Corollary~\ref{nofs-sup2}(ii) for such $T$'s,
$\bscr_m(T)\Ge 0$ if and only if $m\in
\zbb_+\setminus\{1\}$ (see e.g.,
Example~\ref{prz-do-na}; cf.\ also
\cite[Section~9]{C-J-J-S19}).
   \hfill$\diamondsuit$
   \end{rem}
   \subsection{\label{Sec3.3-n}A simplified representation with
applications}
   First, following \break
Proposition~\ref{traj-pd-op}, we simplify the
previous representations of CPD operators.
   \begin{thm} \label{dyl-an}
For $T\in \ogr{\hh}$, the following conditions are
equivalent{\em :}
   \begin{enumerate}
   \item[(i)] $T$ is CPD,
   \item[(ii)] there exists a
semispectral measure $M\colon \borel{\rbb_+} \to
\ogr{\hh}$ with compact support such that
   \begin{align}  \label{cziki-1}
((X-1)^2q)\lrangle{T} = \int_{\rbb_+} q(x) M(\D x),
\quad q\in \cbb[X],
   \end{align}
   \item[(ii$^\prime$)] there exist a Hilbert space
$\kk$, $R\in \ogr{\hh,\kk}$ and $S\in \ogr{\kk}_+$
such that
   \begin{align} \label{dur-a}
((X-1)^2q)\lrangle{T} = R^* q(S) R, \quad q\in
\cbb[X],
   \end{align}
   \item[(iii)] there exists a
semispectral measure $M\colon \borel{\rbb_+} \to
\ogr{\hh}$ with compact support such that
   \begin{align} \label{cziki-2}
T^{*n}\bscr_2(T)T^n = \int_{\rbb_+} x^n M(\D x), \quad
n\in \zbb_+,
   \end{align}
   \item[(iii$^\prime$)] there exist a Hilbert space
$\kk$, $R\in \ogr{\hh,\kk}$ and $S\in \ogr{\kk}_+$
such that
   \begin{align} \label{dur-c}
T^{*n}\bscr_2(T)T^n = R^* S^n R, \quad n\in \zbb_+.
   \end{align}
   \end{enumerate}
Moreover, the measures in {\em (ii)} and {\em (iii)}
are unique and coincide with that in Theorem~{\em
\ref{dyltyprep}(ii)}{\em ;} the triplets $(\kk,R,S)$
in {\em (ii$^{\prime}$)} and {\em (iii$^{\prime}$)}
can be chosen to satisfy \eqref{mini-tr}.
   \end{thm}
   \begin{proof}
(i)$\Rightarrow$(ii$^{\prime}$) Applying the
implication (i)$\Rightarrow$(iv) of
Theorem~\ref{dyltyprep} to the polynomial $p=(X-1)^2
q$ and using \eqref{gim-11}, we get a triplet $(\kk,
R, S)$ satisfying \eqref{dur-a} and \eqref{mini-tr}.

(ii$^{\prime}$)$\Rightarrow$(i) It follows from
\eqref{gim-11} that
   \begin{align*}
p = p(1) + p^{\prime}(1)(X-1) + (X-1)^2
\mathfrak{D}_1^2 p, \quad p\in \cbb[X].
   \end{align*}
Since the mapping $p\mapsto p\lrangle{T}$ is linear,
we obtain
      \allowdisplaybreaks
   \begin{align*}
p\lrangle{T} & = p(1) I - p^{\prime}(1)\bscr_1(T) +
\Big((X-1)^2 \mathfrak{D}_1^2 p\Big)\lrangle{T}
   \\
&\hspace{-1.7ex}\overset{\eqref{dur-a}} = p(1) I -
p^{\prime}(1)\bscr_1(T) + R^*(\mathfrak{D}_1^2 p)(S)R,
\quad p\in \cbb[X],
   \end{align*}
which means that \eqref{checpt-3} holds. Applying
Theorem~\ref{dyltyprep} gives (i).

(ii$^{\prime}$)$\Leftrightarrow$(iii$^{\prime}$) One
can easily check that these two conditions are
equivalent with the same triplet $(\kk,R,S)$ (use
\eqref{nab-bla2} and \eqref{nab-bla}). This together
with the first paragraph of this proof justifies the
second statement of the ``moreover'' part.

Arguing as in the proof of the equivalence
(ii)$\Leftrightarrow$(iv) of
Theorem~\ref{dyltyprep}, we deduce that the
equivalences
(ii)$\Leftrightarrow$(ii$^{\prime}$) and
(iii)$\Leftrightarrow$(iii$^{\prime}$) hold. The
first statement of the ``moreover'' part can be
inferred from Theorem~\ref{dyltyprep}(a) by
observing that conditions \eqref{checpt-2},
\eqref{cziki-1} and \eqref{cziki-2} are
equivalent (cf.\ the proof of the equivalence
(i)$\Leftrightarrow$(ii$^{\prime}$)). This
completes the proof.
   \end{proof}
Many classes of operators are closed under the
operation of taking powers. Among them are the
classes of normaloid, subnormal, $k$-isometric,
$k$-expansive, completely hyperexpansive and
alternatingly hyperexpansive operators (see
\cite[p.\ 99]{Fur}, \cite[Theorem~2.3]{Ja02} and
\cite[Theorem~2.3]{E-J-L06}). On the other hand,
the class of hyponormal operators does not share
this property (see \cite[Problem~209]{Hal}). As
the first application of Theorem~\ref{dyl-an},
we show that the class of CPD operators does
share this property. We also describe the
semispectral and the dilation representations
for powers of CPD operators.
   \begin{thm} \label{pow-dwa}
Suppose that $T\in \ogr{\hh}$ is a CPD operator
and $i\in \nbb \setminus \{1\}$. Then
   \begin{enumerate}
   \item[(i)] $T^i$ is CPD,
   \item[(ii)] if $M$ and $M_i$ are semispectral
measures that correspond respectively
to $T$ and $T^i$ via Theorem~{\em
\ref{dyltyprep}(ii)}, then
   \begin{align} \label{diszcz-2}
M_i (\varDelta) = \tilde M_i
(\psi_i^{-1} (\varDelta)), \quad
\varDelta \in \borel{\rbb_+},
   \end{align}
where $\tilde M_i \colon \borel{\rbb_+}
\to \ogr{\hh}$ is the semispectral
measure defined by
   \begin{align} \label{diszcz-1}
\tilde M_i(\varDelta) =
\int_{\varDelta} (1+ x + \ldots +
x^{i-1})^2 M(\D x), \quad \varDelta \in
\borel{\rbb_+},
   \end{align}
and $\psi_i\colon \rbb_+\to \rbb_+$ is
given by $\psi_i(x)=x^i$ for $x \in
\rbb_+$,
   \item[(iii)] the representing triplet
$(B_i,C_i,F_i)$ of $T^i$ can be described
by applying Theorem~{\em
\ref{dyltyprep}(b)} to $M_i$ in place of
$M$,
   \item[(iv)] if $(\kk,R,S)$ is as in Theorem~{\em
\ref{dyltyprep}(iv)}, then the triplet
$(\kk,R_i,S^i)$ with
   \begin{align} \label{zacyt-1}
R_i:=(I+S + \ldots + S^{i-1}) R,
   \end{align}
corresponds to $T^i$ via Theorem~{\em
\ref{dyltyprep}(iv)}.
   \end{enumerate}
   \end{thm}
   \begin{proof}
(i)\&(ii) First, it is easily seen that
   \begin{align}  \label{ad-dyc1}
\bscr_2(T^i) = (1-X^i)^2 \langle T \rangle.
   \end{align}
Let $M$ be as in Theorem~\ref{dyltyprep}(ii). By
the ``moreover'' part of Theorem~\ref{dyl-an},
$M$ satisfies \eqref{cziki-1}. Clearly, the set
functions $\tilde M_i$ and $M_i$ defined by
\eqref{diszcz-1} and \eqref{diszcz-2},
respectively, are semispectral measures that are
compactly supported. Applying \eqref{form-ua}
and the measure transport theorem, we get (for
the definition of $\nabla_T$,
see~\eqref{nubile})
   \allowdisplaybreaks
   \begin{align*}
(T^{i})^{*n} \bscr_2(T^i) (T^{i})^{n} &
\overset{\eqref{ad-dyc1}}=
(\nabla_T)^{in} \left((1-X^i)^2 \langle
T \rangle\right)
   \\
& \overset{\eqref{nab-bla}} = \left(X^{in}
(1-X^i)^2\right) \langle T \rangle
   \\
& \hspace{1.7ex}= \left(X^{in} (1+ X + \ldots +
X^{i-1})^2(1-X)^2\right) \langle T \rangle
   \\
& \overset{\eqref{cziki-1}} = \int_{\rbb_+} x^{in} (1+
x + \ldots + x^{i-1})^2 M(\D x)
   \\
& \hspace{1.7ex} = \int_{\rbb_+} \psi_i(x)^n \tilde
M_i(\D x)
   \\
& \hspace{1.7ex} = \int_{\rbb_+} t^n M_i(\D t), \quad
n\in \zbb_+.
   \end{align*}
Using Theorem~\ref{dyl-an}(iii) and the
``moreover'' part of this theorem, we
see that (i) and (ii) hold.

(iii) Obvious.

(iv) Let $(\kk,R,S)$ be as in
Theorem~\ref{dyltyprep}(iv). Denote by
$E_{S}$ and $E_{S^i}$ the spectral
measures of $S$ and $S^i$,
respectively. In view of
\cite[Theorem~6.6.4]{B-S87}, we~have
   \begin{align} \label{pow-dra}
E_{S^i}(\varDelta) = E_{S}(\psi_i^{-1}
(\varDelta)), \quad \varDelta \in
\borel{\rbb_+}.
   \end{align}
According to the proof of the
implication (iii)$\Rightarrow$(ii) of
Theorem~\ref{dyltyprep},
   \begin{align} \label{nai-fakt}
M(\varDelta) = R^*E_{S}(\varDelta) R,
\quad \varDelta \in \borel{\rbb_+}.
   \end{align}
It follows from \eqref{mini-tr} and
Lemma~\ref{minim-2}(a) that
   \begin{align} \label{nai-fakt-2}
\kk = \bigvee
\big\{E_{S}(\varDelta)\ob{R}\colon
\varDelta\in \borel{\rbb_+}\big\}.
   \end{align}
Define the function $\zeta_i\colon
\rbb_+ \to \rbb_+$ by $\zeta_i(x)=1+ x
+ \ldots + x^{i-1}$ for $x\in \rbb_+$.
Using \eqref{diszcz-2} and
\eqref{diszcz-1} and applying the
Stone-von Neumann functional calculus,
we~get
   \allowdisplaybreaks
   \begin{align} \notag
\is{M_i(\varDelta)h}{h} & =
\Big\langle\int_{\psi_i^{-1}
(\varDelta)} \zeta_i(x)^2 M(\D
x)h,h\Big\rangle
      \\ \notag
&\hspace{-1.75ex}\overset{\eqref{form-ua}}=
\int_{\psi_i^{-1} (\varDelta)}
\zeta_i(x)^2 \langle M(\D x)h,h \rangle
   \\ \notag
&\hspace{-2.2ex}\overset{\eqref{nai-fakt}}=
\int_{\psi_i^{-1} (\varDelta)}
\zeta_i(x)^2 \Big\langle R^*E_{S}(\D
x)Rh,h\Big\rangle
   \\ \notag
&= \Big\langle R^* \int_{\rbb_+}
\chi_{\psi_i^{-1} (\varDelta)}(x)
\zeta_i(x)^2 E_{S}(\D x)Rh,h\Big\rangle
   \\ \notag
&= \Big\langle R^* \int_{\rbb_+}
\zeta_i(x) E_{S}(\D x) E_{S}(\psi_i^{-1}
(\varDelta)) \int_{\rbb_+} \zeta_i(x)
E_{S}(\D x) Rh,h\Big\rangle
   \\ \notag
&\hspace{-1.75ex}\overset{\eqref{pow-dra}}=
\Big\langle R^*(I+S + \ldots + S^{i-1})
E_{S^i}(\varDelta) (I+S + \ldots +
S^{i-1})Rh,h\Big\rangle
   \\   \label{ach-cha}
&\hspace{-1.75ex}
\overset{\eqref{zacyt-1}}=\is{R_i^*
E_{S^i}(\varDelta)R_ih}{h}, \quad h\in
\hh, \quad \varDelta \in
\borel{\rbb_+}.
   \end{align}
Since the operator $I+S + \ldots +
S^{i-1}$ commutes with $E_{S}$ and is
invertible in $\ogr{\kk}$, we obtain
   \begin{multline*}
\bigvee
\big\{E_{S^i}(\varDelta)\ob{R_i}\colon
\varDelta\in \borel{\rbb_+}\big\}
   \\
\overset{ \eqref{zacyt-1} \&
\eqref{pow-dra}}= \bigvee \big\{(I+S +
\ldots + S^{i-1})E_{S}(\psi_i^{-1}
(\varDelta)) \ob{R}\colon \varDelta\in
\borel{\rbb_+}\big\}
   \\
= (I+S + \ldots + S^{i-1}) \bigvee
\big\{E_{S}(\psi_i^{-1} (\varDelta))
\ob{R}\colon \varDelta\in
\borel{\rbb_+}\big\}
      \\
= (I+S + \ldots + S^{i-1}) \bigvee
\big\{E_{S}(\varDelta) \ob{R}\colon
\varDelta\in \borel{\rbb_+}\big\}
\overset{\eqref{nai-fakt-2}}= \kk.
   \end{multline*}
Hence, by Lemma~\ref{minim-2}(a), $\bigvee
\{(S^i)^n \ob{R_i} \colon n\in \zbb_+\} = \kk$.
Using (ii) and \eqref{ach-cha} and applying the
Stone-von Neumann functional calculus to the
operator $S^i$, we verify that equalities
\eqref{checpt-3} and \eqref{mini-tr} hold with
$(T^i, R_i, S^i)$ in place of $(T, R, S)$. This
shows (iv) and completes the proof.
   \end{proof}
The following corollary extends the
formula~\eqref{stary-num} of
Corollary~\ref{nofs-sup2} to the case of powers
of CPD operators.
   \begin{cor}
Suppose $T\in \ogr{\hh}$ is a CPD operator and
$M$ is as in Theorem~{\em \ref{dyltyprep}(ii)}.
Then
   \begin{align} \label{dar-ek-0}
\bscr_m(T^i) = \int_{\rbb_+}
(1-x^i)^{m-2} (1+ x + \ldots +
x^{i-1})^2 M(\D x), \quad m\ge 2, \,
i\Ge 1.
   \end{align}
   \end{cor}
   \begin{proof}
In view of Corollary~\ref{nofs-sup2}, it
suffices to consider the case $i \Ge 2$. By
assertions (i) and (ii) of
Theorem~\ref{pow-dwa}, $T^i$ is CPD and the
semispectral measure $M_i$ corresponding to
$T^i$ via Theorem~\ref{dyltyprep}(ii) is given
by \eqref{diszcz-2} and \eqref{diszcz-1}. Using
\eqref{form-ua} and the measure transport
theorem, we obtain
   \begin{gather} \notag
\text{\em for any Borel function $f\colon \rbb_+\to
\cbb$, $f\in L^1(M_i) \iff f\circ \psi_i \in
L^1(\tilde M_i)$,}
   \\ \label{dar-ek}
\int_{\rbb_+} f \D M_i = \int_{\rbb_+} f (x^i) (1+ x +
\ldots + x^{i-1})^2 M(\D x), \quad f\in L^1(M_i).
   \end{gather}
Applying \eqref{stary-num} to $T^i$ and
\eqref{dar-ek} to $f(x)=(1-x)^{m-2}$,
we get \eqref{dar-ek-0}.
   \end{proof}
As the second application of
Theorem~\ref{dyl-an}, we give a characterization
of CPD operators of class $\gqb$ (a class of
operators having upper triangular $2 \times 2$
block matrix form) by using the Taylor spectrum
approach developed in \cite{C-J-J-S19}. We also
describe the semispectral and the dilation
representations for such operators. According to
Corollary~\ref{nofs-sup2}(i), $\bscr_{2k}(T)\Ge
0$ for all $k\in \zbb_+$ whenever $T$ is CPD. We
will show in Theorem~\ref{cpd-q} below that the
single inequality $\bscr_{2k}(T)\Ge 0$ with
$k\Ge 1$ completely characterizes CPD operators
of class $\gqb$. Following \cite{C-J-J-S19}, we
say that $T\in \ogr{\hh}$ is {\em of class
$\gqb$} if it has a block matrix form
   \begin{align*}
T = \begin{bmatrix} {V } & {E}\\ {0} & {Q}
   \end{bmatrix}
   \end{align*}
with respect to an orthogonal decomposition $\hh=\hh_1
\oplus \hh_2$, where $\hh_1$ and $\hh_2$ are nonzero
Hilbert spaces and $V\in \ogr{\hh_1},$ $E\in
\ogr{\hh_2,\hh_1}$ and $Q\in \ogr{\hh_2}$ satisfy
   \begin{align} \label{pacz-1}
V^*V=I, \; V^*E=0, \; QE^*E=E^*EQ \; \textrm{ and } \;
QQ^*Q=Q^*QQ.
   \end{align}
(In particular, by the square root
theorem $|Q|$ and $|E|$ commute.) If this
is the case, we write
$T =  \big[\begin{smallmatrix} V & E \\
0 & Q \end{smallmatrix}\big] \in
\gqbh$. The Taylor spectrum of a pair
$(T_1,T_2)$ of commuting operators
$T_1,T_2\in \ogr{\hh}$ is denoted by
$\sigma(T_1,T_2)$. It is worth pointing
out that in view of
\cite[Theorem~3.3]{C-J-J-S19} for any
nonempty compact subset $\varGamma$ of
$\rbb_+^2$ and any separable infinite
dimensional Hilbert space $\hh_2$,
there exist a nonzero Hilbert
space $\hh_1$ and $T =  \big[\begin{smallmatrix} V & E \\
0 & Q \end{smallmatrix}\big] \in \gqbh$
$($relative to $\hh=\hh_1\oplus \hh_2$$)$ such
that $\sigma(|Q|,|E|) = \varGamma$. This
important fact enables us to find the spectral
region for conditional positive definiteness of
operators of class $\gqb$ (see
Theorem~\ref{cpd-q} and Figure~\ref{fig1}). For
a more thorough discussion of these topics the
reader is referred to \cite{C-J-J-S19}. Before
stating Theorem~\ref{cpd-q}, we prove an
auxiliary lemma which is of some independent
interest.
   \begin{lem} \label{jdof2com}
Let $A,B \in \ogr{\hh}$ be two commuting
normal operators. Then
   \begin{align} \label{jadra-1}
\jd{AB} & = \jd{A} + \jd{B},
   \\  \label{obrazy-1}
\overline{\ob{AB}} & = \overline{\ob{A}}
\cap \overline{\ob{B}}.
   \end{align}
   \end{lem}
   \begin{proof}
Let $G\colon \borel{\cbb^2}\to \ogr{\hh}$
be the joint spectral measure of $(A,B)$
(see \cite[Theorem~5.21]{Sch12}). Since
$G(\{(0,0)\}) \Le G(\cbb \times \{0\})$
and thus $\ob{G(\{(0,0)\})} \subseteq
\ob{G(\cbb \times \{0\})}$, we obtain
   \begin{align} \label{sum-dod}
\ob{G(\cbb\times \{0\})}
=\ob{G(\{(0,0)\})} + \ob{G(\cbb\times
\{0\})}.
   \end{align}
Applying the Stone-von Neumann functional
calculus and \eqref{sum-dod} yields
   \allowdisplaybreaks
   \begin{align} \notag
\jd{AB} &= \mathscr N\Big(\int_{\cbb^2}
z_1z_2 \D G(z_1,z_2)\Big)
   \\ \notag
&= \jd{G(\{(z_1,z_2)\in \cbb^2 \colon
z_1z_2 \neq 0\})}
   \\ \notag
&= \ob{G(\{(z_1,z_2)\in \cbb^2 \colon
z_1z_2 = 0\})}
   \\ \notag
&= \ob{G(\{0\}\times \cbb_{*})} +
\ob{G(\cbb\times \{0\})}
   \\ \notag
&=\ob{G(\{0\}\times \cbb_{*})} +
\ob{G(\{(0,0)\})} + \ob{G(\cbb\times
\{0\})}
   \\ \label{sum-dod-2}
&= \ob{G(\{0\}\times \cbb)} +
\ob{G(\cbb\times \{0\})},
   \end{align}
where $\cbb_{*}:=\cbb\setminus \{0\}$.
Similarly,
   \allowdisplaybreaks
   \begin{align}
\label{sum-dod-3} \jd{A} = \mathscr
N\Big(\int_{\cbb^2} z_1 \D
G(z_1,z_2)\Big) = \ob{G(\{0\}\times
\cbb)},
   \\ \label{sum-dod-4}
\jd{B} = \mathscr N\Big(\int_{\cbb^2} z_2
\D G(z_1,z_2)\Big) = \ob{G(\cbb\times
\{0\})}.
   \end{align}
Combining \eqref{sum-dod-2} with
\eqref{sum-dod-3} and \eqref{sum-dod-4},
we get \eqref{jadra-1}. Finally, applying
\eqref{jadra-1} to the adjoints of $A$
and $B$ and taking orthocomplements gives
\eqref{obrazy-1}.
   \end{proof}
   \begin{thm} \label{cpd-q}
Suppose that $T =  \big[\begin{smallmatrix} V & E \\
0 & Q \end{smallmatrix}\big] \in \gqbh$.
Then the following conditions are
equivalent{\em :}
   \begin{enumerate}
   \item[(i)] $T$ is CPD,
   \item[(ii)] $\sigma(|Q|,|E|) \subseteq \{(s,t)\in
\rbb_+^2\colon s^2 + t^2 \Le 1\} \cup
([1,\infty) \times \rbb_+)$,
   \item[(iii)] $\bscr_{2k}(T) \Ge 0$ for every
$($equivalently, for some$)$ $k\in \nbb$.
   \end{enumerate}
Moreover, if $T$ is CPD, then the following
assertions~hold{\em :}
   \begin{enumerate}
   \item[(a)] $A:=(I-|Q|^2-|E|^2)(I-|Q|^2)\in
\ogr{\hh_2}_+$, the operators $Q$, $|Q|$, $|E|$
and $A$ commute and
   \begin{align}  \label{sq-sem-m}
M(\varDelta) = 0 \oplus \sqrt{A} \,
P_{|Q|^2}(\varDelta) \sqrt{A}, \quad
\varDelta \in \borel{\rbb_+},
   \end{align}
where $M$ is as in Theorem~{\em
\ref{dyltyprep}(ii)} and $P_{|Q|^2}$ is the
spectral measure of $|Q|^2$,
   \item[(b)] the representing triplet
$(B,C,F)$ of $T$ is described by
Theorem~{\em \ref{dyltyprep}(b)},
   \item[(c)] the triplet $(\kk,R,S)$ defined
below corresponds to $T$ via Theorem~{\em
\ref{dyltyprep}(iv)}{\em :}
   \begin{gather} \label{adjos-1}
\kk := \overline{\ob{A}} =
\overline{\ob{I-|Q|^2-|E|^2}} \cap
\overline{\ob{I-|Q|}},
   \\ \label{adjos-2}
R(h_1 \oplus h_2) := \sqrt{A}h_2, \quad
h_1\in \hh_1, \, h_2\in \hh_2,
   \\ \notag
S := \big(|Q|\big|_{\kk}\big)^2 \quad
(\text{$\kk$ reduces $|Q|$}).
   \end{gather}
   \end{enumerate}
   \end{thm}
   \begin{proof}
(i)$\Leftrightarrow$(ii) Using
\cite[Proposition~3.10 and
Lemma~9.1]{C-J-J-S19}, one can check that
   \begin{align} \label{aha-aha}
T^{*n}\bscr_2(T)T^n = \begin{bmatrix} 0 &
0 \\ 0 & Q^{*n} A Q^n
   \end{bmatrix}
, \quad n \in \zbb_+.
   \end{align}
By the square root theorem and \eqref{pacz-1},
the operators $Q$, $|Q|$, $|E|$ and $A$ commute.
Combined with \cite[(19)]{C-J-J-S19}, this
implies that $A=A^*$ and
   \begin{align}  \label{aha-aha-2}
Q^{*n} A Q^n = Q^{*n} Q^n A = |Q|^{2n} A
= \int_{\rbb_+^2} \tau_n \D G, \quad n\in
\zbb_+,
   \end{align}
where $G$ is the joint spectral measure
of $(|Q|, |E|)$ and $\tau_n\colon
\rbb_+^2 \to \rbb$ is given by
   \begin{align}  \label{aha-aha-3}
\tau_n(s,t) = (1-s^2-t^2)(1-s^2)s^{2n},
\quad s,t\in \rbb_+^2, n \in \zbb_+.
   \end{align}
It follows from Proposition~\ref{traj-pd-op},
\cite[Theorem~2]{Bi94}, \eqref{aha-aha} and
\eqref{aha-aha-2} that $T$ is CPD if and only if
$\{\int \tau_n(s,t) \is{G(\D s, \D
t)h}{h}\}_{n=0}^{\infty}$ is a Stieltjes moment
sequence for every $h\in \hh_2$. By
\cite[Theorem~2.1(i) \& Lemma~4.10]{C-J-J-S19},
the latter holds if and only if $\sigma(|Q|,|E|)
\subseteq \varXi$, where
   \begin{align*}
\varXi:=\left\{(s,t)\in \rbb_+^2\colon
\{\tau_n(s,t)\}_{n=0}^{\infty} \text{ is
a Stieltjes moment sequence}\right\}.
   \end{align*}
In view of \eqref{aha-aha-3}, it is
easily seen that
   \begin{align*}
\varXi & = \{(s,t)\in \rbb_+^2\colon
(1-s^2-t^2)(1-s^2) \Ge 0\}
   \\
& = \{(s,t)\in \rbb_+^2\colon s^2 + t^2
\Le 1\} \cup ([1,\infty) \times \rbb_+),
   \end{align*}
which shows that (i) and (ii) are
equivalent.

(ii)$\Leftrightarrow$(iii) This
equivalence is a direct consequence of
\cite[Theorem~9.2(i)]{C-J-J-S19}.

We now prove the ``moreover'' part.
Assume that (i) holds.

(a) Applying the spectral mapping theorem
(see e.g.,
\cite[Theorem~2.1]{C-J-J-S19}), we~get
   \begin{align*}
\sigma(A) \overset{\eqref{aha-aha-3}}=
\sigma(\tau_{0}(|Q|,|E|)) =
\tau_{0}(\sigma(|Q|,|E|)) \subseteq
\rbb_+,
   \end{align*}
which together with $A=A^*$ implies that $A\in
\ogr{\hh_2}_+$. Now, it is clear that the set
function $M\colon \borel{\rbb_+} \to \ogr{\hh}$
defined by \eqref{sq-sem-m} is a semispectral
measure with compact support. Recall that $A$
commutes with $|Q|$. Using this fact and
applying the Stone-von Neumann functional
calculus, we deduce from \eqref{aha-aha},
\eqref{aha-aha-2} and the square root theorem
that
   \allowdisplaybreaks
   \begin{align*}
\is{T^{*n}\bscr_2(T)T^n h}{h} & =
\is{Q^{*n} A Q^n h_2}{h_2}
   \\
& = \|(|Q|^2)^{n/2} \sqrt{A} h_2\|^2
   \\
& = \int_{\rbb_+} x^n
\is{\sqrt{A}P_{|Q|^2}(\D
x)\sqrt{A}h_2}{h_2}
   \\
& = \int_{\rbb_+} x^n \is{M(\D x)h}{h},
   \\
&
\hspace{-1.7ex}\overset{\eqref{form-ua}}=
\Big\langle \int_{\rbb_+} x^n M(\D x)h,
h\Big\rangle, \quad h=h_1 \oplus h_2 \in
\hh, \, n\in \zbb_+.
   \end{align*}
This shows that condition (iii) of
Theorem~\ref{dyl-an} holds. Applying the
``moreover'' part of this theorem completes the
proof of (a).

(b) Obvious.

(c) First, note that by \eqref{adjos-2},
   \begin{align} \label{Wik-wsp-1}
\overline{\ob{R}} =
\overline{\ob{\sqrt{A}}} =
\overline{\ob{A}}=\kk,
   \end{align}
so $R$ is well defined. Since $A$
commutes with $|Q|$, we see that $\kk$
reduces $|Q|$, and thus $S=
\big(|Q|\big|_{\kk}\big)^2\in
\ogr{\kk}_+$. Moreover, $\kk$ reduces
$P_{|Q|^2}$ to the spectral measure $P_S$
of $S$. Using \eqref{sq-sem-m},
\eqref{adjos-2} and \eqref{Wik-wsp-1}, we
easily obtain
   \begin{gather*}
M(\varDelta) = R^*P_S(\varDelta) R, \quad
\varDelta \in \borel{\rbb_+},
   \\
\kk = \bigvee
\big\{P_S(\varDelta)\ob{R}\colon
\varDelta\in \borel{\rbb_+}\big\}.
   \end{gather*}
Therefore, in view of the proof of the
implication (ii)$\Rightarrow$(iv) of
Theorem~\ref{dyltyprep}, the triplet $(\kk,R,S)$
satisfies condition (iv) of
Theorem~\ref{dyltyprep}.

It remains to prove the second equality
in \eqref{adjos-1}. To do this, first
observe that $\jd{I-|Q|^2}=\jd{I-|Q|}$,
which implies that
$\overline{\ob{I-|Q|^2}}=\overline{\ob{I-|Q|}}$.
From this and Lemma~\ref{jdof2com} it
follows easily that the second equality
in \eqref{adjos-1} holds. This completes
the proof.
   \end{proof}
The following is a direct consequence of
Theorem~\ref{cpd-q} and the proof of
\cite[Corollary~3.2]{C-J-J-S19}.
   \begin{cor} \label{qqg}
Let $\hh_1$ and $\hh_2$ be nonzero Hilbert
spaces such that
   \begin{align} \label{hajhad}
\dim \hh_1 \Ge \aleph_0 \text{ and } \dim \hh_1
\Ge \dim \hh_2,
   \end{align}
and let $X,Y\in \ogr{\hh_2}$ be commuting
positive selfadjoint operators such that
   \begin{align} \label{alres}
\sigma(X,Y) \subseteq \{(s,t)\in \rbb_+^2\colon
s^2 + t^2 \Le 1\} \cup ([1,\infty) \times
\rbb_+).
   \end{align}
Then there exist $V\in \ogr{\hh_1}$, $E\in
\ogr{\hh_2, \hh_1}$ and $Q\in \ogr{\hh_2}$ such
that $|Q|=X$, $|E|=Y$,
$\big[\begin{smallmatrix} V & E \\
0 & Q \end{smallmatrix}\big] \in \gqbh$ and $T:=
\big[\begin{smallmatrix} V & E \\
0 & Q \end{smallmatrix}\big]$ is a CPD operator
which satisfies assertion {\em (a)} of
Theorem~{\em \ref{cpd-q}}.
   \end{cor}
   \begin{rem} \label{imprym}
Assuming that \eqref{hajhad} holds, the above
results allow to describe the ranges of the
mappings $\varPsi_{\hh}$ and
$\widetilde\varPsi_{\hh}$ restricted to
operators of class $\gqb$ (for the definitions
of $\varPsi_{\hh}$ and
$\widetilde\varPsi_{\hh}$, see \eqref{psur} and
\eqref{vsur}). First, we consider the case of
the mapping $\widetilde\varPsi_{\hh}$.
Take any $T =  \big[\begin{smallmatrix} V & E \\
0 & Q \end{smallmatrix}\big] \in \gqbh$ which is
CPD with $M$ as in Theorem~\ref{dyltyprep}(ii)
and set $X=|Q|$ and $Y=|E|$. Then by
Theorem~\ref{cpd-q}, $X$ and $Y$ are commuting
positive selfadjoint operators, they satisfy
\eqref{alres} and
   \begin{align} \label{sq-sem-m}
M(\varDelta) = 0 \oplus \sqrt{A} \,
P_{X^2}(\varDelta) \sqrt{A}, \quad \varDelta \in
\borel{\rbb_+},
   \end{align}
where
   \begin{align} \label{sq-sem-pi}
A =(I-X^2-Y^2)(I-X^2).
   \end{align}
Conversely, if $X,Y \in \ogr{\hh_2}$ are
commuting positive selfadjoint operators which
satisfy \eqref{alres}, then by
Corollary~\ref{qqg},
there is $T=\big[\begin{smallmatrix} V & E \\
0 & Q \end{smallmatrix}\big] \in \gqbh$ which is
CPD and satisfies \eqref{sq-sem-m} and
\eqref{sq-sem-pi}, where $M$ is as in
Theorem~\ref{dyltyprep}(ii).

The case of the mapping $\varPsi_{\hh}$ can be
deduced from the above description of
$\widetilde\varPsi_{\hh}$,
Subsection~\ref{Subs.1.3} and
\cite[Proposition~3.10(ii)]{C-J-J-S19}.
   \hfill $\diamondsuit$
   \end{rem}
   \begin{cor}
Suppose that $T =  \big[\begin{smallmatrix} V & E \\
0 & Q \end{smallmatrix}\big] \in \gqbh$ is CPD
and $S$ is as in Theorem~{\em \ref{cpd-q}(c)}.
Then the following conditions are equivalent{\em
:}
   \begin{enumerate}
   \item[(i)] $S=|Q|^2$,
   \item[(ii)] $1 \notin \sigma_{\mathrm
p}(|Q|^2+|E|^2)$ and $1 \notin
\sigma_{\mathrm p}(|Q|)$.
   \end{enumerate}
   \end{cor}
   \begin{figure}
   \begin{tikzpicture}[scale=.37, transform shape]
\tikzset{vertex/.style =
{shape=circle,draw,minimum size=1em}}
\tikzset{edge/.style = {->,> = latex'}}

\node[] (1) at (-6.5, -6.5)
{$\scalebox{1.5}{(0, 0)}$};

\node[] (1) at (0.0, -6.5)
{$\scalebox{1.5}{(1, 0)}$};

\node[] (1) at (-6.9, 0.0)
{$\scalebox{1.5}{(0, 1)}$};

\draw[->, line width=0.5mm] (-6,-6) ->
(4,-6);

\draw[line width=0.5mm] (-6,-6) ->
(-6,0);

\draw[->, line width=0.05mm] (-6,0) ->
(-6,3);

\def\Radius{6}

\filldraw[fill opacity=0.7, line
width=0.5mm, fill=gray] (-6,-6) --
(0,-6) arc (0:90:6cm) -- cycle;
   \end{tikzpicture}
   \hspace{10ex}
   \begin{tikzpicture}[scale=.37, transform shape]
\tikzset{vertex/.style =
{shape=circle,draw,minimum size=1em}}
\tikzset{edge/.style = {->,> = latex'}}

\fill [gray] (0, -6) rectangle (4, 3);

\node[] (1) at (-6.5, -6.5)
{$\scalebox{1.6}{(0, 0)}$};

\node[] (1) at (0.3, -6.5)
{$\scalebox{1.6}{(1, 0)}$};

\node[] (1) at (-6.8, -0.1)
{$\scalebox{1.6}{(0, 1)}$};

\draw[->, line width=0.5mm] (-6,-6) ->
(4,-6);

\draw[line width=0.5mm] (0,-6) ->
(0,3);

\draw[->, line width=0.05mm] (-6,0) ->
(-6,3);

\draw[line width=0.5mm] (0,-6) ->
(0,3);

\def\Radius{6}

\filldraw[fill opacity=0.7, line
width=0.5mm, fill=gray] (-6,-6) --
(0,-6) arc (0:90:6cm) -- cycle;
   \end{tikzpicture}
\caption{Spectral regions for
subnormality (left) and conditional
positive definiteness (right) of
operators of class $\gqb.$}
   \label{fig1}
   \end{figure}
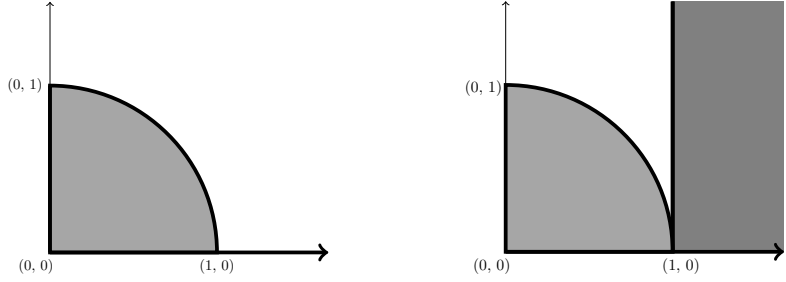
Regarding Theorem~\ref{cpd-q}, it is worth
mentioning that in view of
\cite[Theorem~1.2]{C-J-J-S19} the operator
$T =  \big[\begin{smallmatrix} V & E \\
0 & Q \end{smallmatrix}\big] \in \gqbh$
is subnormal if and only if
   \begin{align*}
\sigma(|Q|,|E|) \subseteq \{(s,t)\in
\rbb_+^2\colon s^2 + t^2 \Le 1\} \cup
([1,\infty) \times \{0\}).
   \end{align*}
For the reader's convenience, the
spectral regions for subnormality and
conditional positive definiteness of
operators of class $\gqb$ are
illustrated in Figure \ref{fig1}.
   \subsection{\label{Sec4.1}Subnormality} In view of Theorem~\ref{lamb},
any subnormal operator $T\in \ogr{\hh}$ has the
property that the sequence
$\{\|T^nh\|^2\}_{n=0}^{\infty}$ is PD for every
$h\in \hh$. As a consequence, any subnormal
operator is CPD. The converse implication is not
true in general (see \cite[Example~5.4]{Sto}).
In this subsection, we deal with the problem of
finding necessary and sufficient conditions for
subnormality written in terms of conditional
positive definiteness. Theorem~4.1 in
\cite{Sto}, which is the first result in this
direction formulated for $d$-tuples of
operators, shows that a contraction is subnormal
if and only if it is CPD. The main result of
this subsection, namely Theorem~\ref{glow-main},
generalizes \cite[Theorem~4.1]{Sto}. In
particular, it covers the case of strongly
stable operators (see Corollary~\ref{cpdalp}).

Our first goal is to characterize those CPD
operators that are subnormal in terms of the
parameters $B,C,F$ appearing in statement (ii)
of Theorem~\ref{cpdops}.
   \begin{thm} \label{subn-1}
Let $T\in \ogr{\hh}$. Then the following
statements are equivalent{\em :}
   \begin{enumerate}
   \item[(i)] $T$ is subnormal,
   \item[(ii)] $T$ is CPD and its representing triplet $(B,C,F)$
satisfies the following conditions{\em :}
   \begin{enumerate}
   \item[(ii-a)] $\frac{1}{(x-1)^2} \in L^1(F)$ and
$\int_{\rbb_+} \frac{1}{(x-1)^2} F(\D x) \Le I$,
   \item[(ii-b)] $\frac{1}{x-1} \in L^1(F)$ and
$B=\int_{\rbb_+} \frac{1}{x-1} F(\D x)$,
   \item[(ii-c)] $C=0$.
   \end{enumerate}
   \end{enumerate}
Moreover, if {\em (ii)} holds and $G$ is the
semispectral measure of $T$ {\em (see
\eqref{tobemom})}, then
   \allowdisplaybreaks
   \begin{align}   \notag
F& =M, \text{ where $M$ is as in Theorem~{\em
\ref{dyltyprep}(ii)},}
      \\ \notag
 B & = \int_{\rbb_+} (x-1) \, G\circ
 \phi^{-1}(\D x),
   \\   \label{ale-sp-a}
F(\varDelta) & = \int_{\varDelta} (x-1)^2
\,G\circ \phi^{-1}(\D x), \quad \varDelta \in
\borel{\rbb_+},
   \\  \notag
G\circ\phi^{-1}(\varDelta) & =\int_{\varDelta}
\frac{1}{(x-1)^2} F(\D x) + \delta_1(\varDelta)
\Big(I - \int_{\rbb_+} \frac{1}{(x-1)^2} F(\D
x)\Big), \quad \varDelta \in \borel{\rbb_+}.
   \end{align}
   \end{thm}
   \begin{proof}
(i)$\Rightarrow$(ii) It follows from Theorem~
\ref{lamb} that $\{\|T^n h\|^2\}_{n=0}^{\infty}$ is a
Stieltjes moment sequence for every $h\in \hh$. Hence,
by Theorem~\ref{dyszcz3}, we have
   \allowdisplaybreaks
   \begin{align} \label{ineq1}
&\int_{\rbb_+} \frac{1}{(x-1)^2} \is{F(\D x)h}{h} \Le
\|h\|^2, & \quad h\in \hh,
   \\ \label{ineq2}
&\is{Bh}{h}=\int_{\rbb_+} \frac{1}{x-1} \is{F(\D
x)h}{h}, & \quad h\in \hh,
   \\ \label{ineq3}
&\is{Ch}{h}=0, & \quad h\in \hh.
   \end{align}
It follows from \eqref{form-ua}, \eqref{ineq1}
and \eqref{ineq3} that conditions \mbox{(ii-a)}
and \mbox{(ii-c)} are satisfied. In turn,
\eqref{ineq1} yields $\frac{1}{x-1} \in L^1(F)$.
Combined with \eqref{ineq2}, this
implies~\mbox{(ii-b)}.

(ii)$\Rightarrow$(i) Applying \eqref{form-ua}
and Theorem~\ref{dyszcz3} again, we deduce that
the sequence $\{\|T^n h\|^2\}_{n=0}^{\infty}$ is
PD for all $h\in \hh$. Hence, by Theorem~
\ref{lamb}, $T$ is subnormal.

The ``moreover'' part can be deduced
straightforwardly from \eqref{tobemom}
and the corresponding part of
Theorem~\ref{dyszcz3} (that $F=M$
follows from (ii-c) and
Theorem~\ref{dyltyprep}(b)). This
completes the proof.
   \end{proof}
   \begin{cor}
Let $T\in \ogr{\hh}$ be a subnormal
operator, $G$ be the semispectral measure
of $T$, $N$ be the minimal normal
extension of $T$ and $F$ be as in
Theorem~{\em \ref{cpdops}(ii)}. Then
   \begin{enumerate}
   \item[(i)]
$r(T)=\|T\|=\sup\big\{|z|\colon z \in \supp{G}\big\}$,
   \item[(ii)] $\sigma(N)=\supp{G}$ and $\sigma(N^*N) =
\{|z|^2\colon z \in \supp{G}\}$,
   \item[(iii)] if $G(\tbb)=0$, where
$\tbb=\{z\in \cbb \colon |z|=1\}$, then
   \begin{enumerate}
   \item[(iii-a)] the measures  $F$ and $G\circ \phi^{-1}$ are mutually
absolutely continuous,
   \item[(iii-b)] $\sigma(N^*N)=\supp{F}$,
   \item[(iii-c)] $\|T\|^2 =
\sup\supp{F}$.
   \end{enumerate}
   \end{enumerate}
   \end{cor}
   \begin{proof}
The first equality in (i) is a
consequence of \eqref{subn-norm}. It
follows from \cite[Proposition~4]{Ju-St}
that
   \begin{align} \label{dwa-sup-ty}
\sigma(N)=\supp{G},
   \end{align}
which gives the first equality in (ii). Using
\cite[Corollary~II.2.17]{Con91}, we obtain
   \begin{align*}
\|T\|=\|N\| \overset{\eqref{subn-norm}} = r(N)
\overset{\eqref{dwa-sup-ty}}= \sup\{|z|\colon z \in
\supp{G}\}.
   \end{align*}
This yields the second equality in (i). The second
equality in (ii) follows from \eqref{dwa-sup-ty} and
\cite[eq.\ (14), p.\ 158]{B-S87}. It remains to prove
(iii). According to \eqref{ale-sp-a}, $F$ is
absolutely continuous with respect to $G\circ
\phi^{-1}$. In turn, if $\varDelta \in \borel{\rbb_+}$
is such that $F(\varDelta)=0$, then \eqref{ale-sp-a}
implies that $G\circ \phi^{-1}(\varDelta \setminus
\{1\})=0$. Since by assumption $G\circ
\phi^{-1}(\{1\})=0$, we see that $G\circ
\phi^{-1}(\varDelta)=0$. This means that the measures
$F$ and $G\circ \phi^{-1}$ are mutually absolutely
continuous, therefore (iii-a) holds. As a consequence,
$\supp{F} = \supp{G\circ \phi^{-1}}$. Combined with
\cite[Lemma~3(5)]{Ci-St07}, this implies (iii-b).
Finally, (iii-c) is a direct consequence of (i), (ii)
and (iii-b).
   \end{proof}
   \begin{cor} \label{subn-m0}
Let $T\in \ogr{\hh}$ be a subnormal
operator and $M$ be as in Theorem~{\em
\ref{dyltyprep}(ii)}. Then $M=0$ if and
only if $T$ is an isometry.
   \end{cor}
   \begin{proof}
If $M=0$, then by Theorem~\ref{subn-1},
$B=C=0$, so by \eqref{cdr5}, $T$ is an
isometry. Conversely, if $T$ is an
isometry, an application of the identity
$p\lrangle{T}=p(1)I$, $p\in \cbb[X]$,
gives \eqref{checpt-2} with $M=0$.
   \end{proof}
Theorem~\ref{glow-main} below gives new
necessary and sufficient conditions for
subnormality. Condition (v) of this theorem
comprises the case $D=0$ which is not covered by
Theorem~\ref{rt=1}.
   \begin{thm} \label{glow-main}
Let $T\in\ogr{\hh}$. Then the following conditions are
equivalent{\em :}
   \begin{enumerate}
   \item[(i)] $T$ is a subnormal contraction,
   \item[(ii)] $T$ is a CPD contraction,
   \item[(iii)] $T$ is CPD and the telescopic series
   \begin{align*}
\sum_{n=0}^{\infty} \big(\|T^{n+1}h\|^2 -
\|T^nh\|^2\big)
   \end{align*}
is convergent in $\rbb$ for every $h \in \hh$,
   \item[(iv)] $T$ is CPD and\/\footnote{It follows from
Propositions~\ref{traj-pd-op} and
\ref{zen-jesz-nie} that, under the assumption
that $T$ is CPD, \eqref{quas-tels} is equivalent
to
   \begin{align*}
\sup_{n\in \zbb_+}(\|T^{n+1}h\|^2 - \|T^nh\|^2)
\Le 0, \quad h\in \hh.
   \end{align*}
   }
   \begin{align}  \label{quas-tels}
\lim_{n\to\infty}(\|T^{n+1}h\|^2 - \|T^nh\|^2) = 0,
\quad h\in \hh,
   \end{align}
   \item[(v)] condition {\em (ii)} of Theorem~{\em
\ref{cpdops}} holds with $C=0$, $D=0$,
$F([1,\infty))=0$ and $\frac{1}{(1-x)^2}
\in L^1(F)$, where $D:=B+\int_{[0,1)}
\frac{1}{1-x} F(\D x)$ $($or equivalently
if all of this holds with
``$\frac{1}{1-x} \in L^1(F)$'' in place
of ``$\frac{1}{(1-x)^2} \in L^1(F)$''$)$,
   \item[(vi)]  condition {\em (ii)} of Theorem~{\em \ref{boundiff}}
holds with $D=0$.
   \end{enumerate}
   \end{thm}
   \begin{proof}
The implications
(i)$\Rightarrow$(ii)$\Rightarrow$(iii)$\Rightarrow$(iv)
are obvious because if $T$ is a
contraction, then the sequence
$\{\|T^nh\|^2\}_{n=0}^{\infty}$, being
monotonically decreasing, is convergent
in $\rbb_+$ for all $h\in \hh$.

(iv)$\Rightarrow$(i) This implication can be deduced
from Corollary~\ref{pd2cpd} (applied to
$\gammab_{T,h}$) and Theorems~\ref{cpdops} and
\ref{lamb}.

(iv)$\Rightarrow$(v) It follows from
Theorem~\ref{boundiff} that (iv) implies the variant
of (v) with ``$\frac{1}{1-x} \in L^1(F)$''. That
$\frac{1}{(1-x)^2} \in L^1(F)$ is a consequence of
Theorem~\ref{subn-1} and the fact that (iv) implies
(i).

(v)$\Rightarrow$(vi) Assume that the variant of (v)
with ``$\frac{1}{(1-x)^2} \in L^1(F)$'' holds. Then,
by the Cauchy-Schwarz inequality, $\frac{1}{1-x} \in
L^1(F)$. Observe that (cf.\ \eqref{adycto3})
   \allowdisplaybreaks
   \begin{align*}
T^{*n}T^n & \overset{\eqref{cdr5}}= I + n \left(B +
\int_{[0,1)} \frac{Q_n(x)}{n} F(\D x)\right)
   \\
& \overset{\eqref{rnx-1}} = I - \int_{[0,1)}
\frac{1-x^n}{(x-1)^2} F(\D x), \quad n\in \nbb.
   \end{align*}
This implies that the pair $(D,F)$ with $D=0$
satisfies condition (ii) of
Theorem~\ref{boundiff}.

(vi)$\Rightarrow$(iv) One can apply
Theorem~\ref{boundiff}.
   \end{proof}
There are other ways to prove some implications
of Theorem~\ref{glow-main}. Namely, one can show
the implication (iv)$\Rightarrow$(ii) by using
the ``moreover'' part of
Proposition~\ref{traj-pd-op}. In turn, the
implication (ii)$\Rightarrow$(i) can be deduced
from \cite[Theorem~3.1]{Ag85} and
Corollary~\ref{nofs-sup2}(ii). The implication
(ii)$\Rightarrow$(i) (with a different proof) is
a part of the conclusion of
\cite[Theorem~4.1]{Sto}. Observe also that by
Theorem~\ref{glow-main}, an operator $T \in
\ogr{\hh}$ is subnormal if and only if there
exists $\alpha \in \cbb \setminus \{0\}$ such
that the operator $\alpha T$ satisfies any of
the equivalent conditions (ii)-(vi) of
Theorem~\ref{glow-main}.
   \begin{cor} \label{cpdalp}
Let $T\in \ogr{\hh}$ obey any of the following
conditions{\em :}
   \begin{enumerate}
   \item[(i)] the sequence $\{\|T^n h\|^2\}_{n=0}^{\infty}$
is convergent in $\rbb_+$ for every $h\in
\hh$,
   \item[(ii)] $T$ is strongly stable,
i.e., $\lim_{n\to\infty} \|T^n h\| = 0$
for every $h\in \hh$
$($\cite{Kub,Kub-b}$)$,
   \item[(iii)] $r(T) < 1$.
   \end{enumerate}
Then $T$ is CPD if and only if $T$ is subnormal.
   \end{cor}
   \begin{cor} \label{scalcpd}
Let $T\in \ogr{\hh}$. Then the
following are equivalent{\em :}
   \begin{enumerate}
   \item[(i)] $T$ is subnormal,
   \item[(ii)] $\alpha T$ is CPD for all $\alpha \in
\cbb$,
   \item[(iii)] zero is an accumulation point of
the set of all $\alpha\in \cbb\setminus \{0\}$
for which $\alpha T$ is CPD,
   \item[(iv)] there exists $\alpha \in \cbb\setminus \{0\}$
such that $|\alpha| \, r(T) < 1$ and $\alpha T$
is CPD.
   \end{enumerate}
   \end{cor}
   \begin{cor} \label{cpdalp-c}
Suppose $T\in \ogr{\hh}$ is a non-subnormal CPD
operator. Then $r(T) \Ge 1$ and $\alpha T$ is
not CPD for any complex number $\alpha$ such
that
   \begin{align*}
0 < |\alpha| < \frac{1}{r(T)}.
   \end{align*}
   \end{cor}
Regarding Corollary~\ref{cpdalp-c}, we refer the
reader to Example~\ref{prz-do-na} for an example
of a non-subnormal CPD operator with $r(T) = 1$.
Below we apply the above to certain translations
of quasinilpotent operators (cf.\ \cite{Hal}).
   \begin{cor} \label{quasi-1-nil}
Let $N\in\ogr{\hh}$ and $\alpha\in \cbb$~be such
that $r(N)=0$ and $|\alpha| < 1$. Then $\alpha I
+ N$ is CPD if and only if~$N=0$.
   \end{cor}
   \begin{proof}
If $\alpha I + N$ is CPD, then, since $r(\alpha
I+N) = |\alpha| <1$, we infer from
Corollary~\ref{cpdalp} and \eqref{subn-norm}
that $\|N\|=r(N)=0$, which shows that $N=0$.
   \end{proof}
Concerning Corollary~\ref{quasi-1-nil}, note
that if $N$ is a nilpotent operator with index
of nilpotency $2$ and $\alpha \in \cbb$ is such
that $|\alpha| = 1$, then by
\cite[Theorem~2.2]{Ber-Mar-No13}, $\alpha I+N$
is a strict $3$-isometry, so by
Proposition~\ref{sub-mzero-n}, $\alpha I+N$ is
CPD. It is an open question as to whether there
exists a quasinilpotent operator $N$ which is
not nilpotent and such that $I+N$ is CPD.

According to the above discussion, the class of
CPD operators is not scalable, i.e., it is not
closed under the operation of multiplying by
nonzero complex scalars. Among non-scalable
classes of operators are those which consist of
$m$-isometric and $2$-hyperexpansive operators
(see \cite[Lemma~1.21]{Ag-St1} and
\cite[Lemma~1]{Rich}, respectively). On the
other hand, the classes of normaloid, hyponormal
and subnormal operators are scalable (see
\cite{Fur} for more examples).

Condition \eqref{quas-tels} of
Theorem~\ref{glow-main} gives rise to a link
between the conditional positive definiteness of
a (bounded) operator $T$ and the subnormality of
(in general unbounded) unilateral weighted shift
operators $W_{T,h}$, $h\in \hh$, defined below.
Given an operator $T\in \ogr{\hh}$ and a vector
$h\in \hh$, we denote by $W_{T,h}$ the
unilateral weighted shift in $\ell^2$ with
weights $\{\E^{\frac 12(\|T^{n+1} h\|^2 -
\|T^{n}h\|^2)}\}_{n=0}^{\infty}$, that is
$W_{T,h}= U D_{T,h}$, where $U\in \ogr{\ell^2}$
is the unilateral shift and $D_{T,h}$ is the
diagonal (normal) operator in $\ell^2$ with the
diagonal $\{\E^{\frac 12(\|T^{n+1} h\|^2 -
\|T^{n}h\|^2)}\}_{n=0}^{\infty}$ (with respect
to the standard orthonormal basis of $\ell^2$).
Then for every $h\in \hh$,
   \begin{align} \label{ogr-iff}
   \begin{minipage}{65ex}
{\em $W_{T,h} \in \ogr{\ell^2}$ if and
only if $\sup_{n\in \zbb_+}
(\|T^{n+1}h\|^2 - \|T^nh\|^2) <
\infty${\em ;} if this is the case,
then $\|W_{T,h}\|^2= \E^{\sup_{n\in
\zbb_+} (\|T^{n+1}h\|^2 -
\|T^nh\|^2)}$.}
   \end{minipage}
   \end{align}
In view of \eqref{ogr-iff}, the weighted shift
$W_{T,h}$ is bounded for all $h\in \hh$ if and
only if $T$ satisfies condition \eqref{zal} of
Proposition~\ref{unif-bund}. As discussed in
Remark~\ref{manyrem}, there are CPD operators
$T$ for which $W_{T,h}$ is unbounded for all
nonzero vectors $h\in \hh$ and $r(T)=1$. We show
below that subnormal contractions $T$ are
precisely those for which all weighted shifts
$W_{T,h}$, $h\in \hh$, are bounded, subnormal
and of norm one. For the definition and basic
facts about unbounded subnormal operators we
refer the reader to
\cite{St-Sz85,St-Sz89,St-Sz89-III}.
   \begin{pro} \label{cpd-exo}
Let $T\in\ogr{\hh}$. Then the following
assertions hold{\em :}
   \begin{enumerate}
   \item[(i)] $T$ is CPD
if and only if $W_{T,h}$ is a $($possibly
unbounded\/$)$ subnormal operator for all $h\in
\hh$,
   \item[(ii)]  the following conditions are
equivalent{\em :}
   \begin{enumerate}
   \item[(ii-a)] $T$ is a subnormal
contraction,
   \item[(ii-b)] $T$ is a CPD
contraction,
   \item[(ii-c)] $W_{T,h}$ is subnormal and $\|W_{T,h}\|=1$ for all $h\in
\hh$,
   \item[(ii-d)] $W_{T,h}$ is CPD
and $\|W_{T,h}\|=1$ for all $h\in \hh$.
   \end{enumerate}
   \end{enumerate}
   \end{pro}
   \begin{proof}
(i) By using Lemma~\ref{cpdpd2} and considering
$\sqrt{t} h$ instead of $h$, we deduce that $T$
is CPD if and only if the sequence
$\{\E^{\|T^{n} h\|^2}\}_{n=0}^{\infty}$ is PD
for all $h\in \hh$. Replacing $h$ by $Th$ and
using Theorem~\ref{Stiech}, we see that the
latter holds if and only if $\{\E^{\|T^{n}
h\|^2}\}_{n=0}^{\infty}$ is a Stieltjes moment
sequence for all $h\in \hh$. Finally, applying
\cite[Theorem~4]{St-Sz89} (with
\cite[Remark~3.1.4]{J-J-S12}), we conclude that
$T$ is CPD if and only if $W_{T,h}$ is subnormal
for all $h\in \hh$.

(ii) The equivalences
(ii-a)$\Leftrightarrow$(ii-b) and
(ii-c)$\Leftrightarrow$(ii-d) follow
from the equivalence
(i)$\Leftrightarrow$(ii) of Theorem~
\ref{glow-main}. Noting first that the
sequence $\{\|T^n
h\|^2\}_{n=0}^{\infty}$ is convergent
in $\rbb_+$ for all $h\in \hh$ whenever
$T$ is a contraction and then using (i)
and \eqref{ogr-iff}, we get the
equivalence
(ii-b)$\Leftrightarrow$(ii-c). This
completes the proof.
   \end{proof}
   \begin{table}[h]
   \begin{center}
   \begin{tabularx}{.88\textwidth}
{|X||c|c|} \hline {\bf \makecell{$T$ is CPD \\
and satisfies \ding{192}, \ding{193} or
\ding{194}}} & {\bf $\overset{?}\implies$ $T$
subnormal} & {\bf $r(T)$} \tabularnewline \hline
\hline \ding{192} \makecell{$\exists h\colon
W_{T,h}$ is not bounded} & NO & $ \Ge 1$
\tabularnewline \hline \ding{193}
\makecell{$\forall h\colon W_{T,h}$ is bounded
and $D \neq 0$} & NEVER & $\Le 1$
\tabularnewline \hline \ding{194}
\makecell{$\forall h\colon W_{T,h}$ is bounded
and $D = 0$} & YES & $\Le 1$ \tabularnewline
\hline
   \end{tabularx}
\vspace{2ex}
   \end{center}
   \caption{ When does conditional positive
definiteness imply subnormality?}
   \label{Table}
   \end{table}

We now recapitulate our considerations in
Table~\ref{Table}. Note that if $T\in \ogr{\hh}$
is CPD and $W_{T,h}$ is bounded for all $h\in
\hh$, then by \eqref{ogr-iff} and
Proposition~\ref{unif-bund}, the limit $D:=
(\mbox{\sc wot})\lim_{n \to \infty}
\{T^{*(n+1)}T^{n+1} -
T^{*n}T^n\}_{n=0}^{\infty}$ exists. This is
especially true in cases \ding{193} and
\ding{194}. To get row \ding{192} apply the
Gelfand's formula for spectral radius and
\eqref{ogr-iff}; row \ding{193} follows from
\eqref{subn-norm}, \eqref{ogr-iff} and
Proposition~\ref{zen-jesz-nie}; row \ding{194}
is a consequence of \eqref{ogr-iff} and
Propositions~\ref{traj-pd-op} and \ref{cpd-exo}
(row \ding{194} also follows from
\eqref{ogr-iff} and Theorems~\ref{boundiff}
and~\ref{glow-main}).

   We close this subsection with a new
characterization of completely hyperexpansive
operators. It can be deduced from
\cite[Theorem~2]{At2} and Lemma~\ref{cpdpd2} by
arguing as in the proof of
Proposition~\ref{cpd-exo}(i). Despite the formal
similarity, the characterizations given in
Propositions~\ref{cpd-exo} and \ref{chyp-exo}
are radically different, because all unilateral
weighted shifts appearing in
Proposition~\ref{chyp-exo} are contractive.
   \begin{pro} \label{chyp-exo}
An operator $T\in\ogr{\hh}$ is completely
hyperexpansive if and only if the
unilateral weighted shift on $\ell^2$
with weights $\{\E^{\frac 12(\|T^n h\|^2
- \|T^{n+1}h\|^2)}\}_{n=0}^{\infty}$ is
subnormal for all $h\in \hh$.
   \end{pro}
   \section{A functional calculus and related matters}
   \subsection{\label{Sec4.1-n}A functional calculus}
We begin by discussing the space
$L^{\infty}(M)$. Suppose $M\colon \borel{\rbb_+}
\to \ogr{\hh}$ is a semispectral measure. We
denote by $L^{\infty}(M)$ the Banach space of
all equivalence classes of $M$-essentially
bounded complex Borel functions on $\rbb_+$
equipped with the $M$-essential supremum norm
(see \cite[Appendix]{Sto3}; see also
\cite[Section~12.20]{Rud73}). We customarily
regard elements of $L^{\infty}(M)$ as functions
that are identified by the equality a.e.\ $[M]$,
the latter meaning ``almost everywhere with
respect to $M$''. In particular, the norm on
$L^{\infty}(M)$ takes the form
   \begin{align*}
\|f\|_{L^{\infty}(M)} = \min\{\alpha\in
\rbb_+\colon M(\{x\in \rbb_+\colon |f(x)| >
\alpha\})=0\}, \quad f\in L^{\infty}(M).
   \end{align*}
The relationship between $L^{\infty}(M)$ and the classical
$L^{\infty}(\mu)$ is explained below.
   \begin{align} \label{do-zacyt}
   \begin{minipage}{70ex}
{\em If $\mu$ is a Borel measure on $\rbb_+$, then
$L^{\infty}(M) = L^{\infty}(\mu)$ if and only if $M$ and
$\mu$ are mutually absolutely continuous{\em }; if this is
the case, then $\|f\|_{L^{\infty}(M)} =
\|f\|_{L^{\infty}(\mu)}$ for every $f\in L^{\infty}(M)$.}
   \end{minipage}
   \end{align}
As shown in Example~\ref{notabc} below, it may not be
possible to find a Borel probability measure on
$\rbb_+$ with respect to which a given semispectral
measure is absolutely continuous.
   \begin{exa} \label{notabc}
Let $\varOmega$ be any uncountable bounded subset of
$\rbb_+$ and let $E\colon \borel{\rbb_+} \to
\ogr{\hh}$ be the spectral measure given by
   \begin{align*}
E(\varDelta) = \bigoplus_{x \in \varOmega}
\chi_{\varDelta}(x) I_{\hh_x}, \quad \varDelta \in
\borel{\rbb_+},
   \end{align*}
where each $\hh_x$ is a nonzero Hilbert space.
Clearly, the following holds.
   \begin{align} \label{last-o}
\text{If $\varDelta \in \borel{\rbb_+}$, then
$E(\varDelta)=0$ if and only if $\varDelta \cap
\varOmega=\emptyset$.}
   \end{align}
Suppose to the contrary that $E$ is absolutely
continuous with respect to a finite Borel measure
$\mu$ on $\rbb_+$. Then by \eqref{last-o}, $\mu(\{x\})
> 0$ for every $x\in \varOmega$, which is impossible
because $\mu$ is finite and $\varOmega$ is uncountable
(see \cite[Problem~12, p.\ 12]{Ash}). Plainly, $E$ is
compactly supported and $\supp{E}=\bar \varOmega$.
   \hfill $\diamondsuit$
   \end{exa}
The situation described in Example~ \ref{notabc}
cannot happen when $\hh$ is separable. What is more,
the following statement holds.
   \begin{align} \label{hrum-1}
   \begin{minipage}{70ex}
{\em Suppose $\hh$ is separable and $M\colon \borel{\rbb_+}
\to \ogr{\hh}$ is a nonzero semispectral measure. Then
there exists a Borel probability measure $\mu$ on $\rbb_+$
such that $M$ and $\mu$ are mutually absolutely
continuous.}
   \end{minipage}
   \end{align}
To see this, take an orthonormal basis $\{e_j\}_{j\in J}$
of $\hh$, where $J$ is a countable index set. Let
$\{a_j\}_{j\in J}$ be any system of positive real numbers
such that
   \begin{align} \label{nomal-iz}
\sum_{j \in J} a_j \is{M(\rbb_+)e_j}{e_j} =1.
   \end{align}
(This is possible because $M\neq 0$.) Define the Borel
measure $\mu$ on $\rbb_+$ by
   \begin{align*}
\mu(\varDelta) = \sum_{j\in J} a_j
\is{M(\varDelta)e_j}{e_j}, \quad \varDelta \in
\borel{\rbb_+}.
   \end{align*}
By \eqref{nomal-iz}, $\mu$ is a probability measure.
If $\varDelta \in \borel{\rbb_+}$ is such that
$\mu(\varDelta)=0$, then
   \begin{align*}
0=\is{M(\varDelta)e_j}{e_j}=\|M(\varDelta)^{1/2}e_j
\|^2, \quad j\in J,
   \end{align*}
which implies that $M(\varDelta)=0$. Thus $E$ is
absolutely continuous with respect to $\mu$. That
$\mu$ is absolutely continuous with respect to $M$ is
immediate.

We now prove the following fact.
   \begin{align}
   \begin{minipage}{70ex} \label{linft-con}
{\em If $M\colon \borel{\rbb_+} \to \ogr{\hh}$ is a
nonzero compactly supported semispectral measure, then
$\|f\|_{L^{\infty}(M)}=\|f|_{\varOmega}\|_{C(\varOmega)}$
for every $f\in L^{\infty}(M)$ such that
$f|_{\varOmega} \in C(\varOmega)$, where
$\varOmega:=\supp{M}$.}
   \end{minipage}
   \end{align}
Indeed, the inequality ``$\Le$'' is obvious. If
$\alpha \in \rbb_+$ is such that
   \begin{align*}
M(\{x \in \rbb_+\colon |f(x)|>\alpha\})=0,
   \end{align*}
then $M(\{x \in \varOmega\colon |f(x)|>\alpha\})=0$
and, because the set $\{x \in \varOmega \colon
|f(x)|>\alpha\}$ is open in $\varOmega$, we deduce
that $|f(x)|\Le \alpha$ for all $x\in \varOmega$,
which after taking infimum over such $\alpha$'s yields
the inequality ``$\Ge$''. This proves
\eqref{linft-con}.

As a consequence of \eqref{linft-con}, we have
   \begin{align*}
   \begin{minipage}{75ex}
{\em if $f,g \in L^{\infty}(M)$ are such that
$f|_{\varOmega},g|_{\varOmega} \in C(\varOmega)$ and
$f=g$ a.e.\ $[M]$, then $f|_{\varOmega} =
g|_{\varOmega}$.}
   \end{minipage}
   \end{align*}
The above discussion shows that (still under the
assumptions of \eqref{linft-con}) the map which sends
a function $g\in C(\varOmega)$ to the equivalence
class of any of extensions of $g$ to a complex Borel
function on $\rbb_+$ is an isometry from
$C(\varOmega)$ to $L^{\infty}(M)$. Therefore,
$C(\varOmega)$ can be regarded as a closed vector
subspace of $L^{\infty}(M)$; this fact plays an
important role in Theorem~\ref{dyl-an2}(v) below. As
shown in \eqref{hrum-1} and \eqref{do-zacyt}, if $\hh$
is separable and $M\neq 0$, then
$L^{\infty}(M)=L^{\infty}(\mu)$ for some Borel
probability measure on $\rbb_+$, so $C(\varOmega)$ is
a separable closed vector subspace of
$L^{\infty}(\mu)$ (see \cite[Theorem~V.6.6]{con2}),
while, in general, $L^{\infty}(\mu)$ is not separable
(see \cite[Problem~2, p.\ 62]{Tay-La80}). As is easily
seen, the above facts (except for separability of
$C(\varOmega)$) are true for regular Borel
semispectral measures on topological Hausdorff spaces.

   We are now ready to construct an
$L^{\infty}(M)$-functional calculus that is built up
on the basis of Agler's hereditary functional
calculus.
   \begin{thm} \label{dyl-an2}
Suppose that $T\in \ogr{\hh}$ is a CPD operator.
Let $M\colon \borel{\rbb_+} \to \ogr{\hh}$ be a
compactly supported semispectral measure
satisfying \eqref{checpt-2}. Then the map
$\varLambda_{T}\colon L^{\infty}(M) \to
\ogr{\hh}$ given by
   \begin{align} \label{elem-rz}
\varLambda_{T}(f) = \int_{\rbb_+} f \D M, \quad f \in
L^{\infty}(M),
   \end{align}
is continuous and linear. It has the following
properties{\em :}
   \begin{enumerate}
   \item[(i)]
$\varLambda_{T}(q) = ((X-1)^2q)\lrangle{T}$ for every
$q\in \cbb[X]$,
   \item[(ii)] $\varLambda_{T}$ is positive\/\footnote{Since
$\varLambda_{T}$ is a positive map on a commutative
$C^*$-algebra $L^{\infty}(M)$, the Stinespring theorem
implies that $\varLambda_{T}$ is completely positive
(see \cite[Theorem~4]{Stin55}).}, i.e.,
$\varLambda_{T}(f) \Ge 0$ whenever $f\in
L^{\infty}(M)$ and $f\Ge 0$ a.e.\ ~$[M]$,
   \item[(iii)] there exist a Hilbert space
$\kk$, $R\in \ogr{\hh,\kk}$ and $S\in \ogr{\kk}_+$
such that {\em \eqref{mini-tr}} holds and
   \begin{align} \label{dilat-R}
\varLambda_{T}(f) = R^* f(S) R, \quad f \in
L^{\infty}(M),
   \end{align}
   \item[(iv)] $\|\varLambda_{T}\|=\|\bscr_2(T)\|$,
   \item[(v)] if $M$ is nonzero and $\varOmega:=\supp{M},$ then
$\cbb[X]$ is dense in $C(\varOmega)$ in the
$L^{\infty}(M)$-norm,
$\varLambda_{T}|_{C(\varOmega)}\colon C(\varOmega) \to
\ogr{\hh}$ is a unique continuous linear map
satisfying {\em (i)},
$\|\varLambda_{T}|_{C(\varOmega)}\|=\|\varLambda_{T}\|$~
and
   \begin{align} \label{nier-wilo-1}
\bigg\|\sum_{j=0}^n \alpha_j
T^{*j}\bscr_2(T)T^j\bigg\| \Le \|\bscr_2(T)\|
\sup_{x\in \varOmega}\bigg|\sum_{j=0}^n \alpha_j
x^j\bigg|, \quad \{\alpha_j\}_{j=0}^n \subseteq \cbb,
n\in \zbb_+.
   \end{align}
   \end{enumerate}
   \end{thm}
   \begin{proof}
Let $(\kk,R,E)$ be as in Theorem~\ref{Naim-ark}. Since
for every $\varDelta \in \borel{\rbb_+}$,
$E(\varDelta)=0$ if and only if $M(\varDelta)=0$ (see
the proof of \cite[Theorem~4.4]{Ja02}), we get
   \begin{align} \label{baba-1}
\text{$L^{\infty}(M)=L^{\infty}(E)$ and
$\|f\|_{L^{\infty}(M)} = \|f\|_{L^{\infty}(E)}$ for every
$f\in L^{\infty}(M)$.}
   \end{align}
Set $S=\int_{\rbb_+} x E(\D x)$. Applying
\eqref{baba-1} and the Stone-von Neumann functional
calculus, we deduce from \eqref{elem-rz} that
\eqref{dilat-R} is valid and consequently
   \begin{align*}
\|\varLambda_{T}(f)\| \Le \|R\|^2\|f(S)\| = \|R\|^2
\|f\|_{L^{\infty}(E)} = \|R\|^2 \|f\|_{L^{\infty}(M)},
\quad f\in L^{\infty}(M).
   \end{align*}
Hence, $\varLambda_{T}$ is a continuous positive
linear map such that
   \begin{align} \label{dor-s}
\|\varLambda_{T}\|\Le \|R\|^2.
   \end{align}
Applying \eqref{checpt-2} to $p=(X-1)^2q$, we deduce
that $\varLambda_{T}$ satisfies (i). Substituting
$q=f=X^0$ into (i) and \eqref{dilat-R}, we infer from
\eqref{nab-bla2} that
   \begin{align} \label{dzi-d}
\bscr_2(T)=\varLambda_{T}(X^0)=R^*R,
   \end{align}
which together with \eqref{dor-s} yields
$\|\varLambda_{T}\|=\|R\|^2 = \|\bscr_2(T)\|$. Thus,
in view of Lemma~\ref{minim-2}(a), (i)-(iv) hold.

It remains to prove (v). Assume that $M$ is
nonzero. It follows from \eqref{linft-con} and
the Stone-Weierstrass theorem (or the classical
Weierstrass theorem combined with Tietze
extension theorem) that $\cbb[X]$ is dense in
$C(\varOmega)$ in the $L^{\infty}(M)$-norm and
so $\varLambda_{T}|_{C(\varOmega)}\colon
C(\varOmega) \to \ogr{\hh}$ is a unique
continuous linear map satisfying condition~(i).
Since by \eqref{nab-bla2} and \eqref{nab-bla},
   \begin{align} \label{Geoje}
\sum_{j=0}^n \alpha_j T^{*j}\bscr_2(T)T^j =
((X-1)^2q)\lrangle{T}\overset{\mathrm{(i)}}=\varLambda_{T}(q),
   \end{align}
where $q= \sum_{j=0}^n \alpha_j X^j$, we can easily
deduce \eqref{nier-wilo-1} from (iv) and
\eqref{linft-con}. Using \eqref{dzi-d} and (iv) again,
we conclude that
$\|\varLambda_{T}|_{C(\varOmega)}\|=\|\varLambda_{T}\|$.
This proves (v) and thus completes the proof.
   \end{proof}
Before stating a corollary to Theorem~\ref{dyl-an2},
we recall that a monic polynomial $p\in \cbb[X]$ of
degree at least one takes the form (see \cite[p.\
252]{hun74})
   \begin{align}  \label{pazdz-list00}
p=(X-z_1) \cdots (X-z_n),
   \end{align}
where $z_1, \ldots, z_n\in \cbb$. What is more, $p$
can be written as
   \begin{align} \label{pazdz-list01}
p = \sum_{j=0}^n (-1)^{n-j} s_{n-j}(z_1, \ldots,
z_n)X^j,
   \end{align}
where $s_0=1$ and $s_1, \ldots, s_n$ are the
elementary symmetric functions in complex variables
$z_1, \ldots, z_n$ given by
   \begin{align*}
s_j(z_1, \ldots, z_n) = \sum_{1 \Le i_1 < \ldots < i_j
\Le n} z_{i_1} \cdots z_{i_j} \text{ for } z_1,
\ldots, z_n\in \cbb \text{ and } j=1, \dots, n.
   \end{align*}
   \begin{cor} \label{wni-pol}
Assume that $T\in \ogr{\hh}$ is CPD and $M$ and
$\varOmega$ are as in Theorem~{\em
\ref{dyl-an2}(v)}. Then for every $n\in \nbb$,
   \allowdisplaybreaks
   \begin{align}  \notag
\Big\|\sum_{j=0}^n (-1)^j s_{n-j}(\boldsymbol z)
T^{*j}\bscr_2(T)T^j\Big\| & \Le \|\bscr_2(T)\|
\sup_{x\in \varOmega} \prod_{j=1}^n |x-z_j|,
   \\ \label{pazdz-list1}
& \hspace{12.5ex}\boldsymbol z = (z_1, \ldots, z_n)
\in \cbb^n,
   \\ \label{pazdz-list2}
\Big \|\sum_{j=0}^n \binom{n}{j} (-1)^j z^{n-j}
T^{*j}\bscr_2(T)T^j \Big\| &\Le \|\bscr_2(T)\|
\Big(\sup_{x\in \varOmega} |x-z|\Big)^n, \quad z \in
\cbb,
   \\ \label{pazdz-list3}
\hspace{-10ex}\|\bscr_{n+2}(T)\| &\Le \|\bscr_2(T)\|
\Big(\sup_{x\in \varOmega} |x-1|\Big)^n.
   \end{align}
   \end{cor}
   \begin{proof} Applying \eqref{nier-wilo-1} to the
polynomial \eqref{pazdz-list01} and using
\eqref{pazdz-list00}, we get \eqref{pazdz-list1}. The
estimate \eqref{pazdz-list2} is a direct consequence
of \eqref{pazdz-list1}. Finally, the estimate
\eqref{pazdz-list3} follows from \eqref{pazdz-list2}
applied to $z=1$ and the identity
   \begin{align*}
\bscr_{n+2}(T) = \sum_{j=0}^n \binom{n}{j} (-1)^j
T^{*j}\bscr_2(T)T^j, \quad n\in \zbb_+.
   \end{align*}
which can be proved straightforwardly by using
Agler's hereditary functional calculus (see
\eqref{nab-bla2} and \eqref{Geoje}). The
estimate \eqref{pazdz-list3} can also be
inferred from identity \eqref{stary-num} and
statements (viii) and (ix) of
\cite[Theorem~A.1]{Sto3}.
   \end{proof}
In the case of CPD operators,
Lemma~\ref{ide-fin} takes the following form for
$q_0=(X-1)^2$.
   \begin{pro} \label{ide-fin-2}
Let $T\in \ogr{\hh}$ be a CPD operator and let
$M$ be as in Theorem~{\em \ref{dyl-an2}}. Then
the set
   \begin{align*}
\ideal=\{q\in \cbb[X]\colon ((X-1)^2q)\lrangle{T}=0\},
   \end{align*}
is the ideal in $\cbb[X]$ generated by the polynomial
$w_T\in \cbb[X]$ defined by
   \begin{align*}
w_T =
   \begin{cases}
0 & \text{if $\varOmega$ is infinite},
   \\[1ex]
\prod_{u\in \varOmega} (X-u) & \text{if $\varOmega$ is
finite and nonempty},
   \\[1ex]
X^0 & \text{if $\varOmega=\emptyset$, or equivalently if
$T$ is a $2$-isometry},
   \end{cases}
   \end{align*}
where $\varOmega:=\supp{M}$. Moreover, if
$\varOmega=\{u_1, \ldots,u_n\}$, where $n\in \nbb$ and
$u_1, \ldots, u_n$ are distinct, then the following
identity holds
   \begin{align*}
   \sum_{j=0}^n (-1)^j s_{n-j}(u_1, \dots, u_n)
T^{*j}\bscr_2(T)T^j =0.
   \end{align*}
   \end{pro}
   \begin{proof}
First note that by \eqref{elem-rz} and
Theorem~\ref{dyl-an2}(i),
   \begin{align} \label{wt-tc2}
\int_{\varOmega} p(x) M(\D x) = \varLambda_{T}(p)
=((X-1)^2 p)\lrangle{T}, \quad p\in \cbb[X].
   \end{align}
It follows from Lemma~\ref{ide-fin} that $\ideal$ is
an ideal in $\cbb[X]$ generated by some polynomial
$w\in \cbb[X]$. Applying \eqref{wt-tc2} to $p=w^*w$,
we see that $\int_{\varOmega} |w(x)|^2 M(\D x)=0$, and
so, by \eqref{form-ua}, we have
   \begin{align}  \label{wt-tc}
w|_{\varOmega}=0.
   \end{align}

We now consider three cases.

{\sc Case 1} The set $\varOmega$ is infinite.

Since nonzero polynomials may have only finite number
of roots, we deduce from \eqref{wt-tc} that $w=0$.

{\sc Case 2} The set $\varOmega$ is empty (or
equivalently, by Proposition~\ref{sub-mzero-n}, $T$ is
a $2$-isometry).

Then, in view of \eqref{wt-tc2}, $\ideal=\cbb[X]$ and
so $X^0$ generates the ideal $\ideal$.

{\sc Case 3} The set $\varOmega$ is finite and
nonempty.

Set $w_T=\prod_{u\in \varOmega} (X-u)$. Clearly, by
\eqref{wt-tc2}, $w_T\in \ideal$. It follows from the
fundamental theorem of algebra (see
\cite[Theorem~V.3.19]{hun74}) and \eqref{wt-tc} that
the polynomial $w_T$ divides $w$. Since $w$ generates
the ideal $\ideal$, $w$ divides $w_T$ and so
$w_T=\alpha w$, where $\alpha \in \cbb\setminus
\{0\}$. This means that $w_T$ generates $\ideal$.

The ``moreover'' part is a direct consequence of
\eqref{pazdz-list1}.
   \end{proof}
Theorem~\ref{dyl-an2}(i), Proposition~\ref{ide-fin-2}
and \eqref{linft-con} lead to the following corollary.
   \begin{cor}
Let $T\in \ogr{\hh}$ be a CPD operator and let
$M$ and $\varLambda_{T}$ be as in Theorem~{\em
\ref{dyl-an2}}. Then the following assertions
hold{\em :}
   \begin{enumerate}
   \item[(i)] the map $\cbb[X] \ni q  \mapsto ((X-1)^2q)\lrangle{T} \in
\ogr{\hh}$ is injective if and only if $\supp{M}$ is
infinite,
   \item[(ii)] if $q \in \cbb[X]$ is such
that $\varLambda_{T}(q)=0$, then $q=0$ a.e.\ $[M]$,
which means that the restriction of $\varLambda_{T}$
to equivalence classes of polynomials is injective.
   \end{enumerate}
   \end{cor}
   \subsection{\label{appl-to}Analytic implementations}
In Subsection~\ref{Sec4.1-n} we were
discussing the action of the functional
calculus established in
Theorem~\ref{dyl-an2} on polynomials. In
this subsection we concentrate on showing
how this functional calculus may work in
the case of real analytic functions.

Let $T$, $M$ and $\varLambda_{T}$ be as in
Theorem~\ref{dyl-an2}. Assume that $M$ is nonzero. Set
$\varOmega=\supp{M}.$ Suppose that
$\sum_{n=0}^{\infty} a_n x^n$ is a power series in the
real variable $x$ with complex coefficients $a_n$
such~ that
   \begin{align} \label{ple2}
\limsup_{n\to \infty} |a_n|^{1/n} <
\frac{1}{\sup{\varOmega}} \qquad \Big(\text{with
$\frac{1}{0}=\infty$}\Big).
   \end{align}
Then the series $\sum_{n=0}^{\infty} a_n x^n$ is uniformly
convergent on $[0,\sup{\varOmega}]$ to a continuous
function on $\varOmega$, say $f$. Hence by
Theorem~\ref{dyl-an2}, we have
   \begin{align}  \label{ple1}
\varLambda_{T}(f)= \sum_{n=0}^{\infty} a_n
\varLambda_{T}(X^n) \overset{\eqref{Geoje}}=
\sum_{n=0}^{\infty} a_n T^{*n}\bscr_2(T)T^n.
   \end{align}
Let $(\kk,R,S)$ be as Theorem~\ref{dyl-an2}(iii). In
particular, \eqref{mini-tr} holds and
   \begin{align*}
\varLambda_{T}(f) \overset{\eqref{dilat-R}}= R^* f(S)
R.
   \end{align*}
Combined with \eqref{ple1}, this implies that
   \begin{align} \label{ple4}
\sum_{n=0}^{\infty} a_n T^{*n}\bscr_2(T)T^n = R^* f(S) R.
   \end{align}
It follows from Theorem~\ref{dyl-an2}(i) and
\eqref{dilat-R} that \eqref{dur-a} holds. According to
the proof of the implication
(ii$^{\prime}$)$\Rightarrow$(i) of
Theorem~\ref{dyl-an}, \eqref{checpt-3} holds, so by
Theorem~\ref{dyltyprep}(c), $\sigma(S)=\varOmega$ and
$\|S\|=\sup{\varOmega}$. Since the map $C(\sigma(S))
\ni g \longmapsto g(S) \in \ogr{\hh}$ is a unital
isometric $*$-homomorphism (see \cite[Theorem~
VIII.2.6]{con2}), we get (see also \eqref{linft-con}
and \eqref{baba-1})
   \begin{align} \label{Neumann}
f(S) = \sum_{n=0}^{\infty} a_n S^n.
   \end{align}

Concerning \eqref{ple1}, note that
   \begin{align*}
\sum_{n=0}^{\infty} a_n T^{*n}\bscr_2(T)T^n =
\sum_{n=0}^{\infty} a_n \nabla_T^n (\bscr_2(T)),
   \end{align*}
where $\nabla_T$ is as in \eqref{nubile}. Since
$r(\nabla_T)=r(T)^2$ (a general fact which
follows from Gelfand's formula for spectral
radius), we deduce that the series
$\sum_{n=0}^{\infty} a_n \nabla_T^n$ converges
in $\ogr{\ogr{\hh}}$ if $\limsup_{n\to \infty}
|a_n|^{1/n} < \frac{1}{r(T)^2}$. The last
inequality is in general stronger than
\eqref{ple2} because by
Theorem~\ref{dyltyprep}(c),
   \begin{align*}
\frac{1}{r(T)^2} \Le \frac{1}{\sup
\varOmega}.
   \end{align*}

Let us now discuss two important cases. We begin with
$a_n=z^n$ for every $n\in \zbb_+$, where $z\in \cbb$.
Then the above considerations lead to
   \begin{align}  \label{dif-a}
\sum_{n=0}^{\infty} z^n T^{*n}\bscr_2(T)T^n
\overset{(\dag)}= R^* (I-z S)^{-1} R, \quad z\in \cbb, \,
|z|< \frac{1}{\sup{\varOmega}},
   \end{align}
where $(\dag)$ follows from \eqref{ple2},
\eqref{ple4}, \eqref{Neumann} and the Carl Neumann
theorem (see \cite[Theorem~10.7]{Rud73}). In
particular, the following estimate holds (see
\eqref{dzi-d})
   \begin{align*}
\Big \|\sum_{n=0}^{\infty} z^n T^{*n}\bscr_2(T)T^n\Big
\| \Le
|z^{-1}|\frac{\|\bscr_2(T)\|}{\mathrm{dist}(z^{-1},
\varOmega)}, \quad z\in \cbb, \, 0<|z| <
\frac{1}{\sup{\varOmega}}.
   \end{align*}
In view of the previous paragraph and \eqref{dif-a},
we have
   \begin{align*}
(\boldsymbol I-z \nabla_T)^{-1} (\bscr_2(T))= R^* (I-z
S)^{-1} R, \quad z\in \cbb, \, |z|< \frac{1}{r(T)^2},
   \end{align*}
where $\boldsymbol I$ is the identity map on
$\ogr{\hh}$. Note also that if \eqref{dif-a}
holds, then by differentiating the operator
valued functions appearing on both sides of the
equality in \eqref{dif-a} $n$ times at $0$, we
obtain \eqref{dur-c}, which by
Theorem~\ref{dyl-an} implies that $T$ is CPD.

It is a matter of routine to show that for an
operator $T\in \ogr{\hh}$ the operator valued
function appearing on the left-hand side of
\eqref{dif-a}, call it $\varPsi$, is uniquely
determined by the requirement that it be an
analytic $\ogr{\hh}$-valued function defined on
an open disk $\dbb_r = \{z\in \cbb \colon |z| <
r\}$ for some $r\in (0,\infty)$ such~that
   \begin{align} \label{most-1}
\varPsi(z) = \bscr_2(T) + zT^*\varPsi(z)T, \quad z\in
\dbb_r.
   \end{align}
In other words, we have proved that $T$ is CPD
if and only if there exists $r\in (0,\infty)$
such that the analytic function $\varPsi$
associated with $T$ via \eqref{most-1} satisfies
the following equation
   \begin{align*}
\varPsi(z) = R^* (I-z S)^{-1} R, \quad z\in \dbb_r,
   \end{align*}
for some triplet $(\kk,R,S)$ consisting of a Hilbert
space $\kk$, an operator $R\in \ogr{\hh,\kk}$ and a
positive operator $S\in \ogr{\kk}$ such that $r \|S\|
\Le 1$.

In turn, if $a_n=\frac{z^n}{n!}$ for every $n \in
\zbb_+$, where $z\in \cbb$, then
   \begin{align*}
\sum_{n=0}^{\infty} \frac{z^n}{n!} T^{*n}\bscr_2(T)T^n
\overset{(\ddag)}= R^* \E^{zS} R, \quad z\in \cbb,
   \end{align*}
where $(\ddag)$ is a consequence of \eqref{ple2},
\eqref{ple4} and \eqref{Neumann}, or equivalently that
   \begin{align*}
\E^{z \nabla_T}(\bscr_2(T)) = R^* \E^{zS} R, \quad
z\in \cbb.
   \end{align*}
In particular, we have
   \begin{align} \label{dur-b}
\sum_{n=0}^{\infty} \frac{\I^n x^n}{n!} T^{*n}\bscr_2(T)T^n
= R^* \E^{\I x S} R, \quad x\in \rbb.
   \end{align}
Since $\{\E^{\I x S}\}_{x\in \rbb}$ is a uniformly
continuous group of unitary operators, we obtain
   \begin{align*}
\bigg\|\sum_{n=0}^{\infty} \frac{\I^n x^n}{n!}
T^{*n}\bscr_2(T)T^n\bigg\| \overset{\eqref{dur-b}} \Le
\|R\|^2 \overset{\eqref{dzi-d}}= \|\bscr_2(T)\|, \quad
x\in \rbb,
   \end{align*}
or equivalently
   \begin{align*}
\|\E^{ix\nabla_T}(\bscr_2(T))\| \Le \|\bscr_2(T)\|,
\quad x \in \rbb.
   \end{align*}
As in the previous case, we observe that if
\eqref{dur-b} holds, then by differentiating the
operator valued functions appearing on both
sides of the equality in \eqref{dur-b} $n$ times
at $0$, we obtain \eqref{dur-c}, which as we
know implies that $T$ is CPD.
   \subsection{Small supports}
In view of Subsection~\ref{appl-to}, the natural
question arises of when the closed support of
the semispectral measure $M$ associated with a
given CPD operator $T$ via
Theorem~\ref{dyltyprep}(ii) is equal to
$\emptyset$, $\{0\}$ or $\{1\}$. Surprisingly,
the answers to this seemingly simple question
that are given in Propositions~\ref{sub-mzero-n}
and \ref{kop-2izo} (see also
Corollary~\ref{subn-m0}) lead to three
relatively broad classes of operators, including
$2$- and $3$-isometries. The fact that
$3$-isometries are CPD was already proved in
\cite[Proposition~2.7]{Cha-Sh}. In
Proposition~\ref{sub-mzero-n} below, $F$ and $M$
denote the semispectral measures appearing in
Theorems~\ref{cpdops}(ii) and
\ref{dyltyprep}(ii), respectively.
   \begin{pro} \label{sub-mzero-n}
Let $T\in \ogr{\hh}$. Then
   \begin{enumerate}
   \item[(i)] if $T$ is CPD, then $M=0$
if and only if $\varLambda_{T}=0$, or equivalently if
and only if $T$ is a $2$-isometry,
   \item[(ii)] the following conditions are
equivalent{\em :}
   \begin{enumerate}
   \item[(a)] $T$ is CPD
and $F=0$,
   \item[(b)] $T$ is CPD and $\supp{M}\subseteq \{1\}$,
   \item[(c)] $T^{*n}T^n = I - n \bscr_1(T) + \frac{n(n-1)}{2}\bscr_2(T)$
for all $n\in \zbb_+$,
   \item[(d)] $T$ is a $3$-isometry.
   \end{enumerate}
   \end{enumerate}
Moreover, if an $m$-isometry is CPD, then it is
a $3$-isometry.
   \end{pro}
   \begin{proof}
(i) The first equivalence in (i) follows from
\eqref{elem-rz} by considering characteristic
functions, while the second is a direct consequence of
Theorem~\ref{dyl-an2}(iv).

(ii) The implication (a)$\Rightarrow$(c) follows from
\eqref{cdr5}. Straightforward computations shows that
the implication (c)$\Rightarrow$(d) holds. If (d)
holds, then for all $n\Ge 2$,
   \begin{align*}
T^{*n}T^n=((X-1)+1)^n\lrangle{T} = \sum_{j=0}^n
\binom{n}{j} (X-1)^j\lrangle{T} \overset{(*)}=
\sum_{j=0}^2 \binom{n}{j} (-1)^j\bscr_j(T),
   \end{align*}
where $(*)$ follows from Remark~\ref{m-plus-k}. This
yields~(c). If (c) holds, then the right-hand side of
the equality in (c) is nonnegative for all $n\in
\zbb_+$, which implies that $\bscr_2(T) \Ge 0$.
Clearly \eqref{cdr5} holds with $B=-(\bscr_1(T) +
\frac 12 \bscr_2(T))$, $C=\frac 12 \bscr_2(T)$ and
$F=0$, so by Theorem~\ref{cpdops}, (a) holds. By
\eqref{f2m-semi}, (a) and (b) are equivalent.

The ``moreover'' part follows from
\cite[Theorem~3.3]{J-J-S20} and
Proposition~\ref{ojoj1}.
   \end{proof}
As shown below the class of $2$-isometries is
the intersection of the classes of CPD and
$2$-hyperexpansive operators.
   \begin{pro} \label{cpd-ch}
If $T\in \ogr{\hh}$, then the following conditions are
equivalent{\em :}
   \begin{enumerate}
   \item[(i)] $T$ is a $2$-isometry,
   \item[(ii)] $T$ is completely hyperexpansive and CPD,
   \item[(iii)] $T$ is $2$-hyperexpansive and CPD.
   \end{enumerate}
   \end{pro}
   \begin{proof}
(i)$\Rightarrow$(ii) By \cite[Lemma~1]{Rich}, any
$2$-hyperexpansive operator $T$ is expansive, i.e.,
$\bscr_1(T) \Le 0$. This and Remark~\ref{m-plus-k}
implies that any $2$-isometry is completely
hyperexpansive. In view of
Proposition~\ref{sub-mzero-n}, (ii) is valid.

(ii)$\Rightarrow$(iii) This is obvious.

(iii)$\Rightarrow$(i) Applying \eqref{stary-num} to
$m=2$, we see that $M(\rbb_+)\Le 0$, which implies
that $M=0$, so again by \eqref{stary-num} with $m=2$,
$\bscr_2(T)=0$ showing that $T$ is a $2$-isometry.
This completes the proof.
   \end{proof}
The implication (ii)$\Rightarrow$(i) of
Proposition~\ref{cpd-ch} follows also
from \cite[Theorem~ 2]{At2}. The above
enables us to make several comments
related to Theorem~\ref{boundiff} and
Proposition~\ref{wzrostkwadr}.
   \begin{rem} \label{manyrem}
a) First, note that each $2$-isometry $T\in
\ogr{\hh}$ satisfies condition (i) of
Theorem~\ref{boundiff}. Indeed, by
Proposition~\ref{sub-mzero-n}, $T$ is CPD and
   \begin{align*}
T^{*(n+1)}T^{n+1} - T^{*n}T^n = B, \quad n\in \zbb_+,
   \end{align*}
where $B=-\bscr_1(T)$.

b) Suppose that $T\in \ogr{\hh}$ is a strict
$3$-isometry. Then, by
Proposition~\ref{sub-mzero-n}, $T$ is CPD.
However, $T$ does not satisfy condition
\eqref{zal}. In fact, we can show more. By
\eqref{stary-num} and
Proposition~\ref{sub-mzero-n},
   \begin{align*}
\text{$\bscr_2(T) \Ge 0$ and $T^{*(n+1)}T^{n+1} -
T^{*n}T^n = -\bscr_1(T) + n\bscr_2(T)$ for all $n\in
\zbb_+$.}
   \end{align*}
This yields
   \begin{align} \label{fishsoon}
\sup_{n\in \zbb_+} (\|T^{n+1}h\|^2 - \|T^nh\|^2) =
   \begin{cases}
-\is{\bscr_1(T)h}h & \text{if } h\in \jd{\bscr_2(T)},
   \\[.5ex]
\infty & \text{if } h\in \hh \setminus
\jd{\bscr_2(T)}.
   \end{cases}
   \end{align}
Since $T$ is not a $2$-isometry, $\jd{\bscr_2(T)} \neq
\hh$, so $T$ does not satisfy \eqref{zal}.

c) It turns out that there are strict
$3$-isometries $T$ such that
$\jd{\bscr_2(T)}=\{0\}$. Indeed, let $W$ be the
unilateral weighted shift on $\ell^2$ with
weights
$\big\{\frac{\sqrt{n+3}}{\sqrt{n+1}}\big\}_{n=0}^{\infty}$.
It follows from \cite[Proposition~ 8]{At91} and
\cite[Lemma~1.21]{Ag-St1} that $W$ is a strict
$3$-isometry for which $r(W)=1$. We claim that
   \begin{align} \label{fishsoon2}
\jd{\bscr_2(W)}=\{0\}.
   \end{align}
Indeed, it is a matter of routine to verify that
$\bscr_2(W)$ is the diagonal operator (with respect to
the the standard orthonormal basis of $\ell^2$) with
the diagonal
$\big\{\frac{2}{(n+1)(n+2)}\big\}_{n=0}^{\infty}$,
which yields \eqref{fishsoon2}. In particular,
\eqref{fishsoon2} implies that $W$ is a strict
$3$-isometry and, by \eqref{fishsoon},
   \begin{align*}
\sup_{n\in \zbb_+} (\|W^{n+1}h\|^2 - \|W^nh\|^2) =
\infty, \quad h\in \ell^2 \setminus \{0\}.
   \end{align*}

d) Let $W$ be the unilateral weighted shift as in c).
Then $W$ is a $3$-isometry and, by
Proposition~\ref{sub-mzero-n}, we have
   \begin{align} \label{swieta-1}
W^{*n}W^n = I + n B + n^2 C, \quad n\in \zbb_+,
   \end{align}
where $B=-\big(\bscr_1(W) +
\frac{1}{2}\bscr_2(W)\big)$ and
$C=\frac{1}{2}\bscr_2(W)$. We easily check that
$B$ and $C$ are diagonal operators with
diagonals
$\big\{\frac{2n+3}{(n+1)(n+2)}\big\}_{n=0}^{\infty}$
and
$\big\{\frac{1}{(n+1)(n+2)}\big\}_{n=0}^{\infty}$,
respectively, so $B\Ge 0$, $C\Ge 0$ and
$\jd{B}=\jd{C}=\{0\}$. By \eqref{swieta-1},
$\|W^n\| \Le \alpha \cdot n$ for all $n\in
\nbb$, where $\alpha = \sqrt{1+\|B\|+\|C\|}$. We
show that there are no $\varepsilon \in
(0,\infty)$ and $\beta\in \rbb_+$ such that
$\|W^n\| \Le \beta \cdot n^{1-\varepsilon}$ for
all $n\in \nbb$. Indeed, otherwise we have
   \begin{align*}
\is{Ch}{h} \Le \frac{\is{(I+n B+n^2 C)h}{h} }{n^2}
\overset{\eqref{swieta-1}}= \frac{\|W^nh\|^2 }{n^2}
\Le \frac{\beta^2 \|h\|^2}{n^{2\varepsilon}}, \quad
n\in \nbb, \, h\in \ell^2,
   \end{align*}
which contradicts $\jd{C}=\{0\}$. \hfill
$\diamondsuit$
   \end{rem}
We now turn to the case when $\supp{M}=\{0\}$.
We first prove a result that is of some
independent interest (see
\cite[Proposition~8]{Cu90} for the case of
unilateral weighted shifts).
   \begin{lem}\label{restr-subn}
Suppose that the restriction of an operator
$T\in\ogr{\hh}$ to $\overline{\ob{T}}$ is subnormal.
Then $T$ is subnormal if and only if
   \begin{align} \label{wkw-www}
\text{$\int_{\rbb_+} \frac{1}{t} \D\mu_h(t) \Le 1$ for
all $h\in \hh$ such that $\|h\|=1$,}
   \end{align}
where $\mu_h$ stands for the $($unique$)$ representing
measure of the Stieltjes moment sequence $\{\|T^{n+1}
h\|^2\}_{n=0}^{\infty}$.
   \end{lem}
   \begin{proof}
Applying Theorem~\ref{lamb} to
$T|_{\overline{\ob{T}}}$ and using
Lemma~\ref{csmad}, we see that the sequence
$\{\|T^{n+1} h\|^2\}_{n=0}^{\infty}$ is a
determinate Stieltjes moment sequence for every
$h\in \hh$. By Theorem~\ref{lamb}, $T$ is
subnormal if and only if for every $h\in \hh$
for which $\|h\|=1$, the sequence $\{\|T^n
h\|^2\}_{n=0}^{\infty}$ is a Stieltjes moment
sequence, or equivalently, by \cite[Lemma
6.1.2]{J-J-S12}, if and only if condition
\eqref{wkw-www} holds.
   \end{proof}
   \begin{pro}\label{kop-2izo}
For $T\in \ogr{\hh}$, the following conditions are
equivalent{\em :}
   \begin{enumerate}
   \item[(i)] $T$ is CPD
and $\supp{M} = \{0\}$, where $M$ is as in
Theorem~{\em \ref{dyl-an2}},
   \item[(ii)] $\bscr_2(T)T=0$, $\bscr_2(T) \Ge 0$ and
$\bscr_2(T) \neq 0$,
   \item[(iii)] $T^{*n}T^n=I -\bscr_2(T) + n
(\bscr_2(T) - \bscr_1(T))$ for all $n \in \nbb$,
$\bscr_2(T) \Ge 0$ and $\bscr_2(T) \neq 0$,
   \item[(iv)] $T$ satisfies Theorem~{\em \ref{boundiff}(ii)} with
$\supp{F}=\{0\}$.
   \end{enumerate}
Moreover, if {\em (i)} holds, then
   \begin{enumerate}
   \item[(a)] $r(T) = 1$ whenever $T\neq 0$,
   \item[(b)] $T$ is subnormal if and only if
$\bscr_1(T)T=0$ and $\|T\| \Le 1${\em ;} if this is
the case, then $\|T\| =1$ provided $T\neq 0$.
   \end{enumerate}
   \end{pro}
   \begin{proof}
(i)$\Rightarrow$(ii) Substituting $q=X$
into \eqref{Geoje} yields
$T^*\bscr_2(T)T=0$. By
Corollary~\ref{nofs-sup2}, $\bscr_2(T) =
M(\rbb_+) \Ge 0$. Putting this all
together implies (ii).

(ii)$\Rightarrow$(i) Note that the set function
$M\colon \borel{\rbb_+} \to \ogr{\hh}$ defined
by $M(\varDelta)=\chi_{\varDelta}(0) \bscr_2(T)$
for $\varDelta \in \borel{\rbb_+}$ is a
semispectral measure such that $\supp{M} =
\{0\}$. Clearly \eqref{cziki-2} holds, so by
Theorem~\ref{dyl-an}, $T$ is CPD and
\eqref{checpt-2} is valid.

(i)$\Rightarrow$(iii) Let $(B,C,F)$ be the
representing triplet of $T$. According to
Theorem~\ref{dyltyprep}(b), $F=M$, $C=0$ and
$B=-\bscr_1(T)$, so by \eqref{cdr5} and
Corollary~\ref{nofs-sup2},
   \begin{align*}
T^{*n}T^n & = I - n \bscr_1(T) + Q_n(0) \bscr_2(T)
   \\
& = I - n \bscr_1(T) + (n-1) \bscr_2(T), \quad n \in
\nbb.
   \end{align*}
This together with the implication
(i)$\Rightarrow$(ii) gives (iii).

(iii)$\Rightarrow$(iv) As above, the set function
$F\colon \borel{\rbb_+} \to \ogr{\hh}$ defined by
$F(\varDelta) = \chi_{\varDelta}(0)\bscr_2(T)$ for
$\varDelta \in \borel{\rbb_+}$ is a semispectral
measure for which $\supp{F}=\{0\}$. Set $D=\bscr_2(T)
- \bscr_1(T)$. It is easily seen that $D$ and $F$
satisfy Theorem~\ref{boundiff}(ii).

(iv)$\Rightarrow$(i) Apply Theorems~\ref{boundiff} and
\ref{dyltyprep}(b).

We now prove the ``moreover'' part.

(a) If $D\neq 0$, then by
Theorems~\ref{boundiff} and \ref{rt=1},
$r(T)=1$. Suppose that $D=0$. Then by (iii),
$T^{*n}T^n=I -\bscr_2(T)$ for all $n \in \nbb$.
This together with $T\neq 0$ implies that $I
-\bscr_2(T) \neq 0$, so by Gelfand's formula for
spectral radius $r(T)=1$.

(b) Suppose first that $T$ is subnormal. It follows
from (iii) that for every $h\in \hh$,
   \begin{align} \label{tenh-bi1}
\|T^n h\|^2 = \is{(I -\bscr_2(T))h}{h} + n
\is{(\bscr_2(T) - \bscr_1(T))h}{h}, \quad n\in \nbb.
   \end{align}
By Theorem~\ref{lamb}, $\{\|T^{n+1}
h\|^2\}_{n=0}^{\infty}$ is a Stieltjes
moment sequence for every $h\in \hh$.
Combined with \eqref{tenh-bi1} and
\cite[Lemma~4.7]{C-J-J-S19}, this implies
that
   \begin{align*}
\is{(\bscr_2(T) - \bscr_1(T))h}{h}=0, \quad h \in \hh,
   \end{align*}
or equivalently that $T^*\bscr_1(T)T=0$. By (a) and
\eqref{subn-norm}, $T$ is a contraction (in fact,
$\|T\| = 1$ if $T\neq 0$), so $\bscr_1(T) \Ge 0$ and
consequently $\bscr_1(T)T=0$.

In turn, if $\|T\| \Le 1$ and $\bscr_1(T)T=0$, then
$T$ is a contraction whose restriction to
$\overline{\ob{T}}$ is an isometry, so an application
of Lemma~\ref{restr-subn} with $\mu_h:=\|Th\|^2
\delta_1$ shows that $T$ is subnormal. This completes
the proof.
   \end{proof}
Now we give an example of an operator satisfying
condition (i) of Proposition~\ref{kop-2izo}. In
particular, we show that the class of operators
satisfying this condition can contain both
(non-isometric) subnormal and non-subnormal
operators.
   \begin{exa} \label{prz-do-na}
Fix real numbers $a\in (0,\infty)$ and $b\in
[1,\infty)$ such that
   \begin{align}  \label{a-s-s}
\theta:=1 - 2 a + a b >0.
   \end{align}
Define the sequence $\{\lambda_n\}_{n=0}^{\infty}
\subseteq (0,\infty)$ by
   \begin{align*}
   \lambda_n =
   \begin{cases}
\sqrt{a} & \text{ if } n=0,
   \\[.5ex]
\sqrt{\frac{1 + n(b-1)}{1 + (n-1)(b-1)}} & \text{ if }
n \Ge 1.
   \end{cases}
   \end{align*}
Let $W_{a,b}$ be the unilateral weighted shift
on $\ell^2$ with weights
$\{\lambda_n\}_{n=0}^{\infty}$. It follows from
\cite[Lemma~6.1 \& Proposition~6.2(iii)]{Ja-St}
that $W_{a,b}\in \ogr{\ell^2}$ and
   \begin{align} \label{w-nor-b}
\|W_{a,b}\|^2 = \max\big\{a, b\big\}.
   \end{align}
One can also verify that $\bscr_2(W_{a,b})$ is
the diagonal operator (with respect to the the
standard orthonormal basis of $\ell^2$) with the
diagonal $(\theta, 0, 0,\ldots)$. This together
with \eqref{a-s-s} implies that
$\bscr_2(W_{a,b})W_{a,b}=0$, $\bscr_2(W_{a,b})
\Ge 0$ and $\|\bscr_2(W_{a,b})\|=\theta > 0$. In
view of Proposition~\ref{kop-2izo}, the operator
$W_{a,b}$ satisfies condition (i) of this
proposition. From Propositions~\ref{sub-mzero-n}
and \ref{kop-2izo} it follows that $W_{a,b}$ is
a CPD operator which is not $m$-isometric for
any $m\in \nbb$. If $a > 1$, we see that
$\|W_{a}\|= \sqrt{a} > 1$ and
$\|\bscr_2(W_{a})\| = (a-1)^2$, where
$W_{a}:=W_{a,a}$. Since, by Proposition~
\ref{kop-2izo}(a), $r(W_{a})=1$ for every $a\in
(1,\infty)$, we deduce that
   \begin{align*}
\text{$W_{a}$ is not normaloid for all $a
> 1$ and $\lim_{a \to \infty} \|W_{a}\| =
\lim_{a \to \infty} \|\bscr_2(W_{a})\| =
\infty$.}
   \end{align*}
In turn, if $a\in (0,1)$ and $b=1$, then one can
verify that $\bscr_1(W_{a,1})W_{a,1}=0$ and by
\eqref{w-nor-b}, $\|W_{a,1}\| = 1$, so by
Proposition~\ref{kop-2izo}, the operator
$W_{a,1}$ is subnormal and $r(W_{a,1})=1$.
   \hfill $\diamondsuit$
   \end{exa}
We conclude this subsection with a remark
related to Proposition~\ref{kop-2izo} and
Example~\ref{prz-do-na}.
   \begin{rem} \label{rem-t0-ex}
Suppose that $T\in \ogr{\hh}$ is nonzero and
satisfies condition (i) of
Proposition~\ref{kop-2izo} (the zero operator on
nonzero $\hh$ does satisfy (i)). By Proposition~
\ref{sub-mzero-n}, $T$ is not an $m$-isometry
for any $m\in \nbb$. According to condition
(iii) of Proposition~\ref{kop-2izo}, $T^{*n}T^n$
is a polynomial in $n$ if $n$ varies over $\nbb$
however not when $n$ varies over $\zbb_+$.
Indeed, otherwise, since a nonzero polynomial
may have only finite number of roots, we deduce
from (iii) that $I=I -\bscr_2(T)$, which
contradicts $\bscr_2(T)\neq 0$. In other words,
in view of \cite[p.\ 389]{Ag-St1} (see also
\cite[Corollary~3.5]{J-J-S20}), the requirement
that $T^{*n}T^n$ be a polynomial in $n$ if $n$
varies over $\nbb$ is not enough for $T$ to be
an $m$-isometry no matter what is $m$. Finally
note that $T$ falls under Case 3 of the proof of
Theorem~\ref{dyltyprep}(c) and the discussion
performed in Remark~ \ref{waz-rem}a). Indeed, by
Theorem~\ref{dyltyprep}(b),
Proposition~\ref{kop-2izo}(a) and
Corollary~\ref{nofs-sup2}, we see that
$B=-\bscr_1(T)$, $C=0$, $F=M$, $\vartheta: =
\sup\supp{F}=0$, $r(T)=1$ and
   \begin{align*}
D:=B+\int_{\rbb_+} \frac{1}{1-x} F(\D x) = \bscr_2(T)
- \bscr_1(T).
   \end{align*}
Moreover, in view of Example~\ref{prz-do-na}, both
cases $D=0$ and $D\neq 0$ can appear. \hfill
$\diamondsuit$
   \end{rem}
   \textbf{Acknowledgement}. The authors would
like to express their deepest thanks to the
anonymous reviewer for reading the article
carefully and catching any ambiguities, as well
as for suggestions and questions that made the
article more readable and reader friendly. A
part of this paper was written while the first
and the third author visited Kyungpook National
University during the autumn of 2019. They wish
to thank the faculty and the administration of
this unit for their warm hospitality.
   \bibliographystyle{amsalpha}

\begin{thebibliography}{99}
   \bibitem{Ag85} J. Agler,  Hypercontractions and subnormality,
{\em J. Operator Theory} {\bf 13} (1985), 203-217.
   \bibitem{Ag90} J. Agler, A disconjugacy theorem for
Toeplitz operators, {\em Amer. J. Math.} {\bf 112}
(1990), 1-14.
   \bibitem{Ag-St1} J. Agler, M. Stankus, $m$-isometric
transformations of Hilbert spaces, I, {\it Integr.
Equ. Oper. Theory} {\bf 21} (1995), 383-429.
   \bibitem{Ag-St2} J. Agler, M. Stankus, $m$-isometric
transformations of Hilbert spaces, II, {\it Integr.
Equ. Oper. Theory} {\bf 23} (1995), 1-48.
   \bibitem{Ag-St3} J. Agler, M. Stankus, $m$-isometric
transformations of Hilbert spaces, III, {\it Integr.
Equ. Oper. Theory} {\bf 24} (1996), 379-421.
   \bibitem{Ash} R. B. Ash, {\em Probability and measure theory},
Harcourt/Academic Press, Burlington, 2000.
   \bibitem{At91}  A. Athavale, Some operator theoretic calculus
for positive definite kernels, {\em Proc. Amer. Math. Soc.}
{\bf 112} (1991), 701-708.
   \bibitem{At} A. Athavale, On completely hyperexpansive operators,
{\em Proc. Amer. Math. Soc.} {\bf 124} (1996),
3745-3752.
   \bibitem{At2} A. Athavale, The complete hyperexpansivity analog
of the Embry conditions, {\em Studia Math.} {\bf 154}
(2003), 233-242.
   \bibitem{B-C-R} C. Berg, J. P. R. Christensen,
P. Ressel, {\em Harmonic Analysis on Semigroups},
Springer-Verlag, Berlin 1984.
   \bibitem{Ber-Mar-No13}
T. Berm\'{u}dez, A. Martin\'{o}n, J. A. Noda, An
isometry plus a nilpotent operator is an $m$-isometry.
Applications, {\em J. Math. Anal. Appl.} {\bf 407}
(2013), 505-512.
  \bibitem{B-S87} M. Sh. Birman, M. Z. Solomjak,
{\it Spectral theory of selfadjoint operators in Hilbert
space}, D. Reidel Publishing Co., Dordrecht, 1987.
  \bibitem{Bi94} T. M. Bisgaard, Positive definite operator
sequences, {\em Proc. Amer. Math. Soc.} {\bf 121}
(1994), 1185-1191.
   \bibitem{Bi-Sa00} T. M. Bisgaard, Z.  Sasv\'{a}ri,
{\em Characteristic functions and moment sequences.
Positive definiteness in probability,} Nova Science
Publishers, Inc., Huntington, NY, 2000.
  \bibitem{Bis57} E. Bishop, Spectral theory for operators
on a Banach space, {\em Trans. Amer. Math. Soc.}
{\bf 86} (1957), 414-445.
   \bibitem{Bo-Ja} F. Botelho, J. Jamison, Isometric properties of
elementary operators, {\em Linear Algebra Appl.} {\bf 432}
(2010), 357-365.
   \bibitem{B-J-J-S15} P. Budzy\'{n}ski, Z. J. Jab{\l}o\'nski, I. B.
Jung, J. Stochel, Unbounded subnormal composition
operators in $L^2$-spaces, {\em J. Funct. Anal.} {\bf
269} (2015), 2110-2164.
   \bibitem{B-J-J-S17} P. Budzy\'{n}ski, Z. J. Jab{\l}o\'nski, I. B. Jung,
J. Stochel, Subnormality of unbounded composition operators
over one-circuit directed graphs: exotic examples, {\em
Adv. Math.} {\bf 310} (2017), 484-556.
   \bibitem{B-J-J-S18} P. Budzy\'{n}ski, Z. J. Jab{\l}o\'{n}ski,
I. B. Jung, J. Stochel, Unbounded
weighted composition operators in
$L^2$-spaces, {\em Lect. Notes Math.},
Volume 2209, Springer 2018.
   \bibitem{C-J-J-S19}
S. Chavan, Z. J. Jab{\l}o\'{n}ki, I. B. Jung, J.
Stochel, Taylor spectrum approach to
Brownian-type operators with quasinormal entry,
{\em Ann. Mat. Pur. Appl.} {\bf 200} (2021),
881-922.
   \bibitem{Cha-Sh} S. Chavan, V. M. Sholapurkar, Completely
monotone functions of finite order and Agler's
conditions, {\em Studia Math.} {\bf 226} (2015),
229-258.
   \bibitem{Cha-Sh17} S. Chavan, V. M. Sholapurkar,
Completely hyperexpansive tuples of finite order, {\em
J. Math. Anal. Appl.} {\bf 447} (2017), 1009-1026.
   \bibitem{Ci-St07} D. Cicho\'{n}, Jan Stochel,
Subnormality, analyticity and perturbations, {\em
Rocky Mountain J. Math.} {\bf 37} (2007), 1831-1869.
   \bibitem{con2} J.  B.  Conway, {\it A course in functional
analysis}, Graduate Texts in Mathematics {\bf 96},
Springer-Verlag, New York, 1990.
   \bibitem{Con91} J. B. Conway, {\em The theory of
subnormal operators}, Mathematical Surveys and
Monographs, {\bf 36}, American Mathematical Society,
Providence, RI, 1991.
   \bibitem{Cu90} R. E. Curto, Quadratically hyponormal weighted
shifts, {\em Integr. Equ. Oper. Theory} {\bf 13}
(1990), 49-66.
   \bibitem{Cu-Pu93}
R. Curto and M. Putinar, Nearly subnormal
operators and moment problems, {\em J.
Funct. Anal.} {\bf 115} (1993), 480-497.
   \bibitem{Dan} J. Dane\v{s}, On local spectral radius, {\em {\v
C}asopis P{\v e}st. Mat.} {\bf 112} (1987), 177-187.
   \bibitem{Emb73} M. R. Embry, A generalization
of the Halmos-Bram criterion for
subnormality, {\em Acta Sci. Math.
$($Szeged$)$} {\bf 35} (1973), 61-64.
  \bibitem{E-J-L06} G. Exner, I. B. Jung, C. Li, On
$k$-hyperexpansive operators, {\em J. Math. Anal.
Appl.} {\bf 323} (2006), 569-582.
  \bibitem{Fur} T. Furuta, Invitation to
linear operators, Taylor \& Francis, Ltd., London,
2001.
   \bibitem{Gu15} C. Gu, On $(m,p)$-expansive and $(m,p)$-contractive
operators on Hilbert and Banach spaces, {\em J. Math.
Anal. Appl.} {\bf 426} (2015), 893-916.
    \bibitem{Hal50} P. Halmos, Normal dilations and extensions of
operators, {\em Summa Bras. Math.} {\bf 2}
(1950), 124-134.
    \bibitem{Hal} P. R. Halmos, {\em A Hilbert space problem
book}, Springer-Verlag, New York Inc. 1982.
   \bibitem{Hor-Joh} R. A. Horn, C. R. Johnson,
{\em Matrix analysis}, Cambridge University Press,
Cambridge, 1985.
   \bibitem{hun74} T. W. Hungerford, {\em Algebra},
Graduate Texts in Mathematics 73,
Springer-Verlag, New York, 1974.
   \bibitem{Ja02} Z. Jab{\l}o\'{n}ski, Complete hyperexpansivity,
subnormality and inverted boundedness conditions, {\it
Integr. Equ. Oper. Theory} {\bf 44} (2002), 316-336.
   \bibitem{J-J-L-S21} Z. J. Jab{\l}o\'nski, I. B.
Jung, E. Y. Lee, J. Stochel, Conditionally
positive definite weighted shifts, preprint, 28
pp., arXiv:2106.03222.
   \bibitem{J-J-S12} Z. J. Jab{\l}o\'{n}ski, I. B. Jung,
J. Stochel, Weighted shifts on directed trees, {\em
Memoirs of the AMS} {\bf 216} (2012).
   \bibitem{J-J-S12-jfa} Z. J. Jab{\l}o\'nski,  I. B. Jung,
J. Stochel, A non-hyponormal operator generating
Stieltjes moment sequences, {\em J. Funct. Anal.} {\bf
262} (2012), 3946-3980.
   \bibitem{J-J-S20} Z. J. Jab{\l}o\'nski, I. B.
Jung, J. Stochel, $m$-Isometric operators and their
local properties, {\em Linear Algebra Appl.} {\bf 596}
(2020), 49-70.
   \bibitem{J-J-S21p} Z. J. Jab{\l}o\'nski, I. B.
Jung, J. Stochel, Conditionally positive
definite algebraic operators, in preparation.
   \bibitem{Ja-St} Z. Jab{\l}o\'{n}ski, J. Stochel,
Unbounded $2$-hyperexpansive operators, {\em Proc. Edin.
Math. Soc.} {\bf 44} (2001), 613-629.
   \bibitem{Ju-St} I. Jung, J. Stochel, Subnormal
operators whose adjoints have rich point spectrum,
{\em J. Funct. Anal.} {\bf 255} (2008), 1797-1816.
   \bibitem{Kol41} A. N. Kolmogoroff,
Stationary sequences in Hilbert's space,
(Russian) {\em Bolletin Moskovskogo
Gosudarstvenogo Universiteta. Matematika} {\bf
2} (1941), 40 pp.
   \bibitem{Kub} C. S. Kubrusly, Strong stability does not
imply similarity to a contraction, {\em Systems Control
Lett.} {\bf 14} (1990), 397-400.
   \bibitem{Kub-b} C. S. Kubrusly,  {\em An introduction to models and
decompositions in operator theory}, Birkh\"{a}user Boston,
Inc., Boston, MA, 1997.
   \bibitem{lam} A. Lambert, Subnormality and weighted
shifts, {\em J. London Math. Soc.} {\bf 14} (1976),
476-480.
   \bibitem{Lam88} A. Lambert, Subnormal composition
operators, {\em Proc. Am. Math. Soc.} {\bf 103},
750-754 (1988).
   \bibitem{Ma72} P. Masani, On helixes in Hilbert space.
I, {\em Teor. Verojatnost. i Primenen.} {\bf 17}
(1972), 3-20.
    \bibitem{MP89} S. McCullough and V. I. Paulsen, A note
on joint hyponormality, {\em Proc. Amer. Math. Soc.} {\bf
107} (1989), 187-195.
   \bibitem{Ml78} W. Mlak, {\em Dilations of Hilbert space operators
$($general theory$)$}, Dissertationes Math. {\bf
153} (1978), 61 pp.
   \bibitem{Ml83} W. Mlak, Conditionally positive definite functions on
linear spaces, {\em Ann. Polon. Math.} {\bf 42}
(1983), 187-239.
   \bibitem{Ml88} W. Mlak, The  Schr\"{o}dinger  type  couples
related to weighted shifts, {\it Univ. Iagel.
Acta Math.} {\bf 27} (1988), 297-301.
    \bibitem{Pa-Sch} K. R. Parthasarathy, K. Schmidt,
{\em Positive definite kernels, continuous tensor
products, and central limit theorems of probability
theory}, Lecture Notes in Mathematics, Vol. 272,
Springer-Verlag, Berlin-New York, 1972.
   \bibitem{Rich} S. Richter, Invariant subspaces of
the Dirichlet shift, {\em Jour. Reine Angew. Math.}
{\bf 386} (1988), 205-220.
   \bibitem{Rud73} W. Rudin,
{\em Functional analysis,} McGraw-Hill Series in Higher
Mathematics, McGraw-Hill Book Co., New York, 1973.
   \bibitem{Rud76} W. Rudin, {\em Principles of mathematical
analysis}, International Series in Pure and Applied
Mathematics, McGraw-Hill Book Co., New York, 1976.
   \bibitem{Rud87} W. Rudin, {\em Real and Complex
Analysis}, McGraw-Hill Book Co., New York, 1987.
   \bibitem{Sa13} Z. Sasv\'{a}ri, {\em Multivariate
characteristic and correlation functions,} De Gruyter
Studies in Mathematics, 50, Walter de Gruyter \& Co.,
Berlin, 2013.
   \bibitem{Sch12} K. Schm\"{u}dgen, {\em Unbounded self-adjoint
operators on Hilbert space,} Graduate Texts in Mathematics,
265, Springer, Dordrecht, 2012.
   \bibitem{sch} I. Schur, Bemerkungen zur Theorie der
beschr\"{a}nkten Bilinearformenz mit unendlich vielen
Ver\"{a}nderlichen, {\em J. Reine Angew. Math.} {\bf 140}
(1911), 1-29.
   \bibitem{Shi74} A. L. Shields, Weighted shift operators
and analytic function theory, {\em Topics in
operator theory}, pp. 49-128. Math. Surveys, No.
13, Amer. Math. Soc., Providence, R.I., 1974.
   \bibitem{Sh-At} V. M. Sholapurkar,  A. Athavale,
Completely and alternatingly hyperexpansive operators,
{\em J. Operator Theory} {\bf 43} (2000), 43-68.
   \bibitem{sim} B. Simon, The classical moment problem
as a self-adjoint finite difference operator, {\em
Adv. Math.} {\bf 137} (1998), 82-203.
   \bibitem{Sti} T. Stieltjes,
Recherches sur les fractions continues, {\em Anns.
Fac. Sci. Univ. Toulouse} {\bf 8} (1894-1895),
J1-J122; {\bf 9}, A5-A47.
   \bibitem{Stin55} W. F. Stinespring, Positive functions on
$C^*$-algebras, {\em Proc. Amer. Math. Soc.} {\bf 6}
(1955), 211-216.
   \bibitem{Sto0} J. Stochel, The Fubini theorem for
semi-spectral integrals and semi-spectral
representations of some families of operators, {\it
Univ. Iagel. Acta Math.} {\bf 26} (1987), 17-27.
   \bibitem{Sto90} J. Stochel, Seminormal composition operators on
$L^2$ spaces induced by matrices, {\it Hokkaido Math.
J.} {\bf 19} (1990), 307-324.
   \bibitem{Sto} J. Stochel, Characterizations of subnormal
operators, {\em Studia Math.} {\bf 97} (1991),
227-238.
   \bibitem{Sto3} J. Stochel, Decomposition and
disintegration of positive definite kernels on convex
$*$-semigroups, {\em Ann. Polon. Math.} {\bf 56} (1992),
243-294.
   \bibitem{St-St17}  J. Stochel, J. B. Stochel,
Composition operators on Hilbert spaces of entire
functions with analytic symbols, {\em J. Math. Anal.
Appl.} {\bf 454} (2017), 1019-1066.
   \bibitem{St-Sz85} J. Stochel, F. H.
Szafraniec, On normal extensions of
unbounded operators. I, {\it J.
Operator Theory} {\bf 14} (1985),
31-55.
   \bibitem{St-Sz89} J. Stochel, F. H. Szafraniec, On
normal extensions of unbounded operators. II, {\it Acta.
Sci. Math. $($Szeged\/$)$} {\bf 53} (1989), 153-177.
   \bibitem{St-Sz89-III} J. Stochel, F. H.
Szafraniec, On normal extensions of
unbounded operators. III. Spectral
properties, {\it Publ. RIMS, Kyoto
Univ.} {\bf 25} (1989), 105-139.
   \bibitem{Sz-N53} B. Sz.-Nagy,  A moment problem for self-adjoint
operators, {\em Acta Math. Acad. Sci. Hungar.} (1953),
285-293.
   \bibitem{Tay-La80} A. E. Taylor, D. C. Lay,
{\em Introduction to functional analysis,} Second edition,
John Wiley \& Sons, New York-Chichester-Brisbane, 1980.
   \end{thebibliography}
   
   \end{document}